\documentclass[11pt,reqno]{amsart}
\usepackage[margin=2.9cm]{geometry}
\usepackage[english]{babel}
\usepackage[utf8]{inputenc}

\usepackage{graphicx}	
\usepackage[dvipsnames]{xcolor}			
\usepackage{amsfonts, amsmath, amsthm, amssymb}
\usepackage{mathtools, mathrsfs}
\usepackage[shortlabels]{enumitem}
\usepackage{bigints}
\usepackage[final]{pdfpages}
\usepackage{framed}
\usepackage[11pt]{moresize}
\usepackage[font=footnotesize,labelfont=bf]{caption}
\usepackage{subcaption}
\usepackage{tikz}
\usepackage{cite}
\usepackage{wrapfig}
\usepackage{bigints}
\usepackage{afterpage}
\usepackage{todonotes}
\usepackage{setspace}
\usepackage{appendix}
\usepackage{etoolbox}
\newcommand{\e}{\varepsilon}
\newcommand{\f}{\frac}
\newcommand{\im}{\implies}

\newcommand{\la}{\lambda}
\newcommand{\al}{\alpha}
\newcommand{\be}{\beta}

\newcommand{\p}{\partial}
\newcommand{\lra}{\longrightarrow}

\newcommand{\w}{\omega}

\newcommand{\Om}{\Omega}

\newcommand{\C}{\mathbb{C}}
\newcommand{\E}{\mathbb{E}}

\newcommand{\R}{\mathbb{R}}
\newcommand{\Ss}{\mathbb{S}}
\newcommand{\U}{\mathbb{U}}
\newcommand{\Z}{\mathbb{Z}}
\newcommand{\N}{\mathbb{N}}

\newcommand{\cA}{\mathcal{A}}

\newcommand{\cL}{\mathcal{L}}
\newcommand{\cM}{\mathcal{M}}

\newcommand{\cT}{\mathcal{T}}
\newcommand{\cV}{\mathcal{V}}

\newcommand{\ord}{\operatorname{ord}}
\newcommand{\codim}{\operatorname{codim}}

\newcommand{\Gr}{\operatorname{Gr}}
\newcommand{\sgn}{\operatorname{sign}}
\newcommand{\sech}{\operatorname{sech}}
\newcommand{\dist}{\operatorname{dist}}

\newcommand{\spec}{\operatorname{Spec}}
\newcommand{\esspec}{\operatorname{Spec_{\rm ess}}}

\newcommand{\dom}{\operatorname{dom}}
\newcommand{\spn}{\operatorname{span}}

\DeclareMathOperator{\diag}{diag}

\DeclareMathOperator{\Mas}{Mas}

\newcommand{\graph}{\operatorname{graph}}

\newcommand{\pde}[2]{\frac{\partial #1}{\partial #2 }} 
 
\newcommand{\de}[2]{\frac{d #1}{d #2 }} 
\newcommand{\des}[2]{\frac{d^2 #1}{d #2^2 }} 
\newcommand{\dek}[2]{\frac{d^k #1}{d #2^k }} 
\newcommand{\dei}[2]{\frac{d^i #1}{d #2^i }}

\newtheorem{lemma}{Lemma}[section]

\usepackage[colorlinks=true,backref=page,linkcolor={blue}]{hyperref}
\usepackage[capitalise,nameinlink]{cleveref}

\newtheorem{prop}[lemma]{Proposition}
\newtheorem{theorem}[lemma]{Theorem}
\newtheorem{cor}[lemma]{Corollary}
\newtheorem{define}[lemma]{Definition}

\theoremstyle{definition}
\newtheorem{rem}[lemma]{Remark}
\newtheorem{hypo}[lemma]{Hypothesis}

\crefname{rem}{Remark}{Remarks}
\crefname{conj}{Conjecture}{Conjectures}
\crefname{hypo}{Hypothesis}{Hypotheses}
\crefname{enumhypoi}{Hypothesis}{Hypotheses}
\crefname{theorem}{Theorem}{Theorems}
\crefname{cor}{Corollary}{Corollaries}
\crefname{prop}{Proposition}{Propositions}
\crefname{define}{Definition}{Definitions}

\AtBeginEnvironment{appendices}{\crefalias{section}{appendix}}

\definecolor{skyblue}{rgb}{0.85,0.85,1}
\definecolor{deepcerise}{rgb}{0.85, 0.2, 0.53}
\definecolor{psychedelicpurple}{rgb}{0.87, 0.0, 1.0}
\definecolor{electricpurple}{rgb}{0.75, 0.0, 1.0}
\definecolor{purplemunsell}{rgb}{0.62, 0.0, 0.77}
\definecolor{internationalkleinblue}{rgb}{0.0, 0.18, 0.65}
\definecolor{indigoweb}{rgb}{0.29, 0.0, 0.51}

\begin{document}
	
	\title[Fourth-order NLS and the Maslov index]{Detecting eigenvalues in a fourth-order nonlinear Schr\"odinger equation with a non-regular Maslov box}
	
	\author[Mitchell Curran]{Mitchell Curran $^{1,2,\dagger}$}
	\thanks{$^1$Department of Mathematics and Statistics, Auburn University, Auburn, AL 36830 USA}
	\author[Robert Marangell]{Robert Marangell$^{2}$}
	\thanks{$^2$School of Mathematics and Statistics, University of Sydney, Sydney, NSW 2006 Australia}
	\thanks{$^\dagger$Corresponding author. Email: \texttt{mtc0076@auburn.edu}}

	\begin{abstract}
		We use the Maslov index to study the eigenvalue problem arising from the linearisation about solitons in the fourth-order cubic nonlinear Schr\"odinger equation (NLSE). Our analysis is motivated by recent work by Bandara et al., in which the fourth-order cubic NLSE was shown to support infinite families of multipulse solitons. Using a homotopy argument, we prove that the Morse indices of two selfadjoint fourth-order operators appearing in the linearisation may be computed by counting conjugate points, as well as a lower bound for the number of real unstable eigenvalues of the linearisation. We also give a Vakhitov-Kolokolov type stability criterion. The interesting aspects of this problem as an application of the Maslov index are the instances of non-regular crossings, which feature crossing forms with varying ranks of degeneracy. We handle such degeneracies directly via higher order crossing forms, using a definition of the Maslov index developed by Piccione and Tausk.
	\end{abstract}

	\maketitle
	\parskip=0em
	\tableofcontents
	
	\numberwithin{equation}{section}
	\allowdisplaybreaks

	\section{Introduction}
	\parskip=1em
	
	The fourth-order cubic nonlinear Schr\"odinger (NLS) equation
	\begin{equation}\label{4NLS_original}
		i \Psi_t =  - \f{\be_4}{24} \Psi_{xxxx} + \f{\be_2}{2} \Psi_{xx} - \gamma |\Psi|^2\Psi,
	\end{equation}
	models the propagation of pulses in media with Kerr nonlinearity that are subject to both quartic and quadratic dispersion \cite{KarlssonHook94,Akhmediev1994,BGBK21,TABM19}. Here $\Psi$ is the slowly varying complex envelope of the pulse, and $\be_2, \be_4,$ and $\gamma$ are real coefficients.

	Solutions to \eqref{4NLS_original} of the form $\Psi(x,t) = e^{i\be t} \phi(x)$, $\be\in\R$, are called  \emph{standing wave} solutions. Following the convention of \cite{BGBK21}, when the real-valued wave profile $\phi$ is a homoclinic orbit of the associated standing wave equation (see \eqref{SWE}), we will call  $\Psi$ a \emph{soliton solution} of \eqref{4NLS_original}. Karlsson and H\"o\"ok \cite{KarlssonHook94} discovered an exact analytic family of soliton solutions to \eqref{4NLS_original} with a squared hyperbolic secant profile. Akhmediev, Buryak and Karlsson \cite{Akhmediev1994} observed oscillatory behaviour in the tails of solitons for certain values of $\be$. Akhmediev and Buryak \cite{Buryak1995} showed the existence of and derived a stability criterion for bound states of two or more solitons when the single solitons have oscillating tails. Karpman and Shagalov \cite{Karp96,KarpShag97,KarpShag2000} considered the extension of \eqref{4NLS_original} to higher-order nonlinearities and multiple space dimensions. All of these works considered the case of negative quartic and negative quadratic dispersion, $\be_4<0, \be_2<0$.

	More recently, \eqref{4NLS_original} has been the focus of a number of studies following the experimental discovery of \emph{pure quartic solitons} (PQSs) in silicon photonic crystal waveguides \cite{BR_PQS_2016}. These solitons exist through a balance of negative quartic dispersion and the nonlinear Kerr effect, for which $\be_2=0$ and $\be_4<0$. They have attracted much attention for their potential applicability to ultrafast lasers due to their favourable energy-width scaling \cite{ultrafast_lasers17,TABM19}. Following the discovery of PQSs, Tam \emph{et al.} \cite{TABM19} numerically investigated their existence and spectral stability. They also showed \cite{TamALexander_etal_18,TamAlexander_etal_20} that PQSs and solitons of the classical second-order NLS equation, for which $\be_4=0$, are in fact part of a broader continuous family of soliton solutions to \eqref{4NLS_original} for nonpositive dispersion coefficients $\be_4$ and $\be_2$.

	Extending the work of Tam \emph{et al.}, in a series of works Bandara and co-authors \cite{BGBK21,BGBK23,BGBK24} performed a detailed study of the structure of soliton solutions of \eqref{4NLS_original}. In \cite{BGBK21}, they used a dynamical systems approach to show that \eqref{4NLS_original} supports infinite families of multi-pulses (multi-hump solitons) for \emph{both} positive and negative quadratic dispersion (and negative quartic dispersion, $\be_4<0$). By seeking standing wave solutions, they transform \eqref{4NLS_original} into \eqref{SWE}, a fourth-order Hamiltonian ODE with two reversibility symmetries. Solitons of \eqref{4NLS_original} correspond to orbits of \eqref{SWE} that are homoclinic to the origin, which the authors find and track with numerical continuation techniques. In addition to the primary (single-hump) homoclinic orbit, which may have oscillating or non-oscillating decaying tails, they show the existence of connecting orbits from the origin to periodic orbits in the zero energy level (``EtoP connections"). These connecting orbits allow the construction of heteroclinic cycles: homoclinic solutions which follow an EtoP connection from the origin to a periodic orbit, then loop $n$ times around the periodic orbit, and finally follow an EtoP connection back to the origin. The different combinations of EtoP connections and the symmetry properties of the periodic orbits generate cycles that organise infinite families of multi-hump solitons of different symmetry types (symmetric, antisymmetric and nonsymmetric). Moreover, the authors show that there are infinitely many periodic orbits that support EtoP connections, and, hence, a menagerie of families of homoclinic solutions with different symmetry properties. In a subsequent work, Bandara \emph{et al.} discuss connecting orbits between periodic solutions (``PtoP connections"), which not only yield solutions to periodic backgrounds, but can also be combined with EtoP connections to obtain even more families of solitons (with oscillating but decaying tails). Since EtoP and PtoP connections play a vital role in the existence and organisation of homoclinic solutions to \eqref{SWE}, and are entirely determined by the periodic orbits that exist, in \cite{BGBK24} Bandara \emph{et al.} study the underlying periodic orbit structure of equation \eqref{SWE}.

	Several authors have studied the stability of soliton solutions to \eqref{4NLS_original}. In \cite{BGBK21}, Bandara et al. used numerical simulations to show that, while many of the multi-pulse solutions they found were unstable, some were only weakly unstable, and therefore possibly observable in experiments over a number of dispersion lengths. More rigorous stability analyses were undertaken in \cite{NataliPastor15} and \cite{ParkerAceves21}. In \cite{NataliPastor15}, Natali and Pastor proved the orbital stability of the family of solutions found by Karlsson and H\"o\"ok (realised as a single exact solution to the nondimensionalised version of \eqref{4NLS_original}, see \eqref{4NLS}). As they observed (see also \cite[\S II]{TamAlexander_etal_20}), an explicit formula for a solution of the form found by Karlsson and H\"o\"ok exists only for a certain fixed value of the frequency parameter $\be$. Thus, while the classical results of Grillakis, Shatah and Strauss \cite{GSS87,GSS90} are technically valid, they are currently unfeasible from a practical point of view. In \cite{ParkerAceves21}, Parker and Aceves proved the existence of a primary single-hump soliton in certain parameter regimes of \eqref{4NLS_original}, along with an associated discrete family of multi-hump solitons. Under certain assumptions, they proved orbital stability of the primary pulse and spectral instability of the associated multi-pulses.

	Outside of \eqref{4NLS_original}, a number of works have investigated multi-pulses in Hamiltonian systems. The first study was undertaken by Pelinovksy and Chugunova in \cite{ChugPel07}, who proved the existence and spectral stability of two-pulse solutions in the fifth order KdV equation. An extensive study of periodic multi-pulses in the fifth order KdV equation was performed by Parker and Sandstede \cite{ParkerSand22}. In \cite{KPS20}, Kapitula, Parker and Sandstede used a reformulated Krein matrix to study the location and Krein signature of small eigenvalues associated with tail-tail interactions in $n$-pulses in first- and second-order in time Hamiltonian partial differential equations.

	In this paper, we further develop the \emph{spectral} stability theory for \emph{arbitrary} single and multi-hump soliton solutions to \eqref{4NLS_original}. {(We stress that our results are applicable to \emph{any} soliton solution; see \cref{rem:everywhere_applicable}.)} In particular, we seek to determine the existence of unstable eigenvalues of the associated linear operator
	\begin{equation}\label{N:intro}
		N = \begin{pmatrix}
			0 & -L_- \\ L_+ & 0
		\end{pmatrix},
	\end{equation}
	where $L_+$ and $L_-$ are selfadjoint fourth order operators acting in $L^2(\R)$. Our results are not confined to solitons satisfying Hypothesis 2, the first part of Hypothesis 3 and Hypothesis 4 in \cite{ParkerAceves21}. The main tool in our arsenal is a topological invariant known as the \emph{Maslov index}, a signed count of the nontrivial intersections, or \emph{crossings}, of a path of Lagrangian subspaces of $\R^{2n}$ with a fixed reference plane. Since the Maslov index is unable to detect complex eigenvalues, we will restrict our search to those unstable eigenvalues that are real. In this work, we make no comment on the orbital stability of the soliton solutions of interest.

	Our main results are as follows. {\Cref{thm:Lpm_eval_counts} states that the \emph{Morse indices} (the number of positive eigenvalues) of $L_+$ and $L_-$, denoted $P\coloneqq n_+(L_+)$ and $Q\coloneqq n_+(L_-)$, are equal to the number of {\em conjugate points} (see \cref{define:conj_pointsL}) of the operators $L_+$ and $L_-$, respectively.} In \cref{thm:main_lower_bound}, we provide a lower bound for the number of positive real eigenvalues of \eqref{N:intro}. The bound is given in terms of $P$ and $Q$, as well as a correction term which represents the contribution to the Maslov index from the crossing corresponding to the zero eigenvalue of $N$. An immediate consequence of \cref{thm:main_lower_bound} is \cref{cor:JonesGrillakis}, a \emph{Jones-Grillakis} type instability theorem \cite{J88,Grill88,KapProm} which gives a sufficient condition on $P$ and $Q$ for spectral instability of the underlying soliton. We also give a \emph{Vakhitov-Kolokolov} (VK) type stability criterion \cite{VK73,pelinovsky} in \cref{thm:VK_criterion}, where the spectral stability of solitons for which $P=1$ and $Q=0$ is determined by the sign of a certain integral.

	{ \Cref{thm:main_lower_bound} has been proven for abstract selfadjoint operators $L_+$ and $L_-$ via Sylvester's law of inertia \cite{Pel05,pelinovsky}, the Pontryagin invariant subspace theorem \cite{CP10}, and the Krein signature \cite{KKS04,KKS05,KP12art}. Here, we present a new proof, using the Maslov index and tailored to the setting of ordinary differential operators, which does not require projecting off the kernel of $N$ and associated results on eigenvalue counts for constrained selfadjoint operators. An important aspect of our results, in line with recent results on the Maslov index for fourth order operators \cite{BJP24,Howard21_Hormander, Howard21}, is the geometric interpretation of the Morse indices $P$ and $Q$ as the number of $L_+$ and $L_-$ conjugate points, respectively. All of the required data for the lower bound in \cref{thm:main_lower_bound} therefore occurs when the spectral parameter is zero. As highlighted in \cite{Beck_Jaquette22}, this is a convenient feature for numerical computations that is not afforded by a calculation using the Evans function \cite{AGJ90}. In light of this, an alternate form of \eqref{lower_bound}, which is more useful for numerical computations, is given in \cref{rem:alternate_lower_bound}. }

	The key feature of the eigenvalue problem for $N$ that makes it amenable to the Maslov index is the Hamiltonian structure of the eigenvalue equations when written as a spatially dynamic first order system in $\R^8$. This system therefore preserves Lagrangian planes. Moreover, in parameter regimes which guarantee the essential spectrum of $N$ is confined to the imaginary axis, the first order system has an exponential dichotomy on the positive and negative half lines. This gives rise to two-parameter families of Lagrangian planes (in $x$ and $\la$, the spatial and spectral parameters), known as the \emph{unstable} and  \emph{stable bundles}, comprising the solutions that decay to zero exponentially as $x\to -\infty$ and $x\to +\infty$, respectively. These bundles are the central objects of our analysis; their nontrivial intersection at common values $x\in\R$ and $\la\in\R$ encode the positive real eigenvalues of interest. By exploiting homotopy invariance of the Maslov index, we will detect these eigenvalues by instead counting conjugate points, i.e. the nontrivial intersections of the unstable bundle, when the spectral parameter is zero, with the stable subspace of the asymptotic system.

	The Maslov index has been used to study the spectrum of homoclinic orbits in a number of works \cite{J88,BJ95,CJ18,Corn19,CH14,HLS18,BCLJMS18,Howard21,Howard21_Hormander}, namely as a tool to detect real unstable eigenvalues. If monotonicity in the spectral parameter holds, as is often the case in selfadjoint problems \cite{HLS18,BCLJMS18,Howard21}, then it is possible to give an exact count of such eigenvalues in terms of the Maslov index of a Lagrangian path in the spatial parameter. Howard, Latushkin and Sukhtayev \cite{HLS18} proved the equality of the Morse and Maslov indices for matrix-valued Schr\"odinger operators with asymptotically constant symmetric potentials on the line, applying their results to determine the stability of nonlinear waves in various reaction-diffusion systems. Beck \emph{et al.} \cite{BCLJMS18} proved the instability of pulses in gradient reaction-diffusion systems, generalising the instability result for pulses in scalar reaction-diffusion equations (see \cite[\S 2.3.3]{KapProm}). Jones \cite{J88} gave a stability criterion for soliton solutions of the classic cubic NLS equation with a spatially dependent nonlinearity. Bose and Jones \cite{BJ95} proved the stability of a travelling wave in a coupled two-variable reaction diffusion system with diffusion in one variable, in which they used the Maslov index to locate eigenvalues in gradient systems. Chen and Hu \cite{CH14} proved stability and instability criteria for standing pulses in a doubly-diffusive FitzHugh-Nagumo equation. Chardard, Dias and Bridges \cite{CDB09,CDB09part1,CDB11part2} developed numerical tools to compute the Maslov index and study the stability of homoclinic orbits in Hamiltonian systems. Cornwell and Jones \cite{Corn19,CJ18} analysed travelling waves in skew-gradient systems. By showing monotonicity in the spectral parameter \cite{Corn19} and computing the Maslov index at all conjugate points in the travelling wave co-ordinate, they proved \cite{CJ18} the stability of a travelling pulse in a doubly diffusive FitzHugh-Nagumo system.

	The current problem is distinguished from previous works by the presence of \emph{non-regular} crossings. In their work \cite{RS93}, Robbin and Salamon require that all crossings be \emph{regular} i.e. that the crossing form is nondegenerate. This is then topologically extended to all continuous Lagrangian paths (i.e. to those with non-regular crossings) via homotopy invariance. As pointed out in \cite{GPP_full,GPP04}, one issue with extending the definition in this way is the destruction of potentially important information encoded in any non-regular crossings. For example, in the present paper, it turns out that the contribution of the crossing corresponding to the zero eigenvalue of $N$, which is non-regular with respect to the spectral parameter, is determined by the kernel and generalised kernel of $N$. It is not clear how a homotopy argument would capture the same information. Additionally, the technical details of perturbing the path typically requires breaking the structures that generate the path in the first place, making analytical calculations of such perturbations difficult.

	We therefore use the approach of \cite{PT09} to locally compute the Maslov index directly through the use of \emph{higher-order crossing forms}. These generalise the (first-order) crossing form defined in \cite{RS93}, and allow us to calculate the contribution to the Maslov index from non-regular crossings without perturbative arguments. In their work, Deng and Jones \cite{DJ11} derived a second order crossing form in the case where the crossing form was identically zero. This was recently extended in \cite{BJP24} to the cases when \emph{all} lower-order forms are identically zero. Our formulas, adapted from \cite{PT09,GPP04}, do not require this assumption. Moreover, we do not write the Lagrangian path locally as the graph of an infinitesimally symplectic matrix, instead relying solely on the construction of \emph{root functions} and \emph{degeneracy spaces} (see \cref{subsec:maslov_index2}).

	Specific instances of non-regular crossings in the present paper include the conjugate point at the top left corner of the \emph{Maslov box}, which corresponds to the zero eigenvalue of $N$, as well as crossings in the spatial parameter. For the former, we find that the crossing form in the spectral parameter is identically zero; this is a feature of eigenvalue problems for operators of the form of \eqref{N:intro} \cite{CCLM23}. For crossings in the spatial parameter, the crossing form is degenerate, but in general not identically zero, and so the approaches of \cite{DJ11,CCLM23,BJP24} do not apply. This phenomenon appears to be the result of the eigenvalue equations being fourth order, and has been encountered in \cite{Howard21_Hormander,Howard21}. There, the authors circumvented the issue via a semi-definiteness argument; in the present setting, we elect to use higher order crossing forms in an effort to generalise previous efforts with degenerate crossing forms.

	In \cite{CCLM23}, a similar lower bound to that in \cref{thm:main_lower_bound} was derived for an eigenvalue problem of the form of \eqref{eq:N_evp} on a compact interval, where $L_\pm$ are Schr\"odinger operators. In that work, the correction term was computed via an analysis of the \emph{eigenvalue curves}, which represent the evolution of the eigenvalues as the spatial domain is shrunk or expanded. That the spatial domain is the whole real line in the present setting precludes the approach taken in \cite{CCLM23}, where the technical details required the spatial domain to be compact.

	\subsection{Statement of main results}\label{subsec:results}
	
	Following \cite{BGBK21,TamAlexander_etal_20,TABM19}, we will treat the case of negative quartic dispersion, $\be_4<0$, nonzero quadratic dispersion $\be_2\neq0$ and positive Kerr nonlinearity $\gamma>0$, giving rise to the following nondimensionalised version of \eqref{4NLS_original}
	\begin{equation}\label{4NLS}
		i \psi_t = \psi_{xxxx} + \sigma_2 \psi_{xx} - |\psi|^2\psi,
	\end{equation}
	where $\psi:\R \times \R \to \C$ and $\sigma_2 = \sgn\be_2$. (For the transformations used to obtain \eqref{4NLS} from \eqref{4NLS_original} for $\be_4<0, \be_2\neq0$, we refer the reader to \cite[Table 1]{BGBK21}.)  The modifications needed to treat pure quartic solitons for which $\be_2=0$ will be given in \cref{sec:concluding_rems}.
	
	Our focus will be to determine the spectral stability of standing wave solutions
	\begin{equation}\label{SWave}
		\psi(x,t) = e^{i\be t} \phi(x), \qquad \phi(x)\in \R,
	\end{equation}
	to \eqref{4NLS}, subject to perturbations in $L^2(\R;\C)$. Substituting  \eqref{SWave} into \eqref{4NLS} shows that the wave profile $\phi$ satisfies the \emph{standing wave equation}
	\begin{equation}\label{SWE}
		\phi'''' + \sigma_2 \phi '' + \beta \phi - \phi^3 = 0.
	\end{equation}
	Using the change of variables 
	\begin{gather}
		\phi_1 = \phi '' + \sigma_2 \phi, \qquad \phi_2 = \phi, \qquad  \phi_3 =\phi',    	\qquad \phi_4 =\phi''',
	\end{gather}
	we may write \eqref{SWE} as the first order Hamiltonian system 
	\begin{equation}\label{1st_order_SWE_Ham}
		\begin{pmatrix}
			\phi_1' \\
			\phi_2' \\
			\phi_3' \\
			\phi_4' 
		\end{pmatrix} = \begin{pmatrix}
			\phi_4 + \sigma_2 \phi_3 \\
			\phi_3 \\
			\phi_1 - \sigma_2 \phi_2 \\
			-\sigma_2 \phi_1 + \phi_2 - \be \phi_2 + \phi_2^3
		\end{pmatrix}.
	\end{equation}
	Motivated by the families of homoclinic orbits discovered in \cite{BGBK21}, we consider orbits of \eqref{1st_order_SWE_Ham} that are homoclinic to the origin, which correspond to soliton solutions of \eqref{4NLS}. An example is given by the exact solution found by Karlsson and H\"o\"ok \cite{KarlssonHook94},
	\begin{equation}\label{KH_soln}
		\phi_{\text{KH}}(x) = \sqrt{\f{3}{10}} \sech^2\left (\f{x}{2\sqrt{5}}\right ),
	\end{equation}
	which solves \eqref{SWE} for the specific values $\beta=4/25$ and $\sigma_2=-1$.
	
	We will assume that the origin in \eqref{1st_order_SWE_Ham} is hyperbolic. Noting that the eigenvalues $\mu$ of the linearisation about the origin solve
	\begin{equation}\label{origin_eigvals}
		\mu^2 = \f{1}{2}\left (-\sigma_2\pm\sqrt{1-4\be} \right )
	\end{equation}
	(where we used that $\sigma_2^2=1$), hyperbolicity holds provided 
	\begin{equation}\label{assumptions}
		\begin{cases}
			\be>0  & \,\,\,\text{if} \quad\sigma_2=-1 \\
			\be>\f{1}{4}  & \,\,\,\text{if} \quad\sigma_2=1.
		\end{cases}
	\end{equation}
	For technical reasons (see \cref{subsec:4.1}), in the first part of \eqref{assumptions} we additionally require
	\begin{equation}\label{assumption2}
		\be\neq \f{1}{4} \quad \text{if} \quad \sigma_2=-1.
	\end{equation}
	Linearising \eqref{4NLS} by substituting the complex-valued perturbation
	\begin{equation}\label{}
		\psi(x,t) = \left [ \phi(x) + \e \left (u(x) + i v(x) \right )  e^{\la t} \right ]e^{i\be x}
	\end{equation}
	for $u,v \in L^2(\R;\R)$ into \eqref{4NLS}, collecting $O(\e)$ terms and separating into real and imaginary parts leads to the following linearised dynamics in $u$ and $v$:
	\begin{gather}\label{uv_system}
		\begin{aligned}
			-u'''' - \sigma_2 u'' - \be u + 3  \phi^2 u &= \la v \\
			-v'''' - \sigma_2 v'' - \be v +  \phi^2 v &= - \la u.
		\end{aligned}
	\end{gather}
	We can write \eqref{uv_system} as the spectral problem
	\begin{equation}\label{eq:N_evp}
		N \begin{pmatrix}
			u \\ v
		\end{pmatrix} = \la \begin{pmatrix}
			u \\ v
		\end{pmatrix},
	\end{equation}
	where $N$ is the unbounded and densely defined linear operator
	\begin{align}\label{eq:N_linear_operator}
		N = \begin{pmatrix}
			0 & -L_- \\ L_+ & 0 
		\end{pmatrix}, \qquad \begin{cases}
			L_- = - \p_x^4 - \sigma_2 \p_x^2 - \be + \phi^2, \\
			L_+ = - \p_x^4 - \sigma_2 \p_x^2 - \be + 3 \phi^2,
		\end{cases}
	\end{align}
	with 
	\begin{equation}\label{domains}
		\dom(N) = H^4(\R) \times H^4(\R), \qquad \dom(L_\pm) = H^4(\R).
	\end{equation}
	{ Note that the operators $L_\pm$ are selfadjoint with the domain \eqref{domains} (see, for example, \cite{Weidmann87_spec_ODOs}). In addition, we show in \cref{subsec:4.1} that their essential spectra are confined to the negative half line provided \eqref{assumptions} holds.} 
	
	We wish to determine whether the spectrum of $N$  intersects the open right half plane. Because $N$ is Hamiltonian, its spectrum has four-fold symmetry, and instability follows from any part of the spectrum lying off the imaginary axis. We will see in \cref{sec:setup2} that, under  \eqref{assumptions}, the essential spectrum of $N$ is confined to the imaginary axis and bounded away from the origin. Using the Maslov index, our specific goal will be to detect unstable eigenvalues that are purely real.
	
	We point out that the equation $L_-\phi=0$ is just \eqref{SWE}, and, differentiating  \eqref{SWE} with respect to $x$, we have $L_+ \phi_x =0 $. Thus
	\begin{equation*}\label{}
		0\in\spec(L_-)\cap \spec(L_+),
	\end{equation*} 
	with $\phi\in\ker(L_-)$ and $\phi_x\in\ker(L_+)$. We make the following simplicity assumption.
	\begin{hypo}\label{hypo:simplicity_assumption}
		$\ker(L_-)=\spn\{\phi\}$ and $\ker(L_+)=\spn\{\phi_x\}$. 
	\end{hypo}
	Notice that when $\la=0$, the eigenvalue equations \eqref{eq:N_evp} decouple into two independent equations, $L_-v=0$ and $L_+u=0$, so that $\ker(N) = \ker(L_+)\oplus \ker(L_-)$. \Cref{hypo:simplicity_assumption} therefore implies that $\ker(N) = \spn\{(\phi_x,0)^\top, (0,\phi)^\top\}$.
	
	{
	To state our first main result, we first introduce some notation. Let us denote
	\begin{align*}
		P &\coloneqq \# \{\text{positive eigenvalues of } L_+\}, \\
		Q &\coloneqq \# \{\text{positive eigenvalues of } L_-\},
	\end{align*}
	and denote by
	\begin{equation}\label{first_order_systems}
		\mathbf{u}_x=A_+(x,\la)\mathbf{u}, \qquad \qquad \mathbf{v}_x=A_-(x,\la)\mathbf{v},
	\end{equation}
	the first-order systems associated with the eigenvalue equations $L_-v=\la v$ and $L_+u=\la u$, respectively. (The substitutions used to reduce the equations to such first order systems, as well as explicit formulas for $A_+(x,\la)$ and $A_-(x,\la)$, will be given in \cref{sec:L+L__counts}.) For all $\la$ lying outside the essential spectrum of $L_+$ and $L_-$ (see \eqref{essspecL}), the constant-coefficient asymptotic systems associated with those in \eqref{first_order_systems} with coefficient matrices
	\begin{equation}\label{asymptotic_systems}
		A_+(\la) \coloneqq \lim_{x\to \pm\infty} A_+(x,\la), \qquad A_-(\la) \coloneqq \lim_{x\to \pm\infty} A_-(x,\la),
	\end{equation}
	are hyperbolic with two-dimensional stable and unstable subspaces. Thus, each system in \eqref{first_order_systems} admits exponential dichotomies on $\R^+$ and $\R^-$. We denote by $\E^u_+(x,\la)$ (resp. $\E^u_-(x,\la)$) the corresponding $x$- and $\la$-dependent two dimensional subspace of solutions to $\mathbf{u}_x=A_+(x,\la)\mathbf{u}$ (resp. $\mathbf{v}_x=A_-(x,\la)\mathbf{v}$) that decay to zero as $x\to -\infty$, and we denote by $\Ss_+(\la)$ (resp. $\Ss_-(\la)$) the $\la$-dependent two-dimensional stable subspace of the matrix $A_+(\la)$ (resp. $A_-(\la)$). We introduce the following definition for nontrivial intersections of $\E^u_\pm(x,\la)$ with $\Ss_\pm(\la)$ when $\la=0$.
	\begin{define}\label{define:conj_pointsL}
		An $L_+$ \emph{conjugate point} is a value $x_0\in\R$ such that $\E^u_+(x_0,0) \cap \Ss_+(0)\neq \{0\}$. An $L_-$ \emph{conjugate point} is similarly defined via $\E^u_-(x_0,0) \cap \Ss_-(0)\neq \{0\}$. We denote the total number of $L_+$ and $L_-$ conjugate points on $\R$, counted with multiplicity, by
		\begin{equation}\label{pcqc}
			p_c \coloneqq \sum_{x\in\R}\dim\left ( \E^u_+(x,0)\cap \Ss_+(0) \right ) \qquad \text{and} \qquad q_c \coloneqq \sum_{x\in\R} \dim\left ( \E^u_-(x,0)\cap \Ss_-(0) \right )
		\end{equation}
		respectively. 
	\end{define}
	Our first main result states that the Morse indices of $L_+$ and $L_-$ can be computed by counting conjugate points. Similar results for certain classes of selfadjoint fourth order operators may be found in \cite{Howard21_Hormander, Howard21,BJP24}.
	\begin{theorem}\label{thm:Lpm_eval_counts}
		Assume \cref{hypo:simplicity_assumption} and the conditions \eqref{assumptions}--\eqref{assumption2}. Then 
		\begin{equation}\label{eq:P=conj_points}
			P= p_c,
		\end{equation}
		and
		\begin{equation}\label{eq:Q=conj_points}
			Q= q_c.
		\end{equation}
	\end{theorem}
	\begin{rem}
		\Cref{thm:Lpm_eval_counts} will also hold in the case of soliton solutions to \eqref{4NLS} with any even-integer power-law nonlinearity, i.e.
		\begin{equation}\label{general_NLS_powerlaw}
			i \psi_t = \psi_{xxxx} + \sigma_2 \psi_{xx} - |\psi|^{2p}\psi, \qquad p\in\Z^+,
		\end{equation}
		 as studied in \cite{KarpShag97,KarpShag2000}; for more details, see \cref{rem:general_NLS_powerlaw,rem:general_NLS_powerlaw2}. However, with the soliton solutions of \cite{BGBK21} in mind, we have stated our results for the cubic case. 
	\end{rem}
}
	To state our next main result, let us denote 
	\begin{align*}
		n_+(N) &\coloneqq   \#\{\text{positive real eigenvalues of } N\},
	\end{align*}
	and define the quantities 
		\begin{equation}\label{key_integrals}
				\mathcal{I}_1\coloneqq  \int_{-\infty}^{\infty}   \phi_x\, \widehat{v} \,dx, \qquad \quad  
				\mathcal{I}_2\coloneqq  \int_{-\infty}^{\infty} \phi \,\widehat{u}\, dx,
			\end{equation}
		where $\widehat{v}$ is any solution in $ H^4(\R)$ to $-L_- v = \phi_x$ and $\widehat{u}$ is any solution in $ H^4(\R)$ to $L_+u=\phi$. (Note that the equations  $-L_- v = \phi'$ and $L_+u=\phi$ each satisfy a solvability condition that guarantees the existence of such solutions; for details see \cref{sec:proof_lower_bound}.)
	\begin{theorem}\label{thm:main_lower_bound}
		Assume \cref{hypo:simplicity_assumption}, the conditions \eqref{assumptions}--\eqref{assumption2}, and suppose $\mathcal{I}_1, \mathcal{I}_2 \neq0$. The number of positive, real eigenvalues of the operator $N$ satisfies
		\begin{equation}\label{lower_bound}
			n_+(N) \geq |P-Q-\mathfrak{c}|,
		\end{equation}
		where 
		\begin{equation}\label{}
			\mathfrak{c} = \begin{cases}
				1 & \,	\mathcal{I}_1>0, \, \mathcal{I}_2<0, \\
				0 & \,\mathcal{I}_1\mathcal{I}_2>0, \\
				-1  & \,\mathcal{I}_1<0, \,\mathcal{I}_2>0.
			\end{cases}
		\end{equation}
	\end{theorem}
	\begin{rem}
		In the case that either $\mathcal{I}_1$ or $\mathcal{I}_2$ vanishes, an extra calculation  is needed to compute the correction term $\mathfrak{c}$ (the definition of which is given in \eqref{mathfrakc}); for details, see \cref{sec:concluding_rems}.
	\end{rem}
	\begin{rem}\label{rem:everywhere_applicable}
		In this work we make no comment on the existence of soliton solutions to \eqref{4NLS}, i.e. homoclinic solutions to \eqref{1st_order_SWE_Ham}. Rather, we prove that if such a solution exists, then its associated linearised operator $N$ satisfies \cref{thm:main_lower_bound}. { Our analysis makes no assumptions on the specific structure of $\phi$ -- apart from it being a homoclinic solution to \eqref{1st_order_SWE_Ham} -- and hence our results are applicable to all single \emph{and} multi-hump solitons of \eqref{4NLS}.}
	\end{rem}

	The following Jones-Grillakis instability theorem \cite{J88,Grill88,KapProm} is an immediate consequence of \cref{thm:main_lower_bound}.
	\begin{cor}\label{cor:JonesGrillakis}
		Standing waves for which $P-Q \neq -1,0,1$ are spectrally unstable.
	\end{cor}
	
	We also have the following VK-type criterion \cite{VK73,pelinovsky}.
	\begin{theorem}\label{thm:VK_criterion}
		Suppose $P=1$ and $Q=0$. The standing wave $\widehat\psi$ is spectrally unstable if $\mathcal{I}_2>0$ and is spectrally stable if $\mathcal{I}_2<0$.
	\end{theorem}
	\begin{rem}
		If there exists a $C^1$ family of solutions $\be\to \phi (\cdot;\be) \in H^4(\R)$ to the standing wave equation \eqref{SWE}, then $\widehat{u}=\p_\be \phi (\cdot;\be)$ and the integral $\mathcal{I}_2$ is precisely that appearing in the VK criterion for standing waves in the classical (second-order) NLS equation (see \cite[\S4.2]{pelinovsky}), i.e.
		\begin{equation*}\label{}
			\mathcal{I}_2 = \f{1}{2}\pde{}{\be} \int_{-\infty}^{\infty} \phi^2 dx.
		\end{equation*} 
	\end{rem}
	
	{
	\begin{rem}\label{rem:alternate_lower_bound}
		\Cref{thm:Lpm_eval_counts} provides a convenient numerical tool for computing $P$ and $Q$. In light of this, the lower bound in \cref{thm:main_lower_bound} and the sufficient condition in \cref{cor:JonesGrillakis} may be computed in terms of the counts of $L_+$ and $L_-$ conjugate points; for example, \eqref{lower_bound} is equivalent to
		\begin{equation}\label{}
			n_+(N) \geq\left  | p_c - q_c - \mathfrak{c}\, \right|.
		\end{equation}
	\end{rem}
	}
	The paper is organised as follows. { The focus of the early sections of the paper will be on the eigenvalue problem for $N$ and setting up the proof of \cref{thm:main_lower_bound}.} In particular, in \cref{sec:setup2} we write down the first order system associated with \eqref{eq:N_evp} and compute the essential spectrum of $N$. We also define the stable and unstable bundles, the main objects of our analysis. In \cref{sec:symplectic_maslov}, we provide some background material on the Maslov index, which includes the definition of higher order crossing forms due to Piccione and Tausk \cite{PT09}, before describing the homotopy argument that will lead to the proof of the lower bound in \cref{thm:main_lower_bound}. \Cref{sec:L+L__counts} is then devoted to the analysis of the $L_+$ and $L_-$ eigenvalue problems and the proof of \cref{thm:Lpm_eval_counts}. \Cref{sec:proof_lower_bound} is devoted to the proofs of \cref{thm:main_lower_bound,thm:VK_criterion}. In \cref{sec:KH_stability} we apply our theory to confirm the spectral stability of the Karlsson and H\"o\"ok solution \eqref{KH_soln}, which will involve numerically computing the number of $L_+$ and $L_-$ conjugate points. { We focus on this application because it is the only soliton solution for which we had numerics readily available. Indeed, applying our results to the multi-hump solitons found by numerical means in \cite{BGBK21}, as well as the pure quartic solitons obtained by the analytical methods outlined in \cite{BMQ24}, will be the subject of future work.} In \cref{sec:concluding_rems} we give some concluding remarks on our analysis, and in \cref{sec:appendixA} we complete the proof of \cref{thm:Lpm_eval_counts} by removing a certain hypothesis used in the proof presented in \cref{sec:L+L__counts}.
	
	\section{Set-up}\label{sec:setup2}
	
	We first compute the essential spectrum of the operator $N$. Using the change of variables 
	\begin{gather}\label{subs}
		\begin{aligned}
			u_1&=u'' + \sigma_2 u,		& \qquad 	u_2&=u,	& \qquad u_3&= u',	& \qquad  u_4&=u''',  \\
			v_1&=v'' + \sigma_2  v,	& \qquad 	v_2&= -v, &  \qquad v_3&= -v', 	&\qquad   v_4&= v''', \\
		\end{aligned}
	\end{gather}
	we  convert \eqref{uv_system} to the (infinitesimally symplectic) first order system
	\begin{equation}\label{1st_order_sys}
		\begin{pmatrix}
			u_1 \\ v_1 \\ u_2 \\ v_2 \\ u_3 \\ v_3 \\ u_4 \\ v_4 \\
		\end{pmatrix} ' = 
		\left(\begin{array}{@{}c|c@{}}
			0 &  \begin{matrix}
				\sigma_2  & 0 & 1 & 0 \\
				0 & -\sigma_2 & 0 & 1 \\
				1 & 0 & 0 & 0 \\
				0 & 1 & 0 & 0 
			\end{matrix}  \\
			\hline \\[-4mm]
			\begin{matrix}
				1 & 0 & -\sigma_2  & 0  \\ 
				0 & -1 & 0  & -\sigma_2   \\ 
				-\sigma_2  & 0 & \al(x)  & \la  \\
				0 & -\sigma_2 & \la    & \eta(x)
			\end{matrix}  & 0 
		\end{array}\right)
		\begin{pmatrix}
			u_1 \\ v_1 \\ u_2 \\ v_2 \\ u_3 \\ v_3 \\ u_4 \\ v_4 \\
		\end{pmatrix},
	\end{equation}
	where 
	\begin{align*}
		\al(x)\coloneqq  3 \phi(x)^2 -\be +1, \qquad \eta(x)  \coloneqq  - \phi(x)^2 +\be - 1.
	\end{align*}
	Setting
	\begin{equation}
		B =\begin{pmatrix}
			\sigma_2 & 0 & 1 & 0 \\
			0 & -\sigma_2   & 0 & 1 \\
			1 & 0 & 0 & 0 \\
			0 & 1 & 0 & 0 
		\end{pmatrix}, \qquad C(x;\la) = \begin{pmatrix}
			1 & 0 & -\sigma_2  & 0  \\ 
			0 & -1 & 0  & -\sigma_2   \\ 
			-\sigma_2  & 0 & \al(x)  & \la   \\
			0 & -\sigma_2 & \la   & \eta(x)
		\end{pmatrix},
	\end{equation}
	we can write \eqref{1st_order_sys} as
	\begin{equation}\label{1st_order_sys_N_vector}
		\mathbf{w}_x = A(x;\la) \mathbf{w},
	\end{equation}
	where 
	\begin{equation*}
		\mathbf{w} = (u_1, v_1, u_2, v_2, u_3, v_3, u_4, v_4)^\top, \qquad 	A(x;\la) =  \begin{pmatrix}
			0 & B\\ C(x;\la) & 0
		\end{pmatrix}.
	\end{equation*}
	The asymptotic system for \eqref{1st_order_sys} is given by
	\begin{equation}\label{asymp_system}
		\mathbf{w}_x = A_\infty(\la)\mathbf{w}, \qquad A_\infty(\la) \coloneqq \lim_{x\to\pm\infty}A(x;\la).
	\end{equation}
	(The endstates as $x\to\pm\infty$ are the same because $\phi$ is homoclinic to the origin.) It now follows from \cite[Theorem 3.1.11]{KapProm} that the essential spectrum of $N$ is given by the set of $\la\in \C$ for which the matrix $A_\infty(\la)$ has a purely imaginary eigenvalue. A short calculation shows that
	\begin{equation}\label{}
		\esspec(N)  = \{ \la\in \C: \lambda ^2= -\left( - k ^4  +\sigma_2  k ^2 - \beta \right )^2\,\,\,\text{for some}\,\,\, k\in\R \}.
	\end{equation}
	{ The biquadratic polynomial $ -k ^4  +\sigma_2  k ^2 - \beta $ has the global maximum value $-\be$ if $\sigma_2=-1$, or $\f{1}{4} - \be$ if $\sigma_2=1$. Under \eqref{assumptions}, we therefore have
	\begin{equation}\label{essspecN}
		\esspec(N) = \begin{cases}
			 \big(-i\infty, -i \be\big] \cup \big[ i\be,+i\infty  \big)   & \sigma_2=-1, \\
			 \big (-i\infty, -i (\be-\f{1}{4}) \big]\cup \big[ i(\be-\f{1}{4}),+i\infty \big)    & \sigma_2=1,
		\end{cases}
		\end{equation}
	so that $\esspec(N) \subset i\R$ with a spectral gap about the origin. 
	
	It follows that the operator $N-\la I$ of  \eqref{eq:N_evp}--\eqref{domains} is Fredholm for all $\la\in\R$. Hence, the densely-defined closed linear operator 
	\[
	T(\la) : H^1(\R,\R^8) \lra L^2(\R,\R^8), \qquad T(\la)\mathbf{w} \coloneqq \de{\mathbf{w}}{x} - A(x ;\la)\mathbf{w},
	\]
	associated with \eqref{1st_order_sys_N_vector} is also Fredholm (see \cite[\S3.3]{sandstede02}). Moreover, the asymptotic matrix $A_\infty(\la)$ is hyperbolic for all $\la\in\R$, with four eigenvalues with positive real part and four with negative real part:
	\begin{equation}\label{A_asymp_evals}
		\spec(A_\infty(\la)) = \bigg\{\pm\frac{\sqrt{- \sigma_2 \pm \sqrt{1-4 \beta \pm 4  \lambda i \,}}}{\sqrt{2}}\bigg\}.
	\end{equation}
	We denote the corresponding four-dimensional stable and unstable subspaces by $\Ss(\la)$ and $\U(\la)$ respectively. }
	
	By \cite[Theorem 3.2, Remark 3.3]{sandstede02}, \eqref{1st_order_sys} then has exponential dichotomies on $\R^+$ and $\R^-$. That is, for each fixed $\la\in\R$, on each of the intervals $\R^+$ and $\R^-$ the space of solutions to \eqref{1st_order_sys} is the direct sum of two subspaces: one consisting solely of the solutions that decay exponentially for decreasing $x$, and the other of solutions that decay for increasing $x$. By flowing under \eqref{1st_order_sys}, each of these solution spaces can be extended to all of $\R$. Thus we may define, for $(x,\la)\in\R\times \R$, the two-parameter families of subspaces, 
	\begin{align}\label{bundles_defn}
		\begin{split}
			\E^u(x,\la) &\coloneqq \{ \xi\in \R^8 : \xi = \mathbf{w}(x;\la),\,\, \mathbf{w} \,\,\text{solves}\,\,\eqref{1st_order_sys} \,\, \text{and}\,\, \mathbf{w}(x;\la) \to 0 \,\,\text{as} \,\,x\to-\infty \}, \\
			\E^s(x,\la) &\coloneqq \{ \xi\in \R^8 :\xi = \mathbf{w}(x;\la), \,\, \mathbf{w} \,\,\text{solves}\,\,\eqref{1st_order_sys} \,\, \text{and}\,\, \mathbf{w}(x;\la) \to 0 \,\,\text{as} \,\,x\to+\infty \},
		\end{split}
	\end{align}
	corresponding to the evaluation at $x\in\R$ of the solutions to \eqref{1st_order_sys} that decay (exponentially) as $x \to -\infty$ and as $x\to +\infty$, respectively. Following \cite{AGJ90,Corn19}, we call these families the \emph{unstable} and \emph{stable bundles} respectively. Considering $\U(\la),\Ss(\la), \E^u(x,\la), \E^s(x,\la)$ for each $x\in\R$ and $\la\in \R$ as points in the Grassmannian of four-dimensional subspaces of $\R^8$,
	\[
	\Gr_4(\R^8) = \{ V\subset \R^8: \dim V=4\},
	\]
	which (following \cite{F04,HLS17}) we equip with the metric $d(V,U) = \| \mathbb{P}_V - \mathbb{P}_U \|$, where $\mathbb{P}_V$ is the orthogonal projection onto $V$ and $\|\cdot\|$ is any matrix norm, we have that
	\begin{equation}\label{bundle_limits}
		\lim_{x\to-\infty} \E^u(x,\la) = \U(\la), \qquad \lim_{x\to+\infty} \E^s(x,\la) = \Ss(\la).
	\end{equation}
	That is, the orthogonal projections onto $\E^u(x,\la)$ and $\E^s(x,\la)$ converge to those on $\U(\la)$ and $\Ss(\la)$ as $x\to-\infty$ and as $x\to+\infty$, respectively; see \cite[Corollary 2]{PSS97}.

	The key feature of system \eqref{1st_order_sys} that makes it amenable to the Maslov index is the infinitesimally symplectic structure of the coefficient matrix $A(x;\la)$, i.e.
	\begin{equation}\label{infsymp}
		A(x;\la)^TJ + JA(x;\la) = 0,
	\end{equation}
	which follows from the symmetry of $B$ and $C(x;\la)$. This is the motivation for the choice of substitutions \eqref{subs}. Consequently, \eqref{1st_order_sys} induces a flow on the manifold of Lagrangian planes. In particular, the unstable and stable bundles of \eqref{1st_order_sys} are Lagrangian planes of $\R^{8}$ for all $x$ and all $\la$. In addition, we have that $\la_0$ is an eigenvalue of $N$ if and only if for any (and hence all) $x\in\R$ we have
	\[
	\E^u(x,\la_0)\cap\E^s(x,\la_0) \neq \{0\}.
	\]
	Furthermore, the geometric multiplicity of the eigenvalue coincides with the dimension of the Lagrangian intersection,
	\begin{equation}\label{dims_equal}
		\dim \E^u(x,\la_0)\cap\E^s(x,\la_0) = \dim\ker (N-\la_0).
	\end{equation}
	By exploiting homotopy invariance of the Maslov index, we can determine the existence of such intersections by instead analysing the evolution of the unstable bundle $\E^u(x,\la_0)$ when $\la_0=0$; this will amount to counting $L_+$ and $L_-$ conjugate points. This is explained in \cref{subsec:symplectic_evals}.

	\section{A symplectic approach to the eigenvalue problem}\label{sec:symplectic_maslov}

	\subsection{The Maslov index via higher order crossing forms}\label{subsec:maslov_index2}
	In this section we follow the discussions in \cite{A67,RS93,GPP04,GPP_full}. Consider $\R^{2n}$ equipped with the symplectic form 
	\begin{equation}\label{omega2}
		\w(u,v) = \langle Ju,v\rangle_{\R^{2n}}, \qquad J = \begin{pmatrix}
			0_n & -I_n \\ I_n & 0_n 
		\end{pmatrix}.
	\end{equation}
	A \emph{Lagrangian plane} or \emph{Lagrangian subspace} of $\R^{2n}$ is an $n$ dimensional subspace upon which the symplectic form vanishes. We denote the Grassmannian of Lagrangian subspaces by
	\begin{equation}\label{}
		\cL(n) \coloneqq \{ \Lambda \subset \R^{2n} : \dim \Lambda = n, \,\,\w(u,v) = 0 \,\,\forall\,\, u,v \in \Lambda\}.
	\end{equation}
	A \emph{frame} for a Lagrangian subspace $\Lambda\in\cL(n)$ is a $2n\times n$ matrix with rank $n$ whose columns span $\Lambda$. Such a frame may be written in block form by
	\begin{equation*}\label{}
		\begin{pmatrix}
			X \\ Y
		\end{pmatrix}, \qquad X,Y \in \R^{n\times n},
	\end{equation*}
	where $X^\top Y = Y^\top X$; the symmetry of $X^\top Y$ follows from the vanishing of \eqref{omega2}. Such a frame is not unique; right multiplication by an invertible $n\times n$ matrix will yield an alternate frame for $\Lambda$. In particular, if $X$ is invertible then an alternate frame is given by
	\begin{equation*}\label{}
		\begin{pmatrix}
			I \\ YX^{-1}
		\end{pmatrix}, \quad \text{where} \quad\left (YX^{-1}\right )^\top = YX^{-1}.
	\end{equation*}
	Let $\Lambda:[a,b]\to\cL(n)$ be a path in $\cL(n)$. Its Maslov index is, roughly speaking, a signed count of its intersections with a certain codimension-one set. More precisely, Arnol'd \cite{A67} gave the following definition for non-closed paths satisfying certain conditions.

	Fix $V\in\cL(n)$. The \emph{train} $\cT(V)$ of $V$ is the set of all Lagrangian planes that intersect $V$ nontrivially; it may be decomposed into a family of submanifolds via $\cT(V) = \bigcup_{k=1}^n \cT_k(V)$, where $\cT_k(V)\coloneqq \{W\in\cL(n):\dim(W\cap V) = k\}$ is the set of  Lagrangian planes that intersect $V$ in a subspace of dimension $k$. It is shown in \cite{A67} that $\codim \cT_k(V) = k(k+1)/2$; in particular, $\codim \cT_1(V)=1$. From the fundamental lemma of \cite{A67}, $\cT_1(V)$ is two-sidedly embedded in $\cL(n)$; that is, $\cT_1(V)$ is transversely oriented by the velocity field of some (and then of any) one-parameter positive definite Hamiltonian \cite{A67,A85}. Such a vector field thus defines a `positive' and a `negative' side of $\cT_1(V)$. Define a \emph{crossing} to be a value $t_0\in [a,b]$ such that $\Lambda(t_0)\cap V \neq \{0\}$, i.e. $\Lambda(t_0)\in\cT(V)$. The \emph{Maslov index} of any continuous curve $\Lambda:[a,b]\to\cL(n)$ with endpoints lying off the train and with crossings lying only in $\cT_1(V)$ is then defined to be $\nu_+-\nu_-$, where $\nu_+$ is the number of points of passage of $\Lambda$ from the negative to the positive side of $\cT_1(V)$, and $\nu_-$ is defined conversely.

	Arnol'd's definition was extended by Robbin and Salamon \cite{RS93} to paths with arbitrary endpoints and with crossings possibly lying in $\cT_k(V)$ for $k\geq 2$. This was done by exploiting an identification of the tangent space of $\cL(n)$ at some $\Lambda\in\cL(n)$ with the space $S^2\Lambda$ of quadratic forms on $\Lambda$. This lead to the construction of the crossing form, a quadratic form associated with each crossing whose signature determines local contributions to the Maslov index. The definition given by Robbin and Salamon requires that all crossings are regular, with the definition extended to all continuous Lagrangian paths via homotopy invariance (see \cref{prop:enjoys}). As outlined in the introduction, it is desirable to be able to compute the Maslov index directly, without having to use perturbative arguments.

	Piccione and Tausk \cite{PT09} provided the means to do exactly that, defining the Maslov index for analytic Lagrangian paths with non-regular crossings via the \emph{partial signatures} of higher order crossing forms. While strictly speaking their definition is given via the fundamental groupoid (see \cite[\S5.2]{PT09}), and shown to be \emph{computable} via the partial signatures listed below, for our purposes it will suffice to use the latter computational tool as our definition of the Maslov index. This computable formula was previously given by Giamb\`{o}, Piccione and Portaluri \cite{GPP_full,GPP04} through the related notion of the spectral flow of an associated family of symmetric bilinear forms (see \cite[Proposition 3.11]{GPP_full}); for details of the equivalence of these definitions, see \cite[\S4 and \S 5.5]{PT09}.

	The following notions can be found in \cite[\S5.5]{PT09}. Suppose that $\Lambda:[a,b] \to \cL(n)$ is an analytic path of Lagrangian subspaces, and $t=t_0$ is a crossing, that is, $\Lambda(t_0)\cap V\neq \{0\}$. A \emph{$V$-root function} (or simply a \emph{root function} when the choice of $V$ is clear) for $\Lambda$ at $t_0$ is a differentiable mapping $w:[t_0-\e,t_0+\e] \to \R^{2n}$, $\e>0$, such that $w(t)\in\Lambda(t)$ and $w(t_0) \in V$. The \emph{order} of $w$, $\text{ord}(w)$, is the smallest positive integer $k$ such that $w^{(k)}(t_0) \notin V$. This allows one to define a sequence of nested subspaces $W_k(\Lambda,V,t_0)$, $k\geq 1$, via
	\begin{equation*}
		W_k(\Lambda,V,t_0)\coloneqq \{w_0: \exists \, \,\text{a $V$-root function\,\,} w \,\,\text{for}\,\, \Lambda \text{\,\,with\,} \ord(w)\geq k \text{\,\,and\,\,} w(t_0)=w_0\}
	\end{equation*}
	(called the \emph{kth degeneracy space}), which satisfy
	\begin{equation}\label{Wk_properties}
		W_{k+1}(\Lambda,V,t_0) \subseteq W_k(\Lambda,V,t_0)\text{\,\,\,for all\,\,\,} k\geq 1, \qquad W_1(\Lambda,V,t_0) = \Lambda(t_0)\cap V.
	\end{equation}
	(When $\Lambda, V$ and $t_0$ are clear, we will simply write $W_k$.) The first fact in \eqref{Wk_properties} follows immediately from the definition,  while the second follows from the fact that every $V$-root function $w$ has $\ord(w)\geq 1$. One can then define a symmetric bilinear form $\mathfrak{m}_{t_0}^{(k)}(\Lambda,V) : W_k \times W_k \lra \R$ (as in \cite[Definition 5.5.5]{PT09}),
	\begin{equation}\label{kth_crossing_sym_bilinear_form}
		\mathfrak{m}_{t_0}^{(k)}(\Lambda,V)(w_0,v_0)  \coloneqq \dek{}{t}\,\w\big(w(t),v_0 \big )\Big|_{t=t_0}, 
	\end{equation}
	where $w$ is any $V$-root function for $\Lambda$ at $t_0$ with $\ord(w)\geq k$ and $w(t_0)=w_0$. (When $\Lambda$ and $V$ are clear, we will simply write $\mathfrak{m}_{t_0}^{(k)}$.) That  $\mathfrak{m}_{t_0}^{(k)}$ is independent of the choice of root function and therefore well-defined follows from \cite[Corollary 5.5.4]{PT09}; that it is symmetric follows from \cite[Lemma 5.5.3]{PT09}. Using the definition of $W_k$ and $\mathfrak{m}_{t_0}^{(k-1)}$, it is straightforward to show that 
	\begin{equation}\label{Wk1_kernel_kform}
		W_{k} = \ker \mathfrak{m}_{t_0}^{(k-1)}.
	\end{equation}
	In light of \eqref{Wk_properties}, it follows that $W_k$ is simply the subspace of $\Lambda(t_0) \cap V$ upon which the crossing forms up to order $k-1$ are zero.
	
	We define the higher order generalisation of the crossing form defined by Robbin and Salamon to be the quadratic form associated with \eqref{kth_crossing_sym_bilinear_form}.
	\begin{define}\label{define:kth_crossing_form}
		The \emph{$k$th-order crossing form} is the quadratic form
		\begin{equation}\label{kth_crossingform0}
			\mathfrak{m}_{t_0}^{(k)}(\Lambda,V)(w_0) \coloneqq \mathfrak{m}_{t_0}^{(k)}(\Lambda,V)(w_0,w_0)  = \dek{}{t}\,\w\big(w(t),w_0 \big )\Big|_{t=t_0} = \w \left (w^{(k)}(t_0),w_0\right )
		\end{equation}
		for $w_0\in W_k$, where $w$ is any $V$-root function for $\Lambda$ at $t_0$ with $\ord(w)\geq k$ and $w(t_0)=w_0$.
	\end{define}	
	To compute higher order crossing forms in the spatial parameter, we will make use of the following fact.
	
	\begin{lemma}\label{lemma:compute_kth_form_in_practice}
		Let $\mathbf{Z}(t)$ be a frame for $\Lambda(t)$. There exists a $V$-root function $w$ for $\Lambda$ at $t=t_0$ with  $\ord(w)\geq k$ if and only if there exist $k$ vectors $\{h_0,\dots,h_{k-1}\}$ in $\R^n$ such that 
		\begin{equation}\label{k_equations}
			\sum_{j=0}^{i}{i \choose j}\mathbf{Z}^{(i-j)}(t_0) h_{j} \in V \quad \text{for every}\quad i\in\{0,1,\dots,k-1\}. 
		\end{equation} 
		In this case, the $k$th order crossing form \eqref{kth_crossingform0} is given by
		\begin{equation}\label{kth_crossing_form_practice}
			\mathfrak{m}_{t_0}^{(k)}(\Lambda,V)(w_0) =\sum_{j=0}^{k-1} {k \choose j} \,\w( \mathbf{Z}^{(k-j)}(t_0) h_{j},w_0), \qquad w_0\in W_{k}.
		\end{equation}
	\end{lemma}
	\begin{proof}
		Observe that if $w$ is a $V$-root function, then $w(t) = \mathbf{Z}(t) h(t)\in \Lambda(t)$ for some smooth $h:[t_0-\e,t_0+\e]\to\R^n$, and
		\begin{equation}\label{}
			w_i\coloneqq \dei{}{t} w(t) |_{t=t_0} = \sum_{j=0}^{i}{i \choose j}\mathbf{Z}^{(i-j)}(t_0)h^{(j)}(t_0)  \in V \,\,\,\, \text{for every} \,\,\,\, i\in\{0,1,\dots,k-1\}.
		\end{equation}
		Thus the vectors $h_j\coloneqq h^{(j)}(t_0)$ satisfy \eqref{k_equations}. Conversely, suppose there exist vectors $\{h_0,\dots,h_{k-1}\}$ such that \eqref{k_equations} holds. Then 
		\begin{equation}\label{rootfn}
			w(t) = \mathbf{Z}(t) \sum_{j=0}^{k-1}\f{(t-t_0)^j}{j!}h_{j}
		\end{equation}
		is a $V$-root function for $\Lambda$ at $t_0$ with  $\ord(w)\geq k$, as seen by differentiating $w$ $i$ times for $i\in\{0,1,\dots,k-1\}$ at $t_0$. Substituting \eqref{rootfn} into \eqref{kth_crossingform0} yields \eqref{kth_crossing_form_practice}.
	\end{proof}

	A typical application of this lemma is as follows. We first compute the first order form given by \eqref{first-order_RSform} below, where $W_1=\Lambda(t_0)\cap V$. For $k\geq 2$, the space $W_k$ is given via \eqref{Wk1_kernel_kform}. We then find vectors $\{h_0,\dots,h_{k-1}\}$ such that \eqref{k_equations} holds -- in which case $w$ given by \eqref{rootfn} is a root function --  and we compute $\mathfrak{m}_{t_0}^{(k)}$ via \eqref{kth_crossing_form_practice}.

	In the case that $k=1$, for notational convenience we will drop the superscript and write $\mathfrak{m}_{t_0}(\Lambda,V)$. Following \cite{RS93}, a crossing $t_0$ will be called \emph{regular} if $\mathfrak{m}_{t_0}$ is nondegenerate; otherwise, $t_0$ will be called \emph{non-regular}. Denote by $n_+(\mathfrak{m}_{t_0}^{(k)})$ and $n_-(\mathfrak{m}_{t_0}^{(k)})$ the number of positive, respectively negative, squares of $\mathfrak{m}_{t_0}^{(k)}$. Following \cite{PT09,GPP04}, we call the collection of integers
	\[
	n_-(\mathfrak{m}_{t_0}^{(k)}), \quad n_+(\mathfrak{m}_{t_0}^{(k)}), \quad \sgn(\mathfrak{m}_{t_0}^{(k)})=n_+(\mathfrak{m}_{t_0}^{(k)})-n_-(\mathfrak{m}_{t_0}^{(k)}),
	\]
	the \emph{partial signatures} of \eqref{kth_crossingform0}.  The Maslov index of the Lagrangian path $\Lambda$ is then given as follows, as in \cite[Theorem 5.5.9]{PT09} (up to the convention at the endpoints) and \cite[Proposition 3.11]{GPP_full}. As an aside, note that it follows from the arguments in \cite[\S5.5]{PT09} that if $\Lambda$ is analytic, then all crossings are isolated. 
	\begin{define}\label{define:Maslov_GPP}
		Let $V\in\cL(n)$ be fixed, and suppose $\Lambda:[a,b] \to \cL(n)$  is an analytic path of Lagrangian subspaces. The \emph{Maslov index} of $\Lambda$ with respect to $V$ is given by
		\begin{multline}\label{define:Maslov_GPP_eq}
			\Mas(\Lambda,V;[a,b]) = -\sum_{k\geq 1} n_-\left (\mathfrak{m}_{a}^{(k)}\right )+ \sum_{t_0\in(a,b)}\left (  \sum_{k\geq1} \sgn\left (\mathfrak{m}_{t_0}^{(2k-1)}\right )\right ) \\ \qquad \qquad + \sum_{k\geq 1} \left ( n_+\left (\mathfrak{m}_{b}^{(2k-1)}\right )+n_-\left (\mathfrak{m}_{b}^{(2k)}\right )\right ),
		\end{multline}
		where the second sum on the right hand side is taken over all interior crossings, and all sums over $k$ have a finite number of nonzero terms.
	\end{define}	
	Some remarks on the previous definition are in order. First, since all crossings are isolated, the sum over $t_0$ on the right hand side of \eqref{define:Maslov_GPP_eq} has a finite number of nonzero terms. Second, observe that at all interior crossings $t_0\in(a,b)$, only the signatures of the crossing forms of odd order contribute; at the initial point the negative indices of crossing forms of all order contribute; while at the final point, the negative indices of the forms of even order and the positive indices of the forms of odd order contribute. Finally, \cref{define:Maslov_GPP} is best understood in terms of the formula for the spectral flow of a family of symmetric matrices. Namely, one writes the root function as $w(t) = q(t) + R(t) q(t)$, where $q(t)\in V$, $R(t) : V \to W$ for some $W \in \cT_0(V)$, $\Lambda(t) = \graph R(t) = \{ q+ R(t) q: q\in V\}$ and $\Lambda(t_0)\in\cT_0(W)$. One then has a locally defined family of symmetric bilinear forms, $t\mapsto\w(R(t)\cdot,\cdot)|_{V\times V}$ for $t$ near $t_0$, which is degenerate at $t_0$ if and only if $\Lambda(t_0) \cap V \neq\{0\}$. The derivatives of this family at $t_0$, given by \eqref{kth_crossing_sym_bilinear_form}, then determine its spectral flow, i.e. the net change in the number of non-negative eigenvalues, as $t$ increases through $t_0$. For further details, we refer the reader to \cite[Propositions 4.3.9, 4.3.15 and 5.5.7]{PT09} and  \cite[Proposition 2.9, Proposition 3.11]{GPP_full}.

	In \cite{RS93}, Robbin and Salamon exploit the two-sided nature of the train $\cT(V)$ of $V$ to say that a non-degenerate crossing passes from the `positive' side to the `negative' side of the train (or vice versa) according to the signature of the crossing form. (This idea can also be seen in the work of Arnol'd \cite{A85}.) Through the use of higher order crossing forms, we extend this idea to the case where the crossing form is potentially degenerate. In the case of one-dimensional crossings $t_0$, i.e. $\Lambda(t_0)\in\cT_1(V)$, if the first nondegenerate crossing form is of odd order, then $\Lambda$ passes through the train. On the other hand, if the first nondegenerate crossing form is of even order, then $\Lambda$ departs $\cT_1(V)$ in the direction in which it arrived.
	
	It is proven in \cite[Corollary 2.11]{GPP_full} that 
	\begin{equation}\label{dimension_add_up}
		\sum_{k\geq 1} \left ( n_+(\mathfrak{m}_{t_0}^{(k)})  + n_-(\mathfrak{m}_{t_0}^{(k)}) \right ) = \dim \Lambda(t_0) \cap V,
	\end{equation}
	so that by taking sufficiently many higher order crossing forms, a crossing $t_0$ will always contribute $\dim \Lambda(t_0) \cap V$ summands (the signs of which may offset each other) to the Maslov index.
	
	As pointed out in \cite{GPP04}, \cref{define:Maslov_GPP} includes, as a special case, the definition given by Robbin and Salamon in the case that all crossings are regular. To see this, note from the proof of \cref{lemma:compute_kth_form_in_practice} and \eqref{Wk_properties} that, if $\mathbf{Z}(t)$ is a frame for $\Lambda(t)$ and $w_0 = \mathbf{Z}(t_0)h_0 \in \Lambda(t_0) \cap V$, then $w(t) = \mathbf{Z}(t)h_0$ is a root function with $\ord(w)\geq 1$ and 
	\begin{equation}\label{first-order_RSform}
		\mathfrak{m}_{t_0}(\Lambda,V)(w_0) =  \w\left (\mathbf{Z}'(t_0)h_0, \mathbf{Z}(t_0)h_0 \right ) = \left \langle \mathbf{Z}(t_0)^\top J \mathbf{Z}'(t_0) h_0,h_0\right \rangle_{\R^{2n}},
	\end{equation}
	just as in \cite[Theroem 1.1]{RS93}. If $\mathfrak{m}_{t_0}$ is nondegenerate, it follows from \eqref{Wk1_kernel_kform} that $W_2 = \{0\}$ and therefore $W_k=\{0\}$ for $k\geq 3$. Hence $\mathfrak{m}_{t_0}^{(k)}$ is trivial for $k\geq 2$, and from \eqref{dimension_add_up} we have $n_+(\mathfrak{m}_{t_0}) + n_-(\mathfrak{m}_{t_0}) = \dim  \Lambda(t_0)\cap V$. The Maslov index of a path $\Lambda:[a,b]\to \cL(n)$ with only regular crossings is therefore
	\begin{equation}\label{RSdefn}
		\Mas(\Lambda,V;[a,b]) = -n_-\left (\mathfrak{m}_{a}\right ) + \sum_{t_0\in(a,b)}\sgn\left (\mathfrak{m}_{t_0}\right ) + n_+\left (\mathfrak{m}_{b}\right),
	\end{equation}
	as per \cite[\S 2]{RS93} (modulo the convention at the endpoints). 
	
	We will encounter three types of non-regular crossings for paths $\Lambda:[a,b]\to\cL(n)$ in our analysis. The first is a simple interior crossing $t_0\in(a,b)$ for which the first and second order crossing forms are zero, and the third order form is nondegenerate. In this case 
	\begin{equation}\label{simple_3rd_order_nonregular}
		\Mas(\Lambda,V;[t_0-\e,t_0+\e ]) = \sgn \mathfrak{m}_{t_0}^{(3)}.
	\end{equation}
	The second is a two-dimensional interior crossing for which the first order form is degenerate but not identically zero, the second order form is degenerate, and the third order form is nondegenerate. In this case
	\begin{equation}\label{}
		n_+(\mathfrak{m}_{t_0}^{(1)})  + n_-(\mathfrak{m}_{t_0}^{(1)}) + n_+(\mathfrak{m}_{t_0}^{(3)})  + n_-(\mathfrak{m}_{t_0}^{(3)})= \dim \Lambda(t_0) \cap V,
	\end{equation}
	and the contribution of the crossing to the Maslov index of the path is  
	\begin{equation}\label{318}
		\Mas(\Lambda,V;[t_0-\e,t_0+\e ]) =  \sgn \mathfrak{m}_{t_0}^{(1)}  + \sgn \mathfrak{m}_{t_0}^{(3)}.
	\end{equation}
	The third is a two-dimensional crossing occurring at the initial point $t_0=a$, for which the first order form is identically zero and the second order form is nondegenerate. In this case $W_2 = W_1 =\Lambda(t_0)\cap V$, and from \cref{define:Maslov_GPP} we have, for $\e>0$ small enough,
	\begin{equation}\label{second_order_initial}
		\Mas(\Lambda,V;[a,a+\e]) = -n_-(\mathfrak{m}^{(2)}_{a}),
	\end{equation}
	just as in \cite[Proposition 4.15]{CCLM23} and \cite[Proposition 3.10]{DJ11}.

	We summarise the important properties of the Maslov index for the current analysis in the following proposition, as in \cite[Lemma 3.8]{GPP_full} (see also \cite[Theorem 2.3]{RS93}).
	\begin{prop}\label{prop:enjoys}
		The Maslov index enjoys the following properties:
		\begin{enumerate}
			\item (Homotopy invariance.) If two paths $\Lambda_1, \Lambda_2 : [a,b] \lra \cL(n)$ are homotopic with fixed endpoints, then
			\begin{equation}\label{homotopy_property}
				\Mas(\Lambda_1,V; [a,b]) = \Mas(\Lambda_2,\Lambda_0; [a,b]).
			\end{equation}
			\item (Additivity under concatenation.) For $\Lambda(t):[a,c] \lra \cL(n)$ and $a<b<c$, 
			\begin{equation}\label{}
				\Mas(\Lambda,V; [a,c]) = \Mas(\Lambda,V; [a,b]) + \Mas(\Lambda,V; [b,c]).
			\end{equation}
			\item (Symplectic additivity.) Identify the Cartesian product $\cL(n) \times \cL(n)$ as a submanifold of $\cL(2n)$. If $\Lambda=\Lambda_1\oplus\Lambda_2:[a,b] \to \cL(2n)$ where $\Lambda_1,\Lambda_2: [a,b] \to \cL(n)$, and $V=V_1\oplus V_2$ where $V_1, V_2\in\cL(n)$, then
			\begin{equation}\label{case1}
				\Mas(\Lambda,V; [a,b]) = \Mas(\Lambda_1,V_1; [a,b]) + \Mas(\Lambda_2,V_2; [a,b]).
			\end{equation}
			\item (Zero property.) If $\Lambda:[a,b] \lra \cT_k(V)$ for any fixed integer $k$, then
			\begin{equation}\label{}
				\Mas(\Lambda,V; [a,b]) = 0.
			\end{equation}
		\end{enumerate}
	\end{prop}
	We will call a crossing $t=t_0$ \emph{positive} if 
	\begin{equation}\label{}
		\sum_{k\geq 1}  n_+(\mathfrak{m}_{t_0}^{(2k-1)})  =\dim \Lambda(t_0) \cap V ,
	\end{equation}
	and \emph{negative} if 
	\begin{equation}\label{neg_crossings}
		\sum_{k\geq 1}  n_-(\mathfrak{m}_{t_0}^{(2k-1)}) = \dim \Lambda(t_0) \cap V.
	\end{equation}
	In light of \cref{define:Maslov_GPP}, if $t_0$ is a positive interior crossing, or a positive crossing at the final point $t_0=b$, then it contributes $\dim \Lambda(t_0) \cap V$ to the Maslov index. Similarly, if $t_0$ is a negative interior crossing, or a negative crossing at the initial point $t_0=a$, then its contribution is $-\dim \Lambda(t_0) \cap V$. Note, however, from \eqref{define:Maslov_GPP_eq}, that with this convention, the final crossing $t_0=b$ may still contribute $\dim \Lambda(b) \cap V$ if it is not positive, and the initial point $t_0=a$ may still contribute $-\dim \Lambda(a) \cap V$ if it is not negative. 
	
	\begin{rem}\label{rem:monotonicity_geometrically}
		In \cref{sec:L+L__counts}, we will need to make use of the robustness of one dimensional sign-definite crossings. Suppose then $t_0\in[a,b]$ is a one-dimensional crossing, i.e. $\Lambda(t_0)\in\cT_1(V)$. If $t_0$ is positive or negative, then the lowest nonzero crossing form is of odd order, and the order of intersection of the curve $t\mapsto \Lambda(t)$ with the codimension-one submanifold $\cT_1(V)$ is also odd. It follows that the crossing will persist under small perturbations in the train $\cT(V)$. Hence, in a neighbourhood of $t_0$, $\Lambda$ will cross nearby trains $\cT(W)$ transversely and in the same direction, for all $W$ sufficiently close to $V$. 
	\end{rem}
	
	\subsection{The Maslov index for Lagrangian pairs}

	Suppose now that we have a pair of Lagrangian paths $(\Lambda_1,\Lambda_2): [a,b] \to \cL(n)\times \cL(n) $, or a \emph{Lagrangian pair}. Using the symplectic additivity property of \cref{prop:enjoys}, it is possible to define the Maslov index of such an object (as in \cite{RS93,F04,GPP_full,PT09}), where crossings are values $t_0\in[a,b]$ such that $\Lambda_1(t_0) \cap \Lambda_2(t_0)\neq\{0\}$. Precisely, one realises the Lagrangian pair as the path $\Lambda_1\oplus \Lambda_2$ in the space $\cL(2n)$ of Lagrangian planes of $\R^{4n}$ equipped with the symplectic form
	\begin{equation}\label{omega_squared}
		\Om((u_1,u_2)^\top,(v_1,v_2)^\top) = \w(u_1, v_1) - \w(u_2, v_2), \qquad u_1,u_2,v_1,v_2\in\R^{2n}.
	\end{equation}
	Crossings of the pair then correspond to intersections of the path $\Lambda_1\oplus \Lambda_2:[a,b] \to \cL(2n)$ with the diagonal subspace $\triangle = \{(x,x) :x\in\R^{2n}\} \subset \R^{4n}$. The Maslov index of the pair is then defined by
	\begin{equation}\label{define_maslov_pair}
		\Mas(\Lambda_1,\Lambda_2;[a,b]) \coloneqq \Mas(\Lambda_1\oplus \Lambda_2, \triangle;[a,b]).
	\end{equation}
	The right hand side of \eqref{define_maslov_pair} is computed with \cref{define:Maslov_GPP} using the forms \eqref{kth_crossingform0}; doing so leads to the definitions of the following objects pertaining to the Lagrangian pair as in \cite[Exercise 5.18]{PT09}. A mapping $(w_1,w_2):[t_0-\e,t_0+\e] \to \R^{2n}\times \R^{2n}, \e>0,$ will be called a \emph{root function pair} for $(\Lambda_1,\Lambda_2)$ at $t=t_0$ if $w_1(t) \in \Lambda_1(t), w_2(t) \in \Lambda_2(t)$ and $w_1(t_0) = w_2(t_0)$. The \emph{order} of $(w_1,w_2)$, $\ord(w_1,w_2)$ is the smallest positive integer $k$ such that $w_1^{(k)}(t_0) \neq w_2^{(k)}(t_0)$. We then define
	\begin{align}
		\begin{split}
			W_k(\Lambda_1,\Lambda_2,t_0) \coloneqq \{w_0 : \exists \,\,\text{a root}\,\,&\text{function pair \,} (w_1,w_2) \,\,\text{for} \,\,(\Lambda_1,\Lambda_2)\,\,          \\ 
			&\text{with} \,\, \ord(w_1,w_2)\geq k \,\,\text{and}\,\,w_1(t_0) = w_2(t_0) = w_0 \},
		\end{split}
	\end{align}
	and when the choice of $t_0$ is clear we simply write $W_k(\Lambda_1,\Lambda_2)$. As in \eqref{Wk_properties}, we have $W_{k+1}(\Lambda_1,\Lambda_2,t_0) \subseteq W_k(\Lambda_1,\Lambda_2,t_0)$ for $k\geq 1$ and $W_1 = \Lambda_1(t_0) \cap \Lambda_2(t_0)$. The \emph{kth-order relative crossing form} is then the quadratic form
	\begin{align}\label{rel_cross_form}
		\begin{split}
			\mathfrak{m}^{(k)}_{t_0}(\Lambda_1,\Lambda_2)(w_0) &\coloneqq \dek{}{t}\, \w\big(w_1(t), w_0 \big ) \Big|_{t=t_0} -\, \dek{}{t}\w\big(w_2(t), w_0 \big ) \Big|_{t=t_0}, \\
			&= \w\big(w_1^{(k)}(t_0) - w_2^{(k)}(t_0),w_0 \big ),
		\end{split}
	\end{align}
	for $w_0 \in W_k\left (\Lambda_1,\Lambda_2\right )$, where $(w_1,w_2)$ is any root function pair for $(\Lambda_1,\Lambda_2)$ at $t_0$ with $\ord(w_1,w_2)\geq k$. Analogous to \eqref{Wk1_kernel_kform}, we have $W_k(\Lambda_1,\Lambda_2) = \ker \mathfrak{m}^{(k-1)}_{t_0}(\Lambda_1,\Lambda_2)$. The left hand side of \eqref{define_maslov_pair} is thus computed using the forms \eqref{rel_cross_form} in \cref{define:Maslov_GPP}. In the case that $\Lambda_2=V$ is constant, the computation reduces to the Maslov index of the path $\Lambda_1$ with respect to the reference plane $V$ described in \cref{subsec:maslov_index2}.

	The Maslov index is invariant for Lagrangian pairs that are \emph{stratum homotopic}; we give a proof of this fact below. The corresponding result for single paths can be found in\cite[Theorem 2.4]{RS93}. Suppose the pairs $(\Lambda_1,\Lambda_2):[a,b] \to \cL(n)\times\cL(n)$ and $(\widetilde{\Lambda}_1, \widetilde{\Lambda}_2):[a,b] \to \cL(n)\times\cL(n)$ are stratum homotopic, i.e. there exist continuous mappings $H_1, H_2 : [0,1]\times [a,b] \to \cL(n)$ such that
	\begin{align*}
		&&	H_1(0,\cdot) &= \Lambda_1(\cdot), & H_2(0,\cdot) &=\Lambda_2(\cdot)& \\
		&&	H_1(1,\cdot ) &= \widetilde{\Lambda}_1(\cdot), & H_2(1,\cdot) &= \widetilde{\Lambda}_2(\cdot), &
	\end{align*}
	for which $\dim(H_1(s,a)\cap H_2(s,a))$ and $\dim(H_1(s,b)\cap H_2(s,b))$ are constant in $s\in[0,1]$. We then have the following. 
	\begin{lemma}\label{prop:homotopy_invariance_pairs}
		Suppose $\Lambda_1,\Lambda_2,\widetilde{\Lambda}_1,\widetilde{\Lambda}_2$ are as in the previous paragraph. Then
		\begin{equation}\label{eq:stratum_homotopy}
			\Mas(\Lambda_1, \Lambda_2;[a,b]) = \Mas(\widetilde{\Lambda}_1, \widetilde{\Lambda}_2;[a,b]).
		\end{equation}
	\end{lemma}
	\begin{proof}
		Consider the continuous mapping $H = H_1 \oplus H_2 :[0,1]\times [a,b] \to \cL(2n)$. By continuity of $H$ and homotopy invariance (cf. \cref{prop:enjoys}), we have
		\begin{multline}\label{eq:homotopy1}
			\Mas(H(0,\cdot),\triangle;[a,b]) + \Mas(H(\cdot,b),\triangle;[0,1]) \\- \Mas(H(1,\cdot),\triangle;[a,b]) - \Mas(H(\cdot,a),\triangle;[0,1])  = 0.
		\end{multline}
		Using \eqref{define_maslov_pair} we have 
		\[
		\Mas(H(0,\cdot),\triangle;[a,b]) = \Mas(\Lambda_1, \Lambda_2;[a,b]), \quad \Mas(H(1,\cdot),\triangle;[a,b]) = \Mas(\widetilde{\Lambda}_1, \widetilde{\Lambda}_2;[a,b]).
		\]
		By assumption $\dim \left ( H(\cdot,a) \cap \triangle \right)=\dim \left ( H_1(\cdot,a) \cap H_2(\cdot,a) \right )$ and $\dim \left ( H(\cdot,b) \cap \triangle \right)=$\\ $\dim \left ( H_1(\cdot,b) \cap H_2(\cdot,b) \right )$ are constant, so by property (4) of \cref{prop:enjoys} the Maslov indices of the second and fourth terms in \eqref{eq:homotopy1} are zero. Equation \eqref{eq:stratum_homotopy} follows. 
	\end{proof}

	\subsection{The Maslov box}\label{subsec:symplectic_evals}

	We first discuss the regularity and Lagrangian property of the stable and unstable bundles, $\E^{s,u}(x,\la)$, defined in \eqref{bundles_defn} for $x\in\R$ and $\la\in\R$, that will be required for our analysis. We extend $\E^{s,u}$ to $x=\pm\infty$ by setting
	\begin{align}\label{defn_bundles_infty}
		 \E^u(-\infty,\la) \coloneqq \U(\la),\qquad \E^s(+\infty,\la) \coloneqq \Ss(\la), \quad
	\end{align}
	and
	\begin{align}
		 \E^{u}(+\infty,\la) \coloneqq \lim_{x\to+\infty} \E^{u}(x,\la),\qquad \E^{s}(-\infty,\la) \coloneqq \lim_{x\to-\infty} \E^{s}(x,\la). 
	\end{align}
	Thus by definition and \eqref{bundle_limits}, the maps $(x,\la)\mapsto\E^u(x,\la)$ and $(x,\la)\mapsto\E^s(x,\la)$ are continuous on $[-\infty,\infty)\times \R$ and $(-\infty,\infty]\times \R$ respectively. It is known \cite{AGJ90,sandstede02} that the maps $\la\mapsto \E^{s,u}(x,\la)$ are analytic on $\R$ for each fixed $x\in\R$. The map $x\mapsto \E^u(x,0)$ is also analytic on $\R$. To see this, we note that solutions to \eqref{1st_order_sys} inherit the analyticity of the right hand side of \eqref{1st_order_sys}, and that the eigenvectors forming a basis of $\U(0)$ are linearly independent (see \eqref{A_asymp_evals} with $\la=0$ assuming \eqref{assumptions}, \eqref{assumption2}). It follows that the orthogonal projection $\mathbb{P}_{\E^u(x,0)}(x) $ onto $\E^u(x,0)$, given by $\mathbb{P}_{\E^u(x,0)}(x)  =  \mathbf{Z}(x)(\mathbf{Z}(x)^\top\mathbf{Z}(x))^{-1}\mathbf{Z}(x)^\top$ where $\mathbf{Z}(x)$ is a frame for $ \E^u(x,0)$, is analytic in $x$.

	We remark here that the mapping 
	\begin{equation}\label{disc_top_shelf}
		\la\mapsto \E^u(+\infty,\la) = \lim_{x\to\infty} \E^u(x,\la)
	\end{equation}
	is discontinuous at eigenvalues $\la\in\spec(N)$. Indeed, if $\la\notin\spec(N)$, then $\E^u(+\infty,\la)  = \U(\la)$ (again as points on $\Gr_4(\R^8)$), while if $\la\in\spec(N)$ is an eigenvalue then $\E^u(+\infty,\la) \cap \Ss(\la) \neq \{0\}$, i.e. $\E^u(+\infty,\la) \in \cT(\Ss(\la) )$. Now since $\U(\la)\in\cT_0(\Ss(\la))$, and $\cT_0(\Ss(\la))$ is an open subset of $\cL(n)$ with boundary $\cT(\Ss(\la))$, it follows that $\U(\la)$ is bounded away from $\cT(\Ss(\la))$. For more details see the Appendix in \cite{HLS18}.
	
	\begin{rem}\label{rem:compactify}
		The Maslov index is defined for Lagrangian paths over compact intervals. Following \cite{HLS18} we will sometimes compactify $\R$ via the change of variables  
		\begin{equation}\label{}
			x=\ln\left (\f{1+\tau}{1-\tau}\right ), \qquad \tau\in[-1,1].
		\end{equation}
		Notationally we will use a hat to indicate such a change has been made, for example, 
		\begin{equation}\label{}
			\widehat{\E}^{s,u}(\tau,\cdot) \coloneqq \E^{s,u}\left (\ln\left (\f{1+\tau}{1-\tau}\right ),\cdot\right ), \quad \tau\in[-1,1].
		\end{equation}
		In this case, \eqref{defn_bundles_infty} implies that $\widehat{\E}^{u}(-1,\la) = \U(\la)$ and $\widehat{\E}^{s}(1,\la) = \Ss(\la)$.
	\end{rem}
	\begin{lemma}\label{lemma:lagrangian_subspaces}
		The spaces $\E^u(x;\la)$ and $\E^s(x;\la)$ are Lagrangian subspaces of $\R^8$ for all $x\in[-\infty, \infty]$ and $\la\in \R$.
	\end{lemma}
	\begin{proof}
		First, recall that $\dim \U(\la)=\dim \Ss(\la)=4$ (we showed in \eqref{A_asymp_evals} that $A_\infty(\la)$ is hyperbolic with four eigenvalues of positive real part and four of negative real part.) It follows from the continuity of $\E^u$ on $[-\infty,\infty)\times \R$ that $\dim\E^u(x,\la) =4$ for all $(x,\la)\in [-\infty,\infty)\times \R$. A similar argument shows $\dim\E^s(x,\la)=4$ for $(x,\la)\in (-\infty,\infty]\times \R$. 
		
		Next, for $x\in\R$, let $\mathbf{w}_1(x;\la), \mathbf{w}_2(x;\la)\in \E^u(x;\la)$. We have:
		\begin{align*}
			\w (\mathbf{w}_1(x;\la), \mathbf{w}_2(x;\la)) & = \langle J \mathbf{w}_1(x;\la), \mathbf{w}_2(x;\la) \rangle, \\
			&= \int_{-\infty}^x \de{}{s}  \langle J \mathbf{w}_1(s;\la), \mathbf{w}_2(s;\la) \rangle ds, \\
			&= \int_{-\infty}^x  \langle J A(s;\la) \mathbf{w}_1(s;\la), \mathbf{w}_2(s;\la) \rangle + \langle J \mathbf{w}_1(s;\la), A(s;\la)\mathbf{w}_2(s;\la) \rangle ds, \\
			&= \int_{-\infty}^x  \left \langle \left  (A(s;\la)^\top J+J A(s;\la)\right ) \mathbf{w}_1(s;\la), \mathbf{w}_2(s;\la) \right \rangle ds, \\
			& =0,
		\end{align*}
		where we used \eqref{infsymp}, i.e. that $A(x;\la)$ is infinitesimally symplectic. The proof for $\E^s(x;\la)$ is similar, but the integral is taken over $[x,\infty)$. We have shown that $\E^u$ and $\E^s$ are Lagrangian on $\R\times\R$. That this property extends to $x=\pm\infty$ follows from  the closedness of $\cL(n)$ as a submanifold of $\Gr_n(\R^{2n})$.
	\end{proof}
	
	We are now ready to give the homotopy argument that leads to the lower bound of \cref{thm:main_lower_bound}. We consider the Lagrangian pair
	\begin{equation}\label{eq:Lagpath}
		\Gamma \ni (x,\la) \mapsto ( \E^u(x,\la), \E^s(\ell,\la) )\in \cL(4)\times \cL(4),
	\end{equation}
	where we choose $\ell \gg 1$ large enough so that 
	\begin{equation}\label{eq:ell_large_enough}
		\U(\la)\cap \E^s(x,\la) = \{0\} \,\,\,\,\, \text{for all}\,\,\,\,\, x\geq \ell
	\end{equation}
	(see \cref{rem:why_not_stable_subspace}). Here $\Gamma = \Gamma_1 \cup \Gamma_2 \cup \Gamma_3 \cup \Gamma_4$, where the $\Gamma_i$ are the contours
	\begin{gather}\label{eq:Maslov_box_contours}
		\begin{aligned}
			&\Gamma_1: x\in [-\infty, \ell], \,\,\,\la = 0,  && \Gamma_3: x\in [-\infty, \ell],\,\,\, \la = \la_\infty, \\
			&\Gamma_2: x = \ell, \,\,\,\la \in [0,\la_\infty],  && \Gamma_4: x=-\infty,\,\,\, \la = \la \in [0,\la_\infty],
		\end{aligned}
	\end{gather}
	in the $\la x$-plane (see \cref{fig:box}). The set $\Gamma$ has been dubbed the \emph{Maslov box} \cite{HLS18,Corn19}, and the associated homotopy argument (outlined below) can be seen in as far back as the works of Bott \cite{Bott56}, Edwards \cite{Edwards64}, Arnol'd \cite{A67} and Duistermaat \cite{Duistermaat76}. Notice that along $\Gamma_1$ and $\Gamma_3$, the second entry $\E^s(\ell,\la)$ of the image of the map in \eqref{eq:Lagpath} is fixed. The Maslov index of \eqref{eq:Lagpath} along these pieces thus reduces to the Maslov index for a single path with respect to a fixed reference plane. Along $\Gamma_2$ and $\Gamma_4$, however, we have a genuine Lagrangian pair. 
	
	Crossings of \eqref{eq:Lagpath} are points $(x,\la)\in \Gamma$ such that 
	\[
	\E^u(x,\la)\cap\E^s(\ell,\la) \neq \{0\}.
	\]
	Recalling that $\la$ is an eigenvalue of $N$ if and only if $\E^u(x,\la)\cap \E^s(x,\la)\neq\{0\}$ for all $x\in\R$, it follows that the $\la$-values of the crossings along $\Gamma_2$ (where $x=\ell$) are exactly the eigenvalues of $N$. In particular, since $0\in\spec(N)$, there is a crossing at the top left corner of the Maslov box, $(x,\la)=(\ell,0)$. Moreover, this crossing is two-dimensional by \eqref{dims_equal} and the fact that $\ker(N) = \spn\{(0,\phi)^\top,(\phi_x,0)^\top\}$ (recall \cref{hypo:simplicity_assumption}). Denoting the solutions of \eqref{1st_order_sys} corresponding to $(0,\phi)^\top$ and $(\phi_x,0)^\top$ by
	\begin{equation}\label{eq:phi_varphi}
		\pmb{\phi}(x)\coloneqq\begin{pmatrix}
			0 \\ \phi''(x) + \sigma_2\phi(x) \\ 0 \\ -\phi(x)\\ 0 \\ -\phi'(x) \\ 0 \\ \phi'''(x) \\
		\end{pmatrix}, \qquad
		\pmb{\varphi}(x)\coloneqq \begin{pmatrix}
			\phi'''(x)+\sigma_2\phi'(x) \\ 0 \\ \phi'(x) \\ 0 \\ \phi''(x) \\ 0 \\ \phi''''(x) \\ 0 \\ 
		\end{pmatrix},
	\end{equation}
	(obtained from \eqref{subs} with $u=0, v=\phi,$ and $u=\phi', v=0,$ respectively), it follows that $\E^u(x;0)\cap \E^s(x;0) = \spn\{\pmb{\phi}(x), \pmb{\varphi}(x)\}$ for all $x\in \R$. Hence, evaluating at $x=\ell$ we see that at the corner crossing $(x,\la)=(\ell,0)$ we have 
	\begin{equation}\label{intersection_at_zero_phis}
		\E^u(\ell,0)\cap \E^s(\ell,0) = \spn\{\pmb{\phi}(\ell), \pmb{\varphi}(\ell)\}.
	\end{equation}
	\begin{rem}\label{rem:why_not_stable_subspace}
		That the path \eqref{disc_top_shelf} is discontinuous in $\la$ prohibits taking $\Gamma_2$ to be at $x=+\infty$. Taking $\Gamma_2$ to be at $x=\ell$ for $\ell$ large enough avoids this issue. Chen and Hu \cite{CH07} showed that by taking $\ell$ large enough so that \eqref{eq:ell_large_enough} holds, the Maslov index of \eqref{eq:Lagpath} along $\Gamma_1$ is independent of the choice of $\ell$. For more details, see \cite{CH07,Corn19}. 
	\end{rem}

	{ 
	Crossings along $\Gamma_1$, i.e. points $(x,\la)=(x_0,0)$ such that
	\begin{equation}\label{crossingsG1}
		\E^u(x_0,0)\cap\E^s(\ell,0) \neq \{0\},
	\end{equation}
	will play a distinguished role in our analysis. It will follow from the arguments in \cref{sec:L+L__counts} (see \cref{lemma:L+_half_left_full_left}) that for these crossings, it suffices to use $\Ss(0)$ in lieu of $\E^s(\ell,0)$ (provided $\ell$ is large enough) as the reference plane. A \emph{conjugate point} is then a value $x_0\in\R$ such that
	\begin{equation}\label{conj_pts}
		\E^u(x_0,0)\cap\Ss(0) \neq \{0\}.
	\end{equation}
	Recalling the decoupling of the eigenvalue equations \eqref{eq:N_evp} when $\la=0$, we similarly have that the first order system \eqref{1st_order_sys} decouples into two independent systems for the $u$ and $v$ variables when $\la=0$. As a result, for each $x\in\R$ we have
	\begin{equation}\label{direct_sums}
		\E^u(x,0) = \E_+^u(x,0) \oplus \E_-^u(x,0),
	\end{equation}
	in the sense that for any $\mathbf{w}\in\E^u(x,0) $ we have
	\begin{equation*}
		\mathbf{w} = \begin{pmatrix}
			u_1 \\ 0 \\ u_2 \\ 0 \\ u_3 \\ 0 \\ u_4 \\ 0  
		\end{pmatrix} + \begin{pmatrix}
			0 \\ v_1 \\ 0\\ v_2 \\ 0 \\ v_3 \\ 0  \\ v_4
		\end{pmatrix},
	\end{equation*}
	where $\mathbf{u}=(u_1, u_2, u_3, u_4)^\top \in \E_+^u(x,0)$ and $\mathbf{v}=(v_1, v_2, v_3, v_4)^\top \in \E_-^u(x,0)$. By the same token, $\Ss(0) =  \Ss_+(0)\oplus \Ss_-(0)$. It follows that 
	\begin{multline*}
		\{x_0\in \R: \E^u(x_0,0)\cap\Ss(0) \neq \{0\}\} = \\ \{x_0\in \R: \E^u_+(x_0,0)\cap\Ss_+(0) \neq \{0\}\}\cup \{x_0\in \R: \E^u_-(x_0,0)\cap\Ss_-(0) \neq \{0\}\},
	\end{multline*}
	and this leads to the classification of conjugate points given in \cref{define:conj_pointsL}.
	}

	\begin{figure}
		\centering
		\begin{picture}(150,175)(-20,-40)
			\put(-1,25){$0$}
			\put(-12,-18){$-\infty$}
			\put(10,-25){\vector(0,1){135}}
			\put(-5,35){\vector(1,0){150}}
			\linethickness{0.3mm}
			\put(10,-15){\line(0,1){100}}
			\put(10,85){\line(1,0){110}}
			\put(10,-15){\line(1,0){110}}
			\put(120,-15){\line(0,1){100}}
			\put(10,-15){\vector(0,1){50}}
			\put(10,85){\vector(1,0){55}}
			\put(120,-15){\vector(-1,0){55}}
			\put(120,85){\vector(0,-1){70}}
			\thinlines
			\put(100,55){\text{ $\Gamma_3$}}
			\put(20,55){\text{$\Gamma_1$}}
			\put(55,70){\text{$\Gamma_2$}}
			\put(55,-8){\text{$\Gamma_4$}}
			\put(150,30){$\lambda$}
			\put(8,120){$x$}
			\put(-2,83){$\ell$}
			\put(122,25){$\lambda_\infty$}
			\put(10,85){\circle*{5}}
			\put(10,65){\circle*{5}}
			
			\put(10,42){\circle*{5}}
			\put(10,0){\circle*{5}}
			\put(45,85){\circle*{5}}
			\put(80,85){\circle*{5}}
			\put(30,95){{\tiny \text{real eigenvalues of $N$}}}
			\put(30,-28){{\tiny \,\,\,\text{no crossings}}}  
			\put(-35,60){{{\tiny conjugate}}}
			\put(-30,52){{{\tiny points}}}
			\put(135,60){{{\tiny no}}}
			\put(126,52){{{\tiny crossings}}}
		\end{picture}
		\caption{Maslov box in the $\la x$-plane, with edges oriented in a clockwise fashion. The crossing at the top left corner $(\la,x)=(0,\ell)$ corresponds to the zero eigenvalue of $N$. It is natural to place $\la$ on the horizontal axis upon viewing it as a spectral parameter taking real values, thus lying on the real axis in the complex plane. }
		\label{fig:box}
	\end{figure}
	Since the solid rectangle $\cM=[-\infty,\ell] \times [0,\la_\infty]$ is contractible and the map \eqref{eq:Lagpath} with domain $\cM$ is continuous, the image of the boundary $\p \cM = \Gamma$ of $\cM$ in $\cL(4)\times \cL(4)$ is homotopic to a fixed point. From homotopy invariance (cf. \cref{prop:enjoys}), it follows that
	\begin{equation}\label{}
		\Mas( \E^u(\cdot, \cdot),\E^s(\cdot, \cdot);\Gamma)=0.
	\end{equation}
	By additivity under concatenation, we can decompose the left hand side into the contributions coming from the constituent sides of the Maslov box, i.e.
	\begin{multline}\label{mas_box_argument}
		\Mas(\E^u(\cdot,0),\E^s(\ell,0);[-\infty,\ell] ) + \Mas(\E^u(\ell,\cdot),\E^s(\ell,\cdot);[0,\la_\infty] )  \\  -\Mas(\E^u(\cdot,\la_\infty),\E^s(\ell,\la_\infty);[-\infty,\ell] )  -\Mas(\E^u(-\infty,\cdot),\E^s(\ell,\cdot);[0,\la_\infty] ) =0.
	\end{multline}
	Note we have included minus signs for the last two terms in order to be consistent with the clockwise orientation of the Maslov box (see \cref{fig:box}). We will show in \cref{sec:proof_lower_bound} that \eqref{eq:Lagpath} has no crossings along $\Gamma_3$ and $\Gamma_4$, and hence these last two Maslov indices are zero. Defining 
	\begin{equation}\label{mathfrakc}
		\mathfrak{c} \coloneqq \Mas(\E^u(\cdot,0),\E^s(\ell,0);[\ell-\e,\ell] ) + \Mas(\E^u(\ell,\cdot),\E^s(\ell,\cdot);[0,\e] ), \quad 0 < \e \ll 1,
	\end{equation}
	to be the contribution to the Maslov index of the corner crossing $(x,\la)=(\ell,0)$, it follows once more from additivity under concatenation that
	\begin{equation}\label{eqn_maslov}
		\Mas(\E^u(\cdot,0),\E^s(\ell,0);[-\infty,\ell-\e] ) + \mathfrak{c} +  \Mas(\E^u(\ell,\cdot),\E^s(\ell,\cdot);[\e,\la_\infty] )  =0.
	\end{equation}
	(We remark here that \cref{define:Maslov_GPP} is applicable to all terms on the left hand side, owing to the analyticity of the Lagrangian pairs appearing therein as per the discussion at the outset of the current subsection.) We will compute the first term of \eqref{eqn_maslov} by counting $L_+$ and $L_-$ conjugate points. By bounding $n_+(N)$ from below by the absolute value of the third term in \eqref{eqn_maslov}, computing $\mathfrak{c}$ and rearranging, we will arrive at the statement of \cref{thm:main_lower_bound}. Before doing so, we turn to the computation of the Morse indices of $L_+$ and $L_-$ via the Maslov index. 
	
	\section{Counting eigenvalues for $L_+$ and $L_-$ via conjugate points} \label{sec:L+L__counts}
	
	In this section we prove \cref{thm:Lpm_eval_counts}, which states that the Morse indices of $L_+$ and $L_-$ are equal to the respective number of conjugate points, counted with multiplicity, on $\R$. The proof uses a homotopy argument involving the Maslov box, similar to that described in \cref{subsec:symplectic_evals}. In order to set the argument up, we introduce the spatially dynamic first order systems and their associated stable and unstable bundles. In what follows, we use a subscript $+$ or $-$ to indicate that objects pertain to the eigenvalue problem for $L_+$ or $L_-$.
	
	The eigenvalue equation for $L_+$,
	\begin{equation}\label{evpL+}
		-u'''' - \sigma_2 u'' - \be u + 3  \phi^2 u = \la u, \qquad u\in H^4(\R),
	\end{equation}
	can be reduced to the following first order system via the $u$ substitutions in \eqref{subs},
	\begin{equation}\label{1st_order_sys_L+}
		\begin{pmatrix}
			u_1  \\ u_2 \\ u_3 \\ u_4 
		\end{pmatrix} ' = 
		\begin{pmatrix}
			0 &  0 & \sigma_2   & 1  \\
			0 & 0 &	1  & 0 \\
			1  & -\sigma_2   & 0 & 0 \\ 
			-\sigma_2   & \al(x)-\la  & 0 & 0
		\end{pmatrix}
		\begin{pmatrix}
			u_1  \\ u_2 \\ u_3 \\ u_4 
		\end{pmatrix},
	\end{equation}
	where $\al(x) =  3 \phi(x)^2 -\be +1$. Similar to \eqref{1st_order_sys_N_vector}, we write this system as
	\begin{equation}\label{}
		\mathbf{u}_x=A_+(x,\la)\mathbf{u},
	\end{equation}
	where $\mathbf{u}=(u_1,u_2,u_3,u_4)^\top$ and
	\begin{align*}\label{}
		A_+(x,\la) = \begin{pmatrix}
			0 & B_+ \\ C_+(x,\la) & 0
		\end{pmatrix}, \quad B_+ = \begin{pmatrix}
			\sigma_2 & 1 \\ 1 & 0
		\end{pmatrix}, \quad C_+(x,\la) = \begin{pmatrix}
			1 & -\sigma_2 \\ -\sigma_2 & \al(x) - \la
		\end{pmatrix}.
	\end{align*}
	Likewise, the eigenvalue equation for $L_-$,
	\begin{equation}\label{evpL-}
		- v'''' - \sigma_2 v'' - \be v +  \phi^2 v = \la v, \qquad v\in H^4(\R),
	\end{equation}
	can be reduced to the following first order system via the $v$ substitutions in \eqref{subs},
	\begin{equation}\label{1st_order_sys_L-}
		\begin{pmatrix}
			v_1  \\ v_2 \\ v_3 \\ v_4 
		\end{pmatrix} ' = 
		\begin{pmatrix}
			0 &  0 & -\sigma_2   & 1  \\
			0 & 0 &	1  & 0 \\
			-1  & -\sigma_2   & 0 & 0 \\ 
			-\sigma_2   & \eta(x)+\la  & 0 & 0
		\end{pmatrix}
		\begin{pmatrix}
			v_1  \\ v_2 \\ v_3 \\ v_4 
		\end{pmatrix}.
	\end{equation}
	where $\eta(x)  = - \phi(x)^2 +\be - 1$. We write this as
	\begin{equation}\label{}
		\mathbf{v}_x=A_-(x,\la)\mathbf{v},
	\end{equation}
	where $\mathbf{v}=(v_1,v_2,v_3,v_4)^\top$ and 
	\begin{equation*}\label{}
		A_-(x,\la) = \begin{pmatrix}
			0 & B_- \\ C_-(x,\la) & 0
		\end{pmatrix}, \quad B_- = \begin{pmatrix}
			-\sigma_2 & 1 \\ 1 & 0
		\end{pmatrix}, \quad C_-(x,\la) = \begin{pmatrix}
			-1 & -\sigma_2 \\ -\sigma_2 & \eta(x) + \la
		\end{pmatrix}.
	\end{equation*}
	The coefficient matrices $A_\pm(x,\la)$ are infinitesimally symplectic, satisfying equation \eqref{infsymp}. In order to be consistent with \eqref{1st_order_sys} at $\la=0$, we have used the same substitutions \eqref{subs} to reduce \eqref{evpL+} and \eqref{evpL-} to \eqref{1st_order_sys_L+} and \eqref{1st_order_sys_L-} respectively. Consequently, $\la$ appears with a different sign in \eqref{1st_order_sys_L+} and \eqref{1st_order_sys_L-}. This will be the reason for the difference in sign of the Maslov indices in \cref{lemma:L+_Mas_left_conjpoints}.

	{ The essential spectra of the operators $L_\pm$ is computed similarly to that of $N$ in \cref{sec:setup2}. Computing the $\la\in\R$ such that the asymptotic matrices $A_\pm(\la)$ defined in \eqref{asymptotic_systems} have a purely imaginary eigenvalue, and again considering the maximum value of $-k ^4  +\sigma_2  k ^2 - \beta $, we obtain 
		\begin{align}
			\esspec(L_\pm) &= \{ \la\in \R: \la = - k^4 + \sigma_2 k^2 -\be  \,\,\,\text{for some}\,\,\, k\in\R \}, \nonumber \\
			&= \begin{cases}
				(-\infty, -\beta] &\,{\sigma_2} =-1, \\
				(-\infty, -\beta+\f{1}{4}]  &\,{ \sigma_2} = 1.
			\end{cases}
			\label{essspecL}
		\end{align}
	Under \eqref{assumptions}, we therefore have $\esspec(L_\pm)\subset \R^-$. For all $\la$ lying to the right of the essential spectrum in \eqref{essspecL}, the asymptotic matrices $A_\pm(\la)$ have eigenvalues
	\[
		\spec(A_+(\la)) = \spec(A_-(\la)) = \Big\{ \pm\frac{\sqrt{-\sigma_2 \pm\sqrt{1-4 \beta -4 \lambda }}}{\sqrt{2}} \Big\}.
	\]}
	We denote the associated two-dimensional stable and unstable subspaces by $\Ss_\pm(\la)$ and $\U_\pm(\la)$ respectively. Reasoning as in \cref{subsec:symplectic_evals}, associated with each of the systems \eqref{1st_order_sys_L+} and \eqref{1st_order_sys_L-} are the stable and unstable bundles,
	\begin{align}\label{USbundlesL+-}
		\begin{split}
			\E_+^u(x,\la) &\coloneqq \{\xi\in\R^4: \xi=\mathbf{u}(x;\la), \,\,\mathbf{u} \text{ solves \eqref{1st_order_sys_L+} and $\mathbf{u}(x;\la)\to 0$ as $x\to -\infty$} \},  \\
			\E_+^s(x,\la) &\coloneqq \{\xi\in\R^4: \xi=\mathbf{u}(x;\la), \,\, \mathbf{u} \text{ solves \eqref{1st_order_sys_L+} and $\mathbf{u}(x;\la)\to 0$ as $x\to +\infty$} \}, \\
			\E_-^u(x,\la) &\coloneqq \{\xi\in\R^4: \xi=\mathbf{v}(x;\la),\,\, \mathbf{v} \text{ solves \eqref{1st_order_sys_L-} and $\mathbf{v}(x;\la)\to 0$ as $x\to -\infty$} \}, \\
			\E_-^s(x,\la) &\coloneqq \{\xi\in\R^4: \xi=\mathbf{v}(x;\la),\,\, \mathbf{v} \text{ solves \eqref{1st_order_sys_L-} and $\mathbf{v}(x;\la)\to 0$ as $x\to +\infty$} \}.
		\end{split}
	\end{align}
	When considered as points on the Grassmannian $\Gr_2(\R^4)$ of two dimensional subspaces of $\R^{4}$, these bundles converge to the asymptotic stable and unstable subspaces as follows,
		\begin{align*}
			\lim_{x\to-\infty} \E_+^u(x,\la) &= \U_+(\la), \qquad &\lim_{x\to+\infty} \E_+^s(x,\la) &= \Ss_+(\la),  \\
			\lim_{x\to-\infty} \E_-^u(x,\la) &= \U_-(\la), \qquad &\lim_{x\to+\infty} \E_-^s(x,\la) &= \Ss_-(\la).
		\end{align*}
	That $\E_+^u(x,\la), \E_-^u(x,\la), \E_+^s(x,\la), \E_-^s(x,\la)$ are Lagrangian subspaces of $\R^4$, with the maps $(x,\la)\mapsto \E_\pm^{u}(x,\la)$ and $(x,\la)\mapsto\E_\pm^s(x,\la)$ continuous on $[-\infty,\infty) \times \R$ and $[-\infty,\infty) \times \R$ respectively, and the maps $\la\mapsto \E_\pm^{u,s}(x,\la)$ and $x\mapsto \E_\pm^{u,s}(x,0)$ analytic on $\R$, follows from the same arguments as in \cref{subsec:symplectic_evals}. We omit the details.
	
	Let us focus momentarily on the $L_+$ problem and consider the path 
	\begin{equation}\label{L+path}
		\Gamma \ni (x,\la) \mapsto ( \E_+^u(x,\la), \E_+^s(\ell,\la) )\in \cL(2)\times \cL(2), 
	\end{equation}
	where $\Gamma$ is given in \eqref{eq:Maslov_box_contours} (see \cref{fig:box}). Crossings of \eqref{L+path} along $\Gamma_2$ now represent eigenvalues of $L_+$. Similar to \eqref{mas_box_argument} we have
	\begin{multline}\label{homotopyL+L-}
			\Mas(\E_+^u(\cdot,0),\E_+^s(\ell,0);[-\infty,\ell] ) + \Mas(\E_+^u(\ell,\cdot),\E_+^s(\ell,\cdot);[0,\la_\infty] )  \\  -\Mas(\E_+^u(\cdot,\la_\infty),\E_+^s(\ell,\la_\infty);[-\infty,\ell] )  -\Mas(\E_+^u(-\infty,\cdot),\E_+^s(\ell,\cdot);[0,\la_\infty] ) =0.
	\end{multline}
	To prove \eqref{eq:P=conj_points} in \cref{thm:Lpm_eval_counts}, we will show that \eqref{L+path} has no crossings on $\Gamma_3$ and $\Gamma_4$, and hence the last two terms in the right hand side of \eqref{homotopyL+L-} are zero. We also show that crossings along $\Gamma_2$ are positive. Along $\Gamma_1$, we show that we can interchange the reference plane  $\E_+^s(\ell,0)$ with $\Ss_+(0)$ over a modified interval, i.e.
		\begin{equation}\label{interchange}
		\Mas(\E_+^u(\cdot,0), \E_+^s(\ell,0); [-\infty,\ell]) = \Mas(\E_+^u(\cdot,0), \Ss_+(0); [-\infty,\infty]),
	\end{equation}
	 and that all crossings of the latter path are negative. By showing that the corner crossing $(x,\la)=(\ell,0)$ has two contributions (the arrival along $\Gamma_1$ and departure along $\Gamma_2$) to the Maslov index of \eqref{L+path}, each of which are zero, \eqref{eq:P=conj_points} will then follow. The proof for the $L_-$ problem will be similar: considering the path $\Gamma \ni (x,\la) \mapsto ( \E_-^u(x,\la), \E_-^s(\ell,\la) )\in \cL(2)\times \cL(2)$, crossings along $\Gamma_1$ are now positive, while crossings along $\Gamma_2$ are negative. The two contributions of the corner crossing will be nonzero but cancel each other out, and \eqref{eq:Q=conj_points} similarly follows.
	
	\subsection{Computing the Maslov index along $\Gamma_1$}\label{subsec:4.1}
	
	In order to compute the right hand side of \eqref{interchange}, we will need a real frame for $\Ss_\pm(0)$ with which we can compute crossing forms. To that end, the matrices $A_\pm(0)$ satisfy $\spec(A_+(0))=\spec(A_-(0)) = \{\pm \mu_1, \pm \mu_2\},$ where
	\begin{equation}\label{spatialevalsL+}
		\mu_1 = \frac{\sqrt{ -\sigma_2 - \sqrt{1-4 \beta }}}{\sqrt{2}}, \qquad \mu_2 = \frac{\sqrt{- \sigma_2 + \sqrt{1-4 \beta }}}{\sqrt{2}}.
	\end{equation}
	Under assumption \eqref{assumptions}, we have $\mu_2 = \bar{\mu}_1\in \C\backslash\R$ whenever $\beta> 1/4$ (for both $\sigma_2=\pm1$),  and $\mu_1,\mu_2\in \R$ when $\sigma_2=-1$ and  $0<\beta< 1/4$. {The additional assumption \eqref{assumption2} ensures that $\mu_1\neq \mu_2$ are distinct.} The eigenvectors corresponding to $-\mu_1$ and $-\mu_2$ are given by
	\begin{equation}
		\mathbf{p}_1 = \begin{pmatrix}
			\mu _2^2 \\  -1 \\ \mu_1 \\ \mu_1^3
		\end{pmatrix}, \,\,
		\mathbf{p}_2 = \begin{pmatrix}
			\mu_1^2 \\ -1 \\ \mu_2 \\ \mu_2^3 
		\end{pmatrix}, 
		\qquad \text{and} \qquad 
		\mathbf{m}_1 = \begin{pmatrix}
			\mu _2^2 \\ 1 \\ -\mu_1 \\ \mu_1^3
		\end{pmatrix}, \,\,
		\mathbf{m}_2 = \begin{pmatrix}
			\mu_1^2 \\ 1 \\ -\mu_2 \\ \mu_2^3 
		\end{pmatrix},
	\end{equation}
	so that
	\begin{equation}\label{kernels}
		\ker\left (  A_+(0) + \mu_i \right ) = \spn \{\mathbf{p}_i\}, \quad \ker\left (  A_-(0) + \mu_i \right ) = \spn \{\mathbf{m}_i\}, \qquad i=1,2.
	\end{equation} 
	Notice that the vectors $\mathbf{p}_i, \mathbf{m}_i$ for $i=1,2$ are complex-valued if $\beta > 1/4$, { and are linearly independent if $\mu_1\neq \mu_2$.} 
	We collect these vectors into the columns of two frames, which we denote with $2\times 2$ blocks $P_i, M_i$, $i=1,2,$ via
	\begin{equation}\label{frames1}
		\begin{pmatrix}
			P_1 \\ P_2
		\end{pmatrix} \coloneqq \begin{pmatrix}
			\mu _2^2  & \mu_1^2  \\ -1 & -1 \\ \mu_1 & \mu_2 \\ \mu_1^3 & \mu_2^3  
		\end{pmatrix}, \qquad \begin{pmatrix}
			M_1 \\ M_2
		\end{pmatrix} \coloneqq \begin{pmatrix}
			\mu _2^2  & \mu_1^2  \\ 1 & 1 \\ -\mu_1 & -\mu_2 \\ \mu_1^3 & \mu_2^3  
		\end{pmatrix}.
	\end{equation}
	All of the matrices $P_i, M_i$ are invertible under \eqref{assumptions} and \eqref{assumption2}. Right multiplying each frame in \eqref{frames1} by the inverse of its upper $2\times 2$ block yields the following \emph{real} frame for $\Ss_\pm(0)$,
	\begin{align}\label{S+frame}
		\mathbf{S}_\pm &= \begin{pmatrix}
			I \\ S_\pm
		\end{pmatrix}, \qquad S_\pm= \frac{1}{\sqrt{2 \sqrt{\beta }-\sigma_2 }} \left(
		\begin{array}{cccc}
			\mp 1   & \sigma_2 -\sqrt{\beta }\\
			\sigma_2 -\sqrt{\beta } & \pm(\sqrt{\beta } \sigma_2 +\beta -1)
		\end{array}
		\right),
	\end{align}
	where $S_+ = P_2P_1^{-1}$ and $S_- = M_2M_1^{-1}$.
	
	An important relation exists between $S_\pm$ and the blocks of the asymptotic matrix $A_\pm(0)$ that will be needed in our analysis. Define $C_\pm(x) \coloneqq C_\pm(x,0)$ and
	\begin{equation}\label{Chat}
		 \widehat{C}_+\coloneqq \lim_{x\to \pm\infty} C_+(x) = \begin{pmatrix}
				1 & -\sigma_2 \\ -\sigma_2 & 1-\be 
			\end{pmatrix},\qquad \widehat{C}_- \coloneqq \lim_{x\to\pm\infty} C_-(x) = \begin{pmatrix}
				-1 & -\sigma_2 \\ -\sigma_2 & \be -1 
			\end{pmatrix}.
	\end{equation} 
	Focusing on the $L_+$ problem, because the columns of the frame $(P_1, P_2)$ are eigenvectors of $A_+(0)$, we have
	\begin{equation}\label{}
		\begin{pmatrix}
			0 & B_+ \\ \widehat{C}_+ & 0 
		\end{pmatrix} \begin{pmatrix}
			P_1 \\ P_2
		\end{pmatrix} = \begin{pmatrix}
			P_1 \\ P_2
		\end{pmatrix}D_+ , \qquad D_+= \diag\{-\mu_1, -\mu_2\}.
	\end{equation}
	That is, $B_+P_2 = P_1 D_+$ and $\widehat{C}_+ P_1 = P_2 D_+$. It follows that 
	\begin{equation}\label{C=SPS+}
		\widehat{C}_+ = P_2 D_+P_1^{-1} = \left ( P_2 P_1^{-1} \right )\left ( P_1 D_+ P_2^{-1} \right ) \left ( P_2 P_1^{-1} \right ) = S_+ B_+ S_+.
	\end{equation}
	Similarly, it can be shown that
	\begin{equation*}%\label{C=SPS-}
		\widehat{C}_- = S_- B_- S_-.
	\end{equation*}
	We are ready to state our first intermediate result towards the proof of \cref{thm:Lpm_eval_counts}: monotonicity of the paths $x\mapsto (\E_\pm^u(x,0),\Ss_\pm(0))$. In what follows, $\langle \cdot, \cdot\rangle$ denotes the Euclidean dot product. Recall the definition of $p_c$ and $q_c$ from \eqref{pcqc}.
	\begin{lemma}\label{lemma:L+_Mas_left_conjpoints}
		Each crossing $x_0\in\R$ of the Lagrangian path $x\mapsto (\E_+^u(x,0),\Ss_+(0))$ is negative, thus
		\begin{equation}\label{eq49}
			\Mas(\E_+^u(\cdot,0), \Ss_+(0); [-\infty,\infty)) = -p_c.
		\end{equation}
		Similarly, each crossing $x_0\in\R$ of $x\mapsto \left (\E_-^u(x,0),\Ss_-(0)\right )$ is positive, so that
		\begin{equation}\label{}
			\Mas(\E_-^u(\cdot,0), \Ss_-(0); [-\infty,\infty)) = q_c.
		\end{equation}
	\end{lemma}
	\begin{rem}\label{rem:open_intervals}
		In the above lemma (and throughout), by having the domain of the Lagrangian paths $x \mapsto (\E_\pm^u(\cdot,0),\Ss_\pm(0))$ as $x\in[-\infty,\infty)$, we mean that $\tau\in[-1,1-\e]$ for the compactified path $\tau\mapsto\widehat{\E}^u_+(\tau,0)$ for some small $\e>0$ (see \cref{rem:compactify}). We emphasise that the final point of the path, $\tau=+1$ ($x=\infty$), which is always a conjugate point because $\E_\pm^u(+\infty,0)\in \cT_1(\Ss_\pm(0))$ on account of \cref{hypo:simplicity_assumption}, is excluded.  
	\end{rem}
	
	\begin{proof}
		We begin with the proof of \eqref{eq49}. Denote a frame for the unstable bundle $\E^u_+(x,0)$ by
		\begin{equation*}\label{}
			\mathbf{U}(x)=\begin{pmatrix}
				X(x) \\ Y(x)
			\end{pmatrix}, \qquad X(x),Y(x)\in\R^{2\times 2}.
		\end{equation*}
		In what follows $x_0\in\R$ is a conjugate point, i.e. $\E^u_+(x_0,0)\cap \Ss_+(0)\neq \{0\}$, and we denote 
		\[
		\mathbf{U}_0\coloneqq \mathbf{U}(x_0), \quad X_0\coloneqq X(x_0), \quad Y_0\coloneqq Y(x_0).
		\]
		We momentarily assume the following.
		\begin{hypo}\label{hypo:nonzero}
			At every crossing $x_0\in\R$, $\phi(x_0)\neq0$.
		\end{hypo} 
		We begin by computing the first order crossing form, i.e. \eqref{kth_crossingform0} with $x$ as the independent variable and $k=1$. To that end, recall from \eqref{Wk_properties} that
		\[
		W_1 = \E^u_+(x,0) \cap \Ss_+(0).
		\]
		Any $w_0\in W_1$ can therefore be written
		\begin{equation}\label{intersection}
			w_0 = \mathbf{U}_0 h_0 = \begin{pmatrix}
				X_0 \\ Y_0
			\end{pmatrix} h_0 = \begin{pmatrix}
				I \\ S_+
			\end{pmatrix}k_0 \in \E^u_+(x,0) \cap \Ss_+(0),
		\end{equation}
		for some $h_0,k_0\in \R^2$. To compute the first order form, as in \eqref{first-order_RSform}, we need to write down $\mathbf{U}'(x_0)$; since the columns of $\mathbf{U}(x)$ satisfy \eqref{1st_order_sys_L+}, we have (recall $C_+(x) \coloneqq C_+(x,0)$)
		\begin{equation}\label{XYdash}
			X'(x) = B_+ Y(x), \qquad Y'(x) = C_+(x) X(x).
		\end{equation}
		Recalling \eqref{Chat}, defining
		\begin{equation}\label{widetildeC}
			\widetilde{C}_+(x) \coloneqq C_+(x) - \widehat{C}_+ = \begin{pmatrix}
				0 & 0 \\ 0 & 3\phi(x)^2
			\end{pmatrix}
		\end{equation}
		 and using \eqref{C=SPS+}, we observe that 
		\begin{align}\label{CminusSBS}
			C_+(x) - S_+B_+S_+ =  \left ( \widehat{C}_+ + \widetilde{C}_+(x)  \right ) - S_+B_+S_+ = \widetilde{C}_+(x).
		\end{align}
		Letting $k_0=(a_0,b_0)^\top$, and using \eqref{intersection}, \eqref{XYdash} and \eqref{CminusSBS}, we compute \eqref{first-order_RSform}:
		\begin{align}
			\mathfrak{m}_{x_0}(\E^u_+(\cdot,0),\Ss_+(0))(w_0) &= \w\left (\mathbf{U}'(x_0)h_0, w_0\right )  = \left \langle  \begin{pmatrix}
				-Y'(x_0) \\ X'(x_0)
			\end{pmatrix} h_0, \begin{pmatrix}
				I \\ S_+ 
			\end{pmatrix}k_0 \right \rangle,  \nonumber\\
			&= - \left \langle\left ( Y'(x_0) - S_+ X'(x_0)\right )   h_0, k_0 \right \rangle, \nonumber\\
			&= 	-\left \langle  \left (  C_+(x_0) -S_+B_+S_+ \right )k_0,k_0  \right \rangle, \nonumber\\
			&= - \widetilde{C}_+(x_0) k_0 = -3 \phi(x_0)^2 b_0^2.   \label{FOF}
		\end{align}
		Under \cref{hypo:nonzero}, we conclude that if  there exist root functions $w$ such that $w_0 = \mathbf{S}_+k_0$ where $k_0=(a_0,b_0)^\top, b_0\neq 0$, then $n_-(\mathfrak{m}_{x_0})=1$. If, however, there exist root functions $w$ such that $b_0=0$, then the form $\mathfrak{m}_{x_0}$ is degenerate, and we need to compute higher order crossing forms. We split the analysis into two cases. In what follows, we suppress $x$-dependence of $C_+$ to simplify notation. 
		
		\emph{Case 1: $\dim \E^u_+(x_0,0) \cap \Ss_+(0)=1$.}
		
		Let $\mathbf{s}_1, \mathbf{s}_2$ denote the first and second columns of the frame $\mathbf{S}_+$ given in \eqref{S+frame}. If for some fixed $a_0\in\R$ and $b_0\neq0$ we have
		\begin{equation}\label{W1_not_spns1}
			W_1 = \E^u_+(x_0,0)\cap \Ss_+(0) = \spn \{a_0\mathbf{s}_1 + b_0 \mathbf{s}_2\},
		\end{equation}
		then from \eqref{FOF} we see that $n_-(\mathfrak{m}_{x_0})=1=\dim \E^u_+(x_0,0) \cap \Ss_+(0)$, and the crossing $x_0$ is negative. On the other hand, if $b_0=0$ so that
		\begin{equation}\label{assumption_W1}
			W_1 = \E^u_+(x_0,0) \cap \Ss_+(0) = \spn\{ \mathbf{s}_1\},
		\end{equation}
		then $\mathfrak{m}_{x_0} =0$, and we need to compute higher order crossing forms. For the rest of Case 1, we assume \eqref{assumption_W1}, so that the vector $k_0$ in \eqref{intersection} is given by $k_0=(a_0,0)^\top$ (where $a_0\neq0$ for a nontrivial intersection.)
		
		Since the crossing form is zero on this (one-dimensional) crossing, we know from \eqref{Wk1_kernel_kform} that 
		\begin{equation}\label{W2}
			W_2 = \ker \mathfrak{m}_{x_0} = W_1 = \spn\{ \mathbf{s}_1\}.
		\end{equation}
		According to \cref{lemma:compute_kth_form_in_practice}, assuming $h_0$ satisfies \eqref{intersection} with $b_0=0$, i.e.
		\begin{equation}\label{h0}
			\mathbf{U}_0h_0 = \begin{pmatrix}
				X_0 \\ Y_0
			\end{pmatrix}h_0 = \begin{pmatrix}
				I \\ S_+
			\end{pmatrix} k_0, \qquad k_0 = \begin{pmatrix}
				a_0 \\ 0
			\end{pmatrix},
		\end{equation}
		we now need to find a vector $h_1\in\R^2$ such that 
		\begin{align}\label{h1_original}
			{\mathbf{U}}'(x_0) h_0 + {\mathbf{U}}_0 h_1 = \begin{pmatrix}
				X'(x_0) \\ Y'(x_0)
			\end{pmatrix} h_0 + \begin{pmatrix}
				X_0 \\ Y_0
			\end{pmatrix} h_1 \in \Ss_+(0).
		\end{align}
		Since $(I, S_+)$ is a frame for $\Ss_+(0)$, this just means
		\begin{equation}\label{h1h0_eqn1}
			Y'(x_0) h_0 + Y_0 h_1= S_+\left ( X'(x_0)h_0 + X_0h_1\right ).
		\end{equation}
		Thus, using \eqref{h0}, \eqref{XYdash}, and that 
		\begin{equation}\label{CSBSk0=0}
			C_+k_0 = S_+B_+S_+k_0,
		\end{equation}
		(see \eqref{widetildeC}, \eqref{CminusSBS}), it follows from \eqref{h1h0_eqn1} that 
		\begin{equation}\label{h1ker}
			\left (Y_0 - S_+X_0 \right ) h_1 = - \left ( Y'(x_0) h_0 - S_+X'(x_0) \right )h_0 = - (C_+ - S_+B_+S_+)k_0=0.
		\end{equation}
		In order to write down $h_1$, we first note from \eqref{h0} that 
		\begin{equation}\label{h0ker}
			\left (Y_0 - S_+X_0\right )h_0=0.
		\end{equation}
		Next, we observe that $ \E^u_+(x_0,0) \cap \Ss_+(0) \cong \ker\left (  Y_0 - S_+X_0 \right )$, where the bijective correspondence is given by $\mathbf{U}_0h_0\leftrightarrow h_0$. To see this, note that any $\mathbf{U}_0 h_0 \in \E^u_+(x_0,0)$ is also in $\Ss_+(0)$ if and only if $Y_0h_0 = S_+X_0h_0$, i.e. $h_0 \in \ker (Y_0 - S_+X_0)$. Under our assumption \eqref{assumption_W1}, it follows that $\dim \ker\left (  Y_0 - S_+X_0 \right ) = 1$. Hence, \eqref{h1ker} and \eqref{h0ker} imply that
		\begin{equation}\label{h1}
			h_1 = \al h_0, \qquad \al \in \R.
		\end{equation}
		We are ready to compute the second order form using \cref{lemma:compute_kth_form_in_practice}. With $W_2$ given by \eqref{W2} and $h_0,h_1$ given by \eqref{h0} and \eqref{h1}, we compute:
		\begin{align}
			\mathfrak{m}^{(2)}_{x_0}(\E^u_+&(\cdot,0),\Ss_+(0))(w_0) = \w({\mathbf{U}}''(x_0) h_0 + 2\mathbf{U}'(x_0) h_1,w_0 ), \nonumber \nonumber  \\
			\begin{split}\nonumber
				&= \left \langle  \begin{pmatrix}
					-Y''(x_0) \\ X''(x_0) 
				\end{pmatrix} h_0, \begin{pmatrix}
					I \\ S_+\\ 
				\end{pmatrix} k_0\right \rangle   + 2 \left \langle   \begin{pmatrix}
					-Y'(x_0) \\ X'(x_0) 
				\end{pmatrix}h_1, \begin{pmatrix}
					I \\ S_+
				\end{pmatrix}k_0  \right \rangle, 
			\end{split} \\
			\begin{split}\label{SOF}
				&= - \left \langle  \left (
				Y''(x_0) - S_+ X''(x_0) \right ) h_0,  k_0\right \rangle   
				- 2 \left \langle   \left ( Y'(x_0) - S_+ X'(x_0) \right ) h_1, k_0  \right \rangle. 
			\end{split} 
		\end{align}
		Differentiating \eqref{XYdash} yields
		\begin{equation}\label{XY''x}
			X''(x) = B_+C_+X(x), \qquad Y''(x) = C_+' X(x) + C_+B_+Y(x).
		\end{equation}
		Using \eqref{h0}, it follows that 
		\begin{align}\label{X''Y''h_0}
			X''(x_0)h_0 = B_+C_+k_0, \qquad Y''(x_0)h_0= C_+'k_0 + C_+B_+ S_+ k_0 = C_+B_+ S_+ k_0,
		\end{align}
		where we used that 
		\begin{align}\label{Cp'}
			\begin{split}
				C'_+ k_0 = \widetilde{C}_+'k_0 = \begin{pmatrix}
					0 & 0 \\ 0 & 6\phi(x)\phi'(x)
				\end{pmatrix} \begin{pmatrix}
					a_0 \\ 0
				\end{pmatrix} = 0.
			\end{split}
		\end{align}
		Using \eqref{X''Y''h_0}, the symmetry of $C_+$ and $S_+B_+S_+$ and \eqref{CSBSk0=0}, the first term of \eqref{SOF} becomes
		\begin{align*}
			- \left \langle  \left (
			Y''(x_0) - S_+ X''(x_0) \right ) h_0,  k_0\right \rangle & = - \left \langle C_+B_+S_+ k_0,k_0 \right \rangle + \left \langle S_+B_+C_+k_0, k_0\right \rangle, \\
			&=  - \left \langle S_+B_+S_+B_+S_+ k_0, k_0 \right \rangle + \left \langle S_+B_+S_+B_+S_+k_0, k_0\right \rangle  = 0.
		\end{align*}
		For the second term of \eqref{SOF}, using \eqref{XYdash}, the symmetry of $C_+$ and $S_+B_+S_+$, \eqref{CSBSk0=0} and \eqref{h1ker}, we obtain
		\begin{align*}
			- 2 \left \langle   \left ( Y'(x_0) - S_+ X'(x_0) \right ) h_1, k_0  \right \rangle & = - 2 \left \langle C_+X_0 h_1, k_0 \right  \rangle + 2 \left \langle S_+B_+Y_0 h_1, k_0\right \rangle, \\
			&= 2 \left \langle S_+B_+\left (Y_0 - S_+X_0 \right )h_1, k_0 \right \rangle =0. 
		\end{align*}
		Hence both terms in \eqref{SOF}, are zero, that is, 
		\[ 
		\mathfrak{m}^{(2)}_{x_0}(\E^u_+(\cdot,0),\Ss_+(0)) =0,
		\] 
		and we need to compute the third order crossing form.
		
		The last calculation shows that $W_3  = \ker \mathfrak{m}^{(2)}_{x_0} = W_2 = W_1$. As per \cref{lemma:compute_kth_form_in_practice}, assuming $h_0$ and $h_1$ are given by \eqref{h0} and \eqref{h1}, we need to find a vector $h_2 \in \R^2$ such that $\mathbf{U}''(x_0) h_0 + 2\mathbf{U}'(x_0) h_1 + \mathbf{U}(x_0) h_2 \in \Ss_+(0)$, that is,
		\begin{equation}\label{h2_original}
			Y''(x_0) h_0 + 2 Y'(x_0) h_1 + Y_0h_2 = S_+ \Big ( X''(x_0)h_0 + 2 X'(x_0)h_1 + X_0 h_2\Big ).
		\end{equation}
		Rearranging, this becomes
		\begin{equation}
			\left (Y_0 - S_+ X_0\right )h_2 =  - \left  (Y''(x_0) - S_+ X''(x_0) \right )h_0 - 2\left ( Y'(x_0)  - S_+X'(x_0) \right )h_1.
		\end{equation}
		Using \eqref{X''Y''h_0}, \eqref{XYdash}, \eqref{h0} and \eqref{CSBSk0=0} in the previous equation yields
		\begin{align}
			\left (Y_0 - S_+X_0 \right ) h_2 &= - (C_+-S_+B_+S_+) B_+S_+ k_0  +2 (S_+B_+ Y_0 - C_+X_0) h_1. \label{h2good}
		\end{align}
		With $w_0\in W_3 = W_2 = W_1$, and $h_0, h_1$ and $h_2$ given by \eqref{h0}, \eqref{h1} and \eqref{h2good}, we are ready to compute the third order form. Using \cref{lemma:compute_kth_form_in_practice}, similar to \eqref{SOF} we arrive at
		\begin{align}
			\mathfrak{m}_{x_0}^{(3)}(\E^u_+(\cdot,0)&,\Ss_+(0))(w_0) = \w({\mathbf{U}}'''(x_0) h_0 + 3\mathbf{U}''(x_0) h_1+3\mathbf{U}'(x_0) h_2,w_0), \nonumber \\
			\begin{split}\label{3rdorderform}
				&= - \left \langle 
				Y'''(x_0) - S_+X'''(x_0) h_0, k_0\right \rangle 
				- 3 \left \langle \left ( Y''(x_0) - S_+ X''(x_0) \right )h_1, k_0  \right \rangle \\
				&\qquad \qquad \qquad \qquad -3  \left \langle \left ( Y'(x_0) - S_+ X'(x_0) \right )h_2, k_0  \right \rangle.
			\end{split} 
		\end{align}
		Differentiating \eqref{XY''x} yields
		\begin{align}
			\begin{split}\label{XY'''x}
				X'''(x) &= B_+C'_+ X(x) + B_+C_+B_+Y(x), \\
				Y'''(x) &= C''_+X(x) + 2 C_+' B_+Y(x) + C_+B_+C_+ X(x).
			\end{split}
		\end{align}
		For the first term of \eqref{3rdorderform}, using \eqref{XY'''x} at $x_0$, \eqref{h0} and \eqref{CSBSk0=0}, as well as that $C_+''k_0 = C_+'k_0=0$, we have
		\begin{align}
			- \left \langle 
			Y'''(x_0) - S_+X'''(x_0) h_0, k_0\right \rangle  & = - \left \langle C_+'' k_0 + 2 C_+'B_+ S_+ k_0 + C_+B_+C_+ k_0, k_0  \right \rangle  \nonumber\\ 
			& \qquad  \qquad + \left \langle S_+B_+C_+B_+S_+k_0,  k_0 \right \rangle, \nonumber \\
			&=  \left \langle S_+B_+ \left (C_+-S_+B_+S_+\right )B_+S_+k_0, k_0 \right \rangle. \label{444}
		\end{align}
		For the second term of \eqref{3rdorderform}, combining \eqref{XY''x} at $x_0$, \eqref{CSBSk0=0} and $C_+'k_0=0$, we find that
		\begin{align}\label{445}
			\begin{split}
				- 3 \left \langle \left ( Y''(x_0) - S_+ X''(x_0) \right )h_1, k_0  \right \rangle & = - 3 \left \langle C_+B_+ Y_0 h_1, k_0  \right \rangle + 3 \left \langle  S_+B_+C_+ X_0 h_1, k_0 \right \rangle,   \\
				&= 3 \big \langle S_+B_+  \left ( C_+X_0 - S_+ B_+Y_0 \right ) h_1,k_0  \big\rangle,  \\
				&= 3 \big \langle S_+B_+   \left ( C_+ - S_+B_+S_+\right )X_0h_1, k_0 \big \rangle  \\
				& \qquad \qquad + 3 \big \langle S_+B_+S_+B_+ \left ( S_+ X_0- Y_0 \right ) h_1,k_0  \big\rangle, \\
				&=0.
			\end{split}
		\end{align}
		In the last line we used that $h_1 \in \ker \left ( Y_0 - S_+ X_0\right)$, as well as that $h_1 = \al h_0$ for some $\al\in\R$ and hence $X_0 h_1 = \al X_0h_0 = \al k_0$, so that
		\begin{equation}
			\left ( C_+ - S_+B_+S_+\right )X_0h_1 = \al \left ( C_+ - S_+B_+S_+\right )k_0 = 0.
		\end{equation}
		For the last term of \eqref{3rdorderform}, using \eqref{XYdash} and \eqref{h2good}, we find that 
		\begin{align}
			\begin{split} \label{lastterm3OF}
				-3  \left \langle \left ( Y'(x_0) - S_+ X'(x_0) \right )h_2, k_0  \right \rangle & = - 3 \left \langle C_+X_0 h_2, k_0  \right \rangle + 3 \left \langle S_+B_+Y_0 h_2, k_0  \right \rangle,   \\
				&= 3 \left \langle S_+B_+ \left ( Y_0 - S_+X_0\right ) h_2, k_0 \right \rangle, \\
				&= -3 \left \langle S_+B_+ \left ( C_+-S_+B_+S_+\right ) B_+S_+k_0, k_0 \right \rangle  \\
				& \qquad  - 6 \langle S_+B_+ \left ( C_+X_0 - S_+B_+Y_0 \right ) h_1, k_0\rangle.
			\end{split}
		\end{align}
		The second term of the last line is minus two times the quantity in \eqref{445}, and is therefore zero. Combining the first term with \eqref{444}, we arrive at
		\begin{align*}
			\mathfrak{m}_{x_0}^{(3)}(\E^u_+(\cdot,0)&,\Ss_+(0))(w_0) = -2 \left \langle S_+B_+ \left ( C_+-S_+B_+S_+\right ) B_+S_+k_0, k_0 \right \rangle =  -\f{6a_0^2 \phi(x_0)^2}{2\sqrt{\beta}-\sigma_2}.
		\end{align*}
		It follows from our assumptions \eqref{assumptions}, \eqref{assumption2} that $2\sqrt{\be}-\sigma_2 > 0$. Hence, under \cref{hypo:nonzero} and with $a_0\neq0$, we have $n_-(\mathfrak{m}_{x_0}^{(3)})=1=\dim \E^u_+(x_0,0) \cap \Ss_+(0)$, and the crossing is negative.
		
		\emph{Case 2: $\dim \E^u_+(x_0,0) \cap \Ss_+(0)=2$.}
		
		In this case, $W_1 = \E^u_+(x,0) = \Ss_+(0)$, and hence $\mathbf{U}_0$ and $\mathbf{S}_+$ are frames for the same Lagrangian plane. Therefore $\mathbf{U}_0 = \mathbf{S}_+M$ for some $2\times 2$ invertible matrix $M$, so that $X_0 = M$ and 
		\begin{equation}\label{YSX}
			Y_0 =S_+X_0.
		\end{equation} 
		Evaluating the first order form on $W_1 = \E^u_+(x,0) = \Ss_+(0)$, we already saw that $n_-(\mathfrak{m}_{x_0})=1$. Since $\ker\mathfrak{m}_{x_0} \neq \{0\}$, we again compute higher order crossing forms; the proof is similar to Case 1 with the following changes. For the second order form, we have
		\begin{equation}
			W_2 = \ker \mathfrak{m}_{x_0} = \spn\{ \mathbf{s}_1\}.
		\end{equation}
		Since $Y_0 =S_+X_0$, equation \eqref{h1ker} indicates that any $h_1\in\R^2$ will satisfy \eqref{h1_original}. With $h_1$ free and $h_0$ given by \eqref{h0} (where $b_0=0, a_0\neq0$), the computation of the second order form is unchanged, owing to \eqref{YSX}. For the third order form, we'll have $W_3 = W_2$, while using \eqref{YSX} in \eqref{h2good} shows that 
		\begin{equation}\label{h1dim2good}
			(S_+B_+ Y_0 - C_+X_0) h_1 = \f{1}{2} (C_+-S_+B_+S_+) B_+S_+ k_0.
		\end{equation}
		We conclude that \emph{any} $h_2\in\R^2$ will satisfy \eqref{h2_original}, \emph{provided} $h_1$ satisfies \eqref{h1dim2good}. With such $h_0, h_1, h_2$, the computation of the third order form is similar; the first term of \eqref{3rdorderform} again gives \eqref{444}, the last term of \eqref{3rdorderform} vanishes due to \eqref{YSX} (see the second line of \eqref{lastterm3OF}), while for the second term of \eqref{3rdorderform} we now have, using \eqref{h1dim2good},
		\begin{align}
			- 3 \left \langle \left ( Y''(x_0) - S_+ X''(x_0) \right )h_1, k_0  \right \rangle  & =  3 \big \langle S_+B_+  \left ( C_+X_0 - S_+ B_+Y_0 \right ) h_1,k_0  \big\rangle,  \nonumber\\
			&= -\f{3}{2} \big \langle S_+B_+  \left ( C_+ - S_+ B_+S_+ \right ) B_+S_+k_0,k_0  \big\rangle.  \label{451}
		\end{align}
		Adding \eqref{444} and \eqref{451} together, we find 
		\begin{align*}
			\mathfrak{m}_{x_0}^{(3)}(\E^u_+(\cdot,0),\Ss_+(0))(w_0) &= -\f{1}{2} \left \langle  S_+ B_+ \left ( C_+ - S_+B_+S_+ \right ) B_+S_+ k_0,k_0 \right \rangle = -\f{3\phi(x_0)^2a_0^2}{2(2\sqrt{\beta} - \sigma_2)}.
		\end{align*}
		Under \eqref{assumptions}, \eqref{assumption2}, \cref{hypo:nonzero} and with $a_0\neq0$, we have $n_-(\mathfrak{m}_{x_0}^{(3)})=1$. Hence, $n_-(\mathfrak{m}_{x_0}^{(1)}) + n_-(\mathfrak{m}_{x_0}^{(3)}) = 2 = \dim \E^u_+(x_0,0) \cap \Ss_+(0)$, and the crossing is negative. 
		
		The proof for the $L_-$ problem is similar, the main difference being the change in sign of the crossing forms. In particular, the same arguments as those above show that the first order crossing form is given by
		\begin{equation*}
			\mathfrak{m}_{x_0}(\E^u_-(\cdot,0),\Ss_-(0))(w_0) = - \left \langle (C_- - S_-B_-S_-) k_0, k_0\right \rangle = \phi(x_0)^2b_0^2,
		\end{equation*}
		while again the second order form $\mathfrak{m}_{x_0}^{(2)}(\E^u_-(\cdot,0),\Ss_-(0))(w_0) = 0$. In the case of a one-dimensional intersection the third order form is now
		\begin{equation*}
			\mathfrak{m}_{x_0}^{(3)}(\E^u_-(\cdot,0),\Ss_-(0))(w_0) = -2 \left \langle S_-B_- \left ( C_--S_-B_-S_-\right ) B_-S_-k_0, k_0 \right \rangle =  \f{2a_0^2 \phi(x_0)^2}{2\sqrt{\beta}-\sigma_2},
		\end{equation*}
		and in the two-dimensional case, 
		\begin{equation*}
			\mathfrak{m}_{x_0}^{(3)}(\E^u_-(\cdot,0),\Ss_-(0))(w_0) = -\f{1}{2} \left \langle  S_+ B_+ \left ( C_+ - S_+B_+S_+ \right ) B_+S_+ k_0,k_0 \right \rangle = \f{\phi(x_0)^2a_0^2}{2(2\sqrt{\beta} - \sigma_2)}.
		\end{equation*}
		We omit the details. Thus, in all cases we have $n_+(\mathfrak{m}_{x_0}) + n_+(\mathfrak{m}_{x_0}^{(3)}) = \dim \E^u_-(x_0,0) \cap \Ss_-(0)$, and the crossings are positive. 
		
		If \cref{hypo:nonzero} fails, i.e. a crossing $x_0$ is such that $\phi(x_0)=0$, then the first order crossing form is identically zero, and in Cases 1 and 2 described above, the third order form is also identically zero. Thus, we need to compute higher order crossing forms. The complete proof of monotonicity in this case is cumbersome and is left to the appendix; see \cref{sec:appendixA}.
	\end{proof}
	
	\begin{rem}\label{rem:general_NLS_powerlaw}
		\Cref{lemma:L+_Mas_left_conjpoints} will also hold in the case of any even-integer power-law fourth-order NLS equation, i.e. \eqref{general_NLS_powerlaw} for any $p\in\N$. In this case, $L_\pm$ are given by 
		\begin{equation}\label{Lpm_powerlaw}
			L_- = - \p_x^4 - \sigma_2 \p_x^2 - \be + \phi^{2p}, \qquad 
			L_+ = - \p_x^4 - \sigma_2 \p_x^2 - \be + (2p+1) \phi^{2p},
		\end{equation}
		and the crossing forms $\mathfrak{m}_{x_0}^{(k)}$ for $k=1,2,3$ will be the same as those in the proof of \cref{thm:Lpm_eval_counts}, but scaled by a positive constant, and with $\phi(x_0)^2$ replaced by $\phi(x_0)^{2p}$. The signs are therefore preserved.
	\end{rem}

	Our next task is to show \eqref{interchange}. 
	
	\begin{lemma}\label{lemma:stratum_homotopic_left}
		For $\ell>0$ large enough, we have
		\begin{equation}\label{eq:Mas_closed_pairs}
			\Mas(\E_+^u(\cdot,0), \E_+^s(\ell,0); [-\infty,\ell]) = \Mas(\E_+^u(\cdot,0), \Ss_+(0); [-\infty,\infty]).
		\end{equation}
		A similar statement holds for the $L_-$ problem.
	\end{lemma}
	\begin{proof}
		We show that the Lagrangian pairs in the left and right hand sides of \eqref{eq:Mas_closed_pairs} are stratum homotopic. To do so, it will be convenient to compactify $\R$ via the change of variables in \cref{rem:compactify}. Thus, defining 
		\begin{equation}\label{}
			\widehat{\E}^{s,u}_\pm(\tau, 0) \coloneqq \E^{s,u}_\pm\left (\ln\left (\f{1+\tau}{1-\tau}\right ), 0\right ),
		\end{equation}
		\eqref{eq:Mas_closed_pairs} is equivalent to 
		\begin{equation}\label{eq:equiv_Mas_pairs}
			\Mas(\widehat{\E}_+^u(\cdot,0),  \widehat{\E}_+^s(\tau_\ell,0); [-1,\tau_\ell]) = \Mas(\widehat{\E}_+^u(\cdot,0), \widehat{\E}_+^s(1,0); [-1,1]),
		\end{equation}
		where $\ell= \ln((1+\tau_\ell)/(1-\tau_\ell))$, i.e. $\tau_\ell= (e^\ell-1)/(e^\ell+1)$, and we used that $\widehat{\E}^s_+(1,0) = \E^s_+(+\infty,0)\coloneqq \Ss_+(0)$ and Rescaling further, we can map $[-1,1]$ to $[-1,\tau_\ell]$ via
		\[
		g(\tau) = \left (\f{1+\tau_\ell}{2}\right ) \tau + \left (\f{\tau_\ell-1}{2} \right ),
		\]
		where $g(-1)=-1$ and $g(1) = \tau_\ell$. This allows us to write both Lagrangian paths in \eqref{eq:equiv_Mas_pairs} over $[-1,1]$, i.e.
		\begin{equation}\label{eq:equiv_final_pairs}
			\Mas(\widehat{\E}_+^u(g(\cdot),0),  \widehat{\E}_+^s(\tau_\ell,0); [-1,1]) = \Mas(\widehat{\E}_+^u(\cdot,0), \widehat{\E}_+^s(1,0); [-1,1]). 
		\end{equation}
		To prove \eqref{eq:equiv_final_pairs}, we use  \cref{prop:homotopy_invariance_pairs} by explicitly constructing a stratum homotopy between the Lagrangian pairs appearing in \eqref{eq:equiv_final_pairs}. We set
		\begin{align}
			H_1(s,\tau) \coloneqq \widehat{\E}_+^u(\tau+(g(\tau)-\tau)s, 0), \qquad 
			H_2(s,\tau) \coloneqq \widehat{\E}_+^s(1+(\tau_\ell-1)s,0),
		\end{align}
		noting that $H_2$ is independent of $\tau$, and that both mappings $(s,\tau) \mapsto H_{1,2}(s,\tau)$ are continuous on $[0,1]\times [-1,1]$. In addition,
		\begin{align*}
			H_1(s,-1) = \widehat{\E}^u_+(-1,0) = \U_+(0), \qquad 
			H_2(s,-1) = \widehat{\E}^s_+(1+(\tau_\ell-1)s,0),
		\end{align*}  
		where we used that $g(-1)=-1$. Since $\U_+(0)\cap \E^s_+(x,0) = \{0\}$ for all $x\geq \ell$ (see \eqref{eq:ell_large_enough}) and $\U_+(0)\cap \Ss_+(0)=\{0\}$, we have $\U_+(0) \cap \widehat{\E}^s_+(\tau,0) = \{0\}$ for all $\tau \in [\tau_\ell,1]$. Hence 
		\[
		H_1(s,-1) \cap H_2(s,-1) = \{0\}
		\]
		for all $s\in[0,1]$. Furthermore, 
		\begin{align*}\label{}
			H_1(s,1) = \widehat{\E}^u_+(1+(\tau_\ell - 1)s,0),  \qquad H_2(s,1) = \widehat{\E}^s_+(1+(\tau_\ell - 1)s,0),
		\end{align*}
		and therefore
		\[
		\dim H_1(s,1) \cap H_2(s,1)=1
		\]
		for all $s\in[0,1]$ by \cref{hypo:simplicity_assumption}. Equation \eqref{eq:equiv_final_pairs} (and thus \eqref{eq:Mas_closed_pairs}) now follows from \cref{prop:homotopy_invariance_pairs}.
	\end{proof}
	
	Finally, we show that the crossings occurring at the final points of each of the paths in \eqref{interchange} (guaranteed by \cref{hypo:nonzero}) have the same contribution to their respective Maslov indices. Hence, we can exclude the final point of each path in \eqref{interchange}. This will complete the computation of the Maslov index along $\Gamma_1$, i.e. the first term of \eqref{homotopyL+L-}. We note that some care is needed when dealing with the final crossing of the pair $x\mapsto\left (\E_+^u(x,0),  \Ss_+(0)\right )$, which is obtained in the limit as $x\to+\infty$. \Cref{lemma:L+_Mas_left_conjpoints} therefore does not apply, since the root functions used in the crossing form calculations either blow up to infinity or decay to zero asymptotically.

	\begin{lemma}\label{lemma:L+_half_left_full_left}
		For $\e>0$ small enough and $\ell>0$ large enough, we have
		\begin{equation}\label{l411}
			\Mas(\E_+^u(\cdot,0), \E_+^s(\ell,0); [-\infty,\ell-\e]) = \Mas(\E_+^u(\cdot,0), \Ss_+(0); [-\infty,\infty)).
		\end{equation} 
		A similar statement holds for the $L_-$ problem.
	\end{lemma}
	\begin{proof}
		By additivity under concatenation (cf. \cref{prop:enjoys}), we can write \eqref{eq:equiv_final_pairs} as
		\begin{align}\label{eq:final_left_pairs}
			\begin{split}
				\Mas(\widehat{\E}_+^u(g(\cdot),0),  \widehat{\E}_+^s(\tau_\ell,0); [-1,1-\e])+\Mas(\widehat{\E}_+^u(g(\cdot),0),  \widehat{\E}_+^s(\tau_\ell,0); [1-\e,1]) \\ = \Mas(\widehat{\E}_+^u(\cdot,0), \widehat{\E}_+^s(1,0); [-1,1-\e]) + \Mas(\widehat{\E}_+^u(\cdot,0), \widehat{\E}_+^s(1,0); [1-\e,1])
			\end{split}
		\end{align}
		for $\e>0$ small. \Cref{hypo:simplicity_assumption} implies that $\widehat{\E}_+^u(\tau,0) \in \cT_1(\widehat{\E}_+^s(\tau,0))$ for $\tau=\tau_\ell, 1$, i.e. that there exists a (one-dimensional) crossing at the final point of each of the paths 
		\begin{equation}\label{eq:2_paths}
			\tau \mapsto \left (\widehat{\E}_+^u(g(\tau),0),\widehat{\E}_+^s(\tau_\ell,0) \right ), \quad \tau \mapsto \left (\widehat{\E}_+^u(\tau,0),\Ss_+(0)\right ), \quad \tau\in[-1,1].
		\end{equation}
		Both of these crossings are isolated. Indeed, undoing the scaling by $g$ one sees that the first path $\tau\mapsto \widehat\E_+^u(\tau,0)$ is analytic at $\tau_\ell<1$. On the other hand, isolation of the final crossing of the second path in \eqref{eq:2_paths} follows from analyticity of $\tau\mapsto \widehat\E_+^u(\tau,0)$ for $\tau\in(-1,1)$, monotonicity of $\tau\mapsto \widehat\E_+^u(\tau,0)$ with respect to $\cT(\Ss_+(0))$, and the fact that $\widehat\E_+^u(\tau,0)$ approaches $\Ss_+(0)$ as $\tau\to 1^-$ (by assumption). Hence we can choose $\e>0$ small enough so that $\tau=1$ is the \emph{only} crossing in the interval $[1-\e,1]$ for the paths in \eqref{eq:2_paths}. With this choice, we now claim that 
		\begin{equation}\label{eq:final_conj_pairs}
			\Mas(\widehat{\E}_+^u(g(\cdot),0),  \widehat{\E}_+^s(\tau_\ell,0); [1-\e,1]) = \Mas(\widehat{\E}_+^u(\cdot,0), \widehat{\E}_+^s(1,0); [1-\e,1]) =0,
		\end{equation} 
		i.e. that the conjugate points occurring at the final points of each of the paths in  \eqref{eq:2_paths} do not contribute to their respective Maslov indices. Assuming the claim, by \eqref{eq:final_left_pairs} we have
		\[
		\Mas(\widehat{\E}_+^u(g(\cdot),0),  \widehat{\E}_+^s(\tau_\ell,0); [-1,1-\e]) = \Mas(\widehat{\E}_+^u(\cdot,0), \widehat{\E}_+^s(1,0); [-1,1-\e]).
		\]
		Recalling \cref{rem:open_intervals}, this is exactly \eqref{l411} (for a different but still arbitrarily small $\e$).
		
		It remains to prove \eqref{eq:final_conj_pairs}. To that end, note that the paths $\tau \mapsto \widehat{\E}_+^u(g(\tau),0)$ and $\tau \mapsto \widehat{\E}_+^u(\tau,0)$ are arbitrarily close to one another: for any given $\delta>0$ small, we can choose $\ell=\ln( (2-\delta)/\delta )$ so that $\tau_\ell = 1-\delta$, in which case
		\[
		|g(\tau) - \tau| = \left (\f{1-\tau_\ell}{2} \right )(\tau+1) \leq \delta,
		\]
		uniformly for $\tau\in[-1,1]$. In addition, for large enough $\ell$ the trains $\cT(\widehat{\E}_+^s(\tau_\ell,0) )$ and $\cT(\Ss_+(0))$ are arbitrarily small perturbations of one another. Since the final crossings of \eqref{eq:2_paths} are both one-dimensional by \cref{hypo:simplicity_assumption}, we conclude that the curves $\tau \mapsto \widehat{\E}_+^u(g(\tau),0)$ and $\tau \mapsto \widehat{\E}_+^u(\tau,0)$ approach $\cT_1(\widehat{\E}_+^s(\tau_\ell,0) )$ and $\cT_1(\Ss_+(0))$, respectively, from the same direction as $\tau\to1^-$. This proves the first equality in \eqref{eq:final_conj_pairs}. It follows from \cref{lemma:L+_Mas_left_conjpoints} and \cref{rem:monotonicity_geometrically} that the crossing at $\tau_\ell<1$ of the path $\tau \mapsto \left (\widehat{\E}_+^u(\tau,0),\widehat{\E}_+^s(\tau_\ell,0) \right )$ is negative, i.e. the crossing at $\tau=1$ of the first path in \eqref{eq:2_paths} is negative. In line with \cref{define:Maslov_GPP}, if the final crossing is one-dimensional and transverse, its contribution to the Maslov index is $+1$ if the path arrives at the train in the positive direction, and zero otherwise. Hence $\Mas(\widehat{\E}_+^u(g(\cdot),0),  \widehat{\E}_+^s(\tau_\ell,0); [1-\e,1])= 0$, and \eqref{eq:final_conj_pairs} follows. The proof for the $L_-$ problem is similar.
	\end{proof}

\subsection{Computing the Maslov index along $\Gamma_2$} In the following we prove monotonicity of the paths $\la \mapsto (\E_\pm^u(\ell,\la), \E_\pm^s(\ell,\la))$.
\begin{lemma}\label{lemma:L+_topP}
	Each crossing of the path of Lagrangian pairs $\la \mapsto (\E_+^u(\ell,\la), \E_+^s(\ell,\la))$ is positive, thus,
	\begin{equation}\label{eq:Mas_P}
		\Mas(\E_+^u(\ell,\cdot), \E_+^s(\ell,\cdot); [\e,\la_\infty]) = P
	\end{equation}
	for $\e>0$ small enough. Similarly, each crossing of the path $\la \mapsto (\E_-^u(\ell,\la), \E_-^s(\ell,\la))$ is negative, and we have
	\begin{equation}\label{eq:Mas_Q}
		\Mas(\E_-^u(\ell,\cdot), \E_-^s(\ell,\cdot); [\e,\la_\infty]) = -Q.
	\end{equation}
\end{lemma}
\begin{proof}
	We begin with the proof of \eqref{eq:Mas_P}. We proceed by computing the (first-order) relative crossing form \eqref{rel_cross_form} at each crossing $\la=\la_0$, given here by
	\begin{equation}\label{relative_crossform_L+}
		\mathfrak{m}_{\la_0}(\E_+^u(\ell,\cdot),\E_+^s(\ell,\cdot))(w_0) =  \de{}{\la} \w (p(\la),  w_0) \big|_{\la=\la_0} - \de{}{\la} \w (q(\la), w_0) \big|_{\la=\la_0},
	\end{equation}
	where $w_0\in W_1(\E_+^u(\ell,\cdot),\E_+^s(\ell,\cdot),\la_0) = \E_+^u(\ell,\la_0) \cap \E_+^s(\ell,\la_0)$ and $(p,q)$ is a root function pair for $\left ( \E_+^u(\ell,\cdot),\E_+^s(\ell,\cdot) \right )$ with $p(\la_0) = q(\la_0)=w_0$. We compute each of the terms on the right hand side separately.
	
	For the first, recall that $p(\la)\in \E_+^u(\ell,\la)$ for $\la\in[\la_0-\e,\la_0+\e]$ with $\e>0$ small. From the definition of $\E_+^u(\ell,\la)$, it follows that there exists a one-parameter family of solutions $\la\mapsto\mathbf{u}(\cdot;\la)$ to \eqref{1st_order_sys_L+} satisfying $\mathbf{u}(x;\la)\to 0$ as $x\to -\infty$, such that $\mathbf{u}(\ell;\la)=p(\la)$ and $\mathbf{u}(\ell;\la_0) = w_0$. Now
	\begin{align}
		\begin{split}\label{eq:omega_calc}
			\begin{split}
				\w \Big(\de{}{\la}&\mathbf{u}(\ell,\la),\mathbf{u}(\ell,\la)\Big )= \int_{-\infty}^\ell \p_x\, \w(\p_\la\mathbf{u}(x;\la), \mathbf{u}(x;\la) ) \,dx, \\
				&= \int_{-\infty}^\ell \w\left (\p_\la \big [A_+(x;\la)\mathbf{u}(x;\la)\big],  \mathbf{u}(x;\la) \right ) +\w(\p_\la \mathbf{u}(x;\la), A_+(x;\la)\mathbf{u}(x;\la) ) \,dx,
			\end{split}\\
			\begin{split}
				&= \int_{-\infty}^\ell 
				\w(\p_\la \left (A_+(x;\la)\right )\mathbf{u}(x;\la),\mathbf{u}(x;\la) ) + \w(A_+(x;\la)\p_\la\mathbf{u}(x;\la),\mathbf{u}(x;\la) )
				\\ &\qquad \qquad \qquad  +\w(\p_\la \mathbf{u}(x;\la),A_+(x;\la)\mathbf{u}(x;\la) ) \,dx,
			\end{split}\\
			\begin{split}
				&= \int_{-\infty}^\ell 
				\w(\p_\la \left (A_+(x;\la)\right )\mathbf{u}(x;\la),\mathbf{u}(x;\la) )  \\ &\qquad  \qquad \qquad + \big\langle [A_+(x;\la)^\top J + J A_+(x;\la)]\p_\la\mathbf{u}(x;\la),\mathbf{u}(x;\la)\big \rangle \,dx,
			\end{split}\\
			&= \int_{-\infty}^\ell \w(\p_\la \left (A_+(x;\la)\right )\mathbf{u}(x;\la) ,\mathbf{u}(x;\la)) \,dx, 
		\end{split}
	\end{align}
	where we used that $\lim_{x\to-\infty}\mathbf{u}(x;\la)=0$ in the first line and \eqref{infsymp} in the last line. Since
	\begin{equation*}\label{}
		\p_\la  A_+(x;\la)= \left(
		\begin{array}{cccccccc}
			0 & 0 & 0 & 0 \\
			0 & 0 & 0 & 0 \\
			0 & 0 & 0 & 0 \\
			0 & -1 & 0 & 0
		\end{array}
		\right), \qquad 
	\end{equation*}
	and $\mathbf{u}=(u_1, u_2,u_3, u_4)^\top$, evaluating the last line of \eqref{eq:omega_calc} at $\la=\la_0$ we have
	\begin{align}\label{eq:relative_L+_firstterm}
		\de{}{\la} \w (p(\la),  w_0) \big|_{\la=\la_0} = \w \Big(\de{}{\la}\mathbf{u}(\ell,\la),\mathbf{u}(\ell,\la)\Big )\big|_{\la=\la_0}= \int_{-\infty}^\ell u_2(x;\la_0)^2 \,dx.
	\end{align}
	For the second term of \eqref{relative_crossform_L+}, we consider  $q(\la) \in \E_+^s(\ell,\la)$ for $\la\in[\la_0-\e,\la_0+\e]$ with $\e>0$ small. For the same $w_0$ used to compute the first term of \eqref{relative_crossform_L+}, there exists a one-parameter family of solutions $\la\mapsto \widetilde{\mathbf{u}}(\cdot;\la)$ to \eqref{1st_order_sys_L+} satisfying $\widetilde{\mathbf{u}}(x;\la)\to 0$ as $x\to +\infty$, such that $\widetilde{\mathbf{u}}(\ell;\la)=q(\la)$ and $\widetilde{\mathbf{u}}(\ell;\la_0)=w_0$. Arguing as previously, but now using the decay at $+\infty$, we have
	\begin{equation}\label{eq:relative_L+_secondterm}
		\de{}{\la} \w (q(\la), w_0) \big|_{\la=\la_0}= \w\Big(\de{}{\la} \widetilde{\mathbf{u}}(\ell;\la) ,\widetilde{\mathbf{u}}(\ell;\la)\Big)\big|_{\la=\la_0} =  -\int_{\ell}^\infty \widetilde{u}_2(x;\la_0)^2\,dx
	\end{equation}
	(where $\widetilde{\mathbf{u}}=(\widetilde{u}_1, \widetilde{u}_2, \widetilde{u}_3, \widetilde{u}_4)^\top$). Importantly, by uniqueness of solutions we have $\widetilde{\mathbf{u}}(\cdot;\la_0) = \mathbf{u}(\cdot;\la_0)$, so that the integrands in \eqref{eq:relative_L+_secondterm} and  \eqref{eq:relative_L+_firstterm} are the same. Therefore,  \eqref{relative_crossform_L+} becomes
	\begin{equation*}
		\mathfrak{m}_{\la_0}(\E_+^u(\ell,\cdot),\E_+^s(\ell,\cdot))(w_0) =  \int_{-\infty}^\infty u_2(x;\la_0)^2 \,dx >0.
	\end{equation*}
	Thus $n_+(\mathfrak{m}_{\la_0}) = \dim\E_+^u(\ell,\la_0)\cap\E_+^s(\ell,\la_0)$, and all crossings are positive. It follows that the Maslov index counts the number of crossings (up to dimension) of the Lagrangian pair $\la \mapsto (\E_+^u(\ell,\la), \E_+^s(\ell,\la))$, $\la\in[\e,\la_\infty]$, for $\e>0$ small enough. But this is precisely a count of the number of positive eigenvalues of $L_+$ up to multiplicity, i.e. equation \eqref{eq:Mas_P} holds.

	For the path $\la \mapsto (\E_-^u(\ell,\la), \E_-^s(\ell,\la))$, $\la\in[0,\la_\infty]$, the argument is similar. Now the Maslov index counts, with negative sign, the number of crossings along $\Gamma_2$. The sign change results from the fact that $\la$ now appears with positive sign in the first order system \eqref{1st_order_sys_L-}, so that
	\begin{equation*}%\label{}
		\p_\la A_-(x;\la) = \left(
		\begin{array}{cccccccc}
			0 & 0 & 0 & 0 \\
			0 & 0 & 0 & 0 \\
			0 & 0 & 0 & 0 \\
			0 & 1 & 0 & 0
		\end{array}
		\right).
	\end{equation*}
	The associated crossing form will then be negative, and by similar reasoning equation \eqref{eq:Mas_Q} holds. 
\end{proof}

\subsection{Computing the Maslov index along $\Gamma_3$ and $\Gamma_4$}

The following lemma shows that there are no crossings along $\Gamma_3$ and $\Gamma_4$.
\begin{lemma}\label{lemma:L+_no_crossings}
	We have $\E_+^u(x,\la_\infty) \cap \E_+^s(\ell,\la_\infty) =\{0\}$ for all $x\in\R$, provided both $\la_\infty>0$ and $\ell>0$ are large enough; hence
	 \[
	 \Mas(\E_+^u(\cdot,\la_\infty), \E_+^s(\ell,\la_\infty); [-\infty,\ell]) = 0.
	 \]
	 In addition, $\U_+(\la) \cap \E_+^s(\ell,\la) =\{0\}$ for all $\la \geq 0$ provided $\ell>0$ is large enough; hence
	\begin{equation*}%\label{}
		\Mas( \U_+(\cdot), \E_+^s(\ell,\cdot);[0,\la_\infty]) = 0.
	\end{equation*}
	Similar statements hold for the paths  $x\mapsto \left (\E_-^u(x,\la_\infty), \E_-^s(\ell,\la_\infty)\right )$ and $\la\mapsto \left (\U_+(\la), \E_+^s(\ell,\la\right )$.
%	 In addition, $\U_+(\la) \cap \E_+^s(\ell,\la) =\{0\}$ for all $\la \geq 0$ provided $\ell>0$ is large enough. Therefore
%	\begin{equation*}%\label{}
%		\Mas(\E_+^u(\cdot,\la_\infty), \E_+^s(\ell,\la_\infty); [-\infty,\ell]) = 	\Mas( \U_+(\cdot), \E_+^s(\ell,\cdot);[0,\la_\infty]) = 0.
%	\end{equation*}
%	Similar statements hold for the paths  $x\mapsto \left (\E_-^u(x,\la_\infty), \E_-^s(\ell,\la_\infty)\right )$ and $\la\mapsto \left (\U_+(\la), \E_+^s(\ell,\la\right )$.
\end{lemma}
\begin{proof}
	The strategy of the following proof mirrors the one given in \cite[\S4]{Corn19} (see also \cite[\S3 and \S5.B]{AGJ90}). For the first statement, we first note that $\spec(L_+)$ is bounded from above. To see this, we write
	\begin{equation*}%\label{eq:D+V}
		L_+ = \cA+\cV, \qquad \cA=-\p_{xxxx}-\sigma_2 \p_{xx}, \quad \cV = -\be + 3\phi(x)^2,
	\end{equation*}
	where $\dom(\cA)=H^4(\R)$, so that $\cA=\cA^*$ is selfadjoint and $\cV$ is bounded and symmetric on $L^2(\R)$. It can be shown that $\cA$ has no point spectrum, and moreover that $\spec(\cA) = \esspec(\cA) = (-\infty,1/4]$ if $\sigma_2=1$, and  $\spec(\cA)= \esspec(\cA)=(-\infty,0]$ if $\sigma=-1$. Consequently, by virtue of \cite[Theorem V.4.10, p.291]{Kato} we have
	\begin{equation*}\label{}
		\dist\left (\spec(L_+), \spec(\cA)  \right ) \leq \|\cV\|,
	\end{equation*}
	so that $\spec(L_+) \subseteq (-\infty,\|\cV\|]$. It then follows that $\E_+^u(\ell,\la)\cap \E^s_+(\ell,\la) =\{0\}$ for all $\la>\|\cV\|$.
	
	Next, we claim that there exists a $\la_\infty>\|\cV\|$ such that 
	\begin{equation*}%\label{claim_transverse_large_la}
		\E_+^u(x,\la) \cap \Ss_+(\la) = \{0\}
	\end{equation*}
	for all $x\in\R$ and all $\la\geq \la_\infty$. Once this is shown, it follows that there exists an $\ell_\infty\gg 1$ such that
	\begin{equation*}\label{}
		\E_+^u(x,\la_\infty) \cap \E^s_+(\ell,\la_\infty) = \{0\}
	\end{equation*}
	for all $x\in\R$ and all $\ell\geq \ell_\infty$, because $\lim_{x\to\infty}\E^s_+(x,\la) = \Ss_+(\la)$. To prove the claim, we mimic the proof of \cite[Lemma 4.1]{Corn19}. Using the change of variables
	\begin{equation}\label{eq:change_variables}
		y=\lambda^{1/4} x, \qquad \widetilde{u}_1=u_1, \quad\widetilde{u}_2= \lambda^{1/2}u_2,   \quad\widetilde{u}_3= \lambda^{1/4}u_3,  \quad\widetilde{u}_4= \lambda^{-1/4}u_4,
	\end{equation}
	transforms \eqref{1st_order_sys_L+} to the system 
	\begin{equation}\label{1st_order_sys_L+_rescaled_lam}
		\de{}{y} \begin{pmatrix}
			\widetilde{u}_1  \\  \widetilde{u}_2 \\  \widetilde{u}_3 \\  \widetilde{u}_4 
		\end{pmatrix} = 
		\begin{pmatrix}
			0 &  0 & \f{\sigma_2}{\sqrt{\la}}    & 1  \\
			0 & 0 &	1  & 0 \\
			1  &-\f{\sigma_2}{\sqrt{\la}}    & 0 & 0 \\ 
			-\f{\sigma_2}{\sqrt{\la}}   & \f{\al\left (\f{y}{\sqrt[4]{\lambda}}\right )}{\la}-1  & 0 & 0
		\end{pmatrix}
		\begin{pmatrix}
			\widetilde{u}_1  \\  \widetilde{u}_2 \\  \widetilde{u}_3 \\  \widetilde{u}_4 
		\end{pmatrix}
	\end{equation}
	(recall that $\al\left (\f{y}{\sqrt[4]{\lambda}}\right ) = 3\phi(\f{y}{\sqrt[4]{\lambda}})^2 - \be +1$). Taking $y\to\pm\infty$, the asymptotic system for \eqref{1st_order_sys_L+_rescaled_lam} is given by
	\begin{equation}\label{1st_order_sys_L+_rescaled_lam_asymptotic}
		\de{}{y} \begin{pmatrix}
			\widetilde{u}_1  \\  \widetilde{u}_2 \\  \widetilde{u}_3 \\  \widetilde{u}_4 
		\end{pmatrix} = 
		\begin{pmatrix}
			0 &  0 & \f{\sigma_2}{\sqrt{\la}}    & 1  \\
			0 & 0 &	1  & 0 \\
			1  &-\f{\sigma_2}{\sqrt{\la}}    & 0 & 0 \\ 
			-\f{\sigma_2}{\sqrt{\la}}   & \f{-\be+1}{\la}-1  & 0 & 0
		\end{pmatrix}
		\begin{pmatrix}
			\widetilde{u}_1  \\  \widetilde{u}_2 \\  \widetilde{u}_3 \\  \widetilde{u}_4 
		\end{pmatrix} .
	\end{equation}
	Denote the stable and unstable subspaces for \eqref{1st_order_sys_L+_rescaled_lam_asymptotic} by $\widetilde{\Ss}_+(\la)$ and $\widetilde{\U}_+(\la)$ respectively, and denote the unstable bundle of \eqref{1st_order_sys_L+_rescaled_lam} by $\widetilde{\E}^u_+(y,\la)$. Then, we have
	\begin{equation*}\label{}
		\E_+^u(x,\la) \cap \Ss_+(\la) = \{0\} \iff \widetilde{\E}^u_+(\lambda^{1/4} x,\la)\cap \widetilde{\Ss}_+(\la) =\{0\},
	\end{equation*}
	since from \eqref{eq:change_variables} one has $\widetilde{\E}^u_+(\lambda^{1/4} x,\la) = M \cdotp \E^u(x,\la)$ and $\widetilde{\Ss}_+(\la)=M\cdotp  \Ss_+(\la)$, where $M=\diag\{1,\la^{1/2},\la^{1/4},$ $\la^{-1/4}\}$ is nonsingular and $``\cdot"$ is the induced action of $M$ on $\R^4$.

	Both the nonautonomous system \eqref{1st_order_sys_L+_rescaled_lam} and the autonomous system \eqref{1st_order_sys_L+_rescaled_lam_asymptotic} induce flows on $\Gr_{2}(\R^{4})$. For the flow associated with \eqref{1st_order_sys_L+_rescaled_lam_asymptotic}, it is known \cite{AGJ90} that $\widetilde{\U}_+(\la)$, the invariant subspace associated with eigenvalues of positive real part, is an attracting fixed point. Thus, since $\cL(2)\subset \Gr_2(\R^4)$, there exists a trapping region $\mathcal{R}\subset \cL(2)$ containing $\widetilde{\U}_+(\la)$. By taking $\la$ large enough, we can ensure that the flow induced by \eqref{1st_order_sys_L+_rescaled_lam} is as close as we like to that induced by \eqref{1st_order_sys_L+_rescaled_lam_asymptotic}, because $\phi\left (\f{y}{\sqrt[4]{\lambda}}\right )^2/\la$ --  the nonautonomous part of \eqref{1st_order_sys_L+_rescaled_lam} -- is close to zero. It follows that $\mathcal{R}\subset \cL(2)$ is also a trapping region for \eqref{1st_order_sys_L+_rescaled_lam}. Furthermore, we can choose $\mathcal{R}$ small enough such that $\mathbb{V} \cap \widetilde{\Ss}_+(\la)=\{0\}$ for all $\mathbb{V}\in \mathcal{R}$, uniformly for $\la$ large enough. To see this, note that clearly $\widetilde{\Ss}_+(\la) \cap \widetilde{\U}_+(\la) =\{0\}$, while taking $\la\to+\infty$ in \eqref{1st_order_sys_L+_rescaled_lam_asymptotic} yields
	\begin{equation*}%\label{1st_order_sys_L+_rescaled_lam_asymptotic_la_infty}
		\de{}{y} \begin{pmatrix}
			\widetilde{u}_1  \\  \widetilde{u}_2 \\  \widetilde{u}_3 \\  \widetilde{u}_4 
		\end{pmatrix} = 
		\begin{pmatrix}
			0 &  0 & 0   & 1  \\
			0 & 0 &	1  & 0 \\
			1  & 0   & 0 & 0 \\ 
			0  & -1  & 0 & 0
		\end{pmatrix}
		\begin{pmatrix}
			\widetilde{u}_1  \\  \widetilde{u}_2 \\  \widetilde{u}_3 \\  \widetilde{u}_4 
		\end{pmatrix},
	\end{equation*}
	which has stable and unstable subspaces $\widetilde{\Ss}_{+_\infty}$ and $\widetilde{\U}_{+_\infty}$ with respective frames $(I,-W)$ and $(I,W)$, where
	\[
	W=\f{1}{\sqrt{2}}\begin{pmatrix}
		1 & 1 \\ 1 & -1
	\end{pmatrix}.
	\]
	Thus, in the limit we also have $\widetilde{\Ss}_{+_\infty}\cap \widetilde{\U}_{+_\infty}=\{0\}$, so we can choose $\mathcal{R}$ as stated. Finally, we note that if $\la>\|\cV\|$ so that $\la\notin \spec(L_+)$, then by \cite[Lemma 3.7]{AGJ90} we have $\lim_{y\to\infty} \widetilde{\E}^u_+(y,\la) = \widetilde{\U}_+(\la)$. All in all, we conclude that for any $\la_\infty>\|\cV\|$, the trajectory $\widetilde{\E}^u_+(\cdot,\la_\infty):[-\infty,\infty]\to \cL(2)$, which starts and finishes at $\widetilde{\U}_+(\la_\infty)$, will remain inside $\mathcal{R}$ and thus always be disjoint from $\widetilde{\Ss}_+(\la_\infty)$. This proves the claim. 
	
	For the second statement of the lemma, the facts that  $\U_+(\la) \cap \,\Ss_+(\la) = \{0\}$ for all $\la\geq 0$ and $\lim_{x\to\infty}\E_+^s(x,\la) = \Ss_+(\la)$ imply that there exists an $\ell_\infty\gg 1$ such that $\U_+(\la)\cap \E_+^s(x,\la)=\{0\}$ for all $x\geq \ell_\infty$. Taking $x=\ell>\ell_\infty$ gives the result.
\end{proof}

\subsection{Proving the Morse-Maslov theorem}
In what follows, we choose $\ell>0$ and $\la_\infty>0$ large enough in accordance with \cref{lemma:L+_no_crossings}.
\begin{proof}[Proof of \cref{thm:Lpm_eval_counts}]
	By homotopy invariance and additivity under concatenation, we have 
	\begin{multline*}%\label{eq:Mas_L+_box}
		\Mas(\E_+^u(\cdot,0),\E_+^s(\ell,0);[-\infty,\ell] ) + \Mas(\E_+^u(\ell,\cdot),\E_+^s(\ell,\cdot);[0,\la_\infty] )  \\  -\Mas(\E_+^u(\cdot,\la_\infty),\E_+^s(\ell,\la_\infty);[-\infty,\ell] )  -\Mas(\E_+^u(-\infty,\cdot),\E_+^s(\ell,\cdot);[0,\la_\infty] ) =0.
	\end{multline*}
	From \cref{lemma:L+_no_crossings} the third and fourth terms on the left hand side vanish. Again using the concatenation property, we find that
	\begin{multline}\label{eq:box1}
		\Mas(\E_+^u(\cdot,0),\E_+^s(\ell,0);[-\infty,\ell-\e] ) + \Mas(\E_+^u(\cdot,0),\E_+^s(\ell,0);[\ell-\e,\ell] ) \\+ \Mas(\E_+^u(\ell,\cdot),\E_+^s(\ell,\cdot);[0,\e] ) +\Mas(\E_+^u(\ell,\cdot),\E_+^s(\ell,\cdot);[\e,\la_\infty] ) =0
	\end{multline}
	where $\e>0$ is small. The second and third terms of \eqref{eq:box1}  represent the contributions to the Maslov index from the conjugate point $(x,\la) = (\ell,0)$ at the top left corner of the Maslov box in the $x$ and $\la$ directions respectively. From \eqref{eq:final_conj_pairs}, \cref{lemma:L+_topP} and \cref{define:Maslov_GPP} we have  
	\begin{equation}\label{eq:box2}
		\Mas(\E_+^u(\cdot,0),\E_+^s(\ell,0);[\ell-\e,\ell] ) = \Mas(\E_+^u(\ell,\cdot),\E_+^s(\ell,\cdot);[0,\e]) =0.
	\end{equation} 
	\Cref{lemma:L+_Mas_left_conjpoints,lemma:L+_half_left_full_left} imply that
	\begin{equation}\label{left_conj_points_L+}
		\Mas(\E_+^u(\cdot,0),\E_+^s(\ell,0);[-\infty,\ell-\e] ) =  -p_c.
	\end{equation}
	The previous three equations along with \cref{lemma:L+_topP} now yield \eqref{eq:P=conj_points}.
	
	The proof for the Morse index of the $L_-$ operator is similar. This time, crossings along $\Gamma_1$ are positive, while crossings along $\Gamma_2$ are negative. Arguing as we did for \eqref{eq:final_conj_pairs}, we have
	\begin{equation}\label{eq:box3}
		\Mas(\E_-^u(\cdot,0),\E_-^s(\ell,0);[\ell-\e,\ell] )  = 1, 
	\end{equation}
	and from \cref{lemma:L+_topP} and \cref{define:Maslov_GPP} we have
	\begin{equation}\label{22}
		\Mas(\E_-^u(\ell,\cdot),\E_-^s(\ell,\cdot);[0,\e]) = -1.
	\end{equation}
	The contributions \eqref{eq:box3} and \eqref{22} coming from the corner crossing $(x,\la)=(\ell,0)$ thus cancel each other out. Applying the same homotopy argument as we did for $L_+$ yields \eqref{eq:Q=conj_points}.
\end{proof}
\begin{rem}\label{rem:general_NLS_powerlaw2}
	As per \cref{rem:general_NLS_powerlaw}, monotonicity of the Lagrangian paths $x\mapsto \E^u_\pm(x,0)$ with respect to the stable subspace $\Ss_\pm(0)$ is preserved in the case of \eqref{Lpm_powerlaw} for any $p\in\N$. Since all other results in the current section remain valid for that case, the proof presented above applies without modification for soliton solutions to \eqref{general_NLS_powerlaw}.
\end{rem}

\section{Proving the lower bound and the Vakhitov-Kolokolov criterion}\label{sec:proof_lower_bound}

We now return to the computation of the Maslov indices appearing on the left hand side of \eqref{mas_box_argument}. After computing each, we provide the proofs of \cref{thm:VK_criterion,thm:main_lower_bound}. We begin with $\Gamma_1$.

\begin{lemma}\label{lemma_P-Q}
	$\Mas(\E^u(\cdot,0), \E^s(\ell,0); [-\infty,\ell-\e]) = Q-P$, where $\e>0$ is small.
\end{lemma}
\begin{proof}
	{ Recalling the arguments in \cref{subsec:symplectic_evals} regarding the decoupling of \eqref{1st_order_sys} when $\la=0$, we have 
	\begin{equation*}\label{}
		\E^u(x,0) = \E_+^u(x,0) \oplus \E_-^u(x,0), \qquad \E^s(\ell,0) = \E_+^s(\ell,0) \oplus \E_-^s(\ell,0).
	\end{equation*}
}
	Now using property (3) of \cref{prop:enjoys}, we have
	\begin{align*}%\label{mas_direct_sum}
		\begin{split}
			\Mas(\E^u(\cdot,0), \E^s(\ell,0) ; [-\infty,\ell-\e]) &= \Mas(\E_+^u(\cdot,0), \E_+^s(\ell,0) ; [-\infty,\ell-\e]) \\
			& \qquad  + \Mas(\E_-^u(\cdot,0), \E_-^s(\ell,0); [-\infty,\ell-\e]),
		\end{split}
	\end{align*}
	and the result follows combining \eqref{left_conj_points_L+} (and the accompanying statement for $L_-$) with \eqref{eq:P=conj_points} and \eqref{eq:Q=conj_points}.
\end{proof}
Next, similar to \cref{lemma:L+_no_crossings} we show that there are no crossings along $\Gamma_3$ and $\Gamma_4$.
\begin{lemma}\label{lemma:Mas_zero_bottom_right}
	We have $\E^u(x,\la_\infty) \cap \E^s(\ell,\la_\infty)= \{0\}$ for all $x\in\R$ provided $\la_\infty>0$ and $\ell>0$ are large enough. In addition, $\U(\la)\cap\E^s(\ell,\la)=\{0\}$ for all $\la \geq 0$ provided $\ell>0$ is large enough. Therefore 
	\[
	\Mas(\E^u(\cdot,\la_\infty), \E^s(\ell,\la_\infty); [-\infty,\ell]) = \Mas( \U(\cdot), \E^s(\ell,\cdot);[0,\la_\infty])=0.
	\]
\end{lemma}
\begin{proof}
	For the first assertion, note that $N$ is a bounded perturbation of a skew-selfadjoint operator, so that its spectrum lies in a vertical strip around the imaginary axis in the complex plane. More precisely, we have that 
	\begin{equation*}\label{}
		iN = \widetilde{\cA}+\widetilde{\cV}, \quad \widetilde{\cA}= i  \begin{pmatrix}
			0 & \p_{xxxx}+\sigma_2\p_{xx} \\
			-\p_{xxxx} - \sigma_2 \p_{xx} & 0
		\end{pmatrix}, \quad \widetilde{\cV}=i\begin{pmatrix}
			0 & \be-\phi^2 \\ -\be+3\phi^2 & 0
		\end{pmatrix},
	\end{equation*}
	where, with $\dom(\widetilde{\cA})= \dom(N)$, $\widetilde{\cA}^*= \widetilde{\cA}$ is selfadjoint in $L^2(\R)$ and $\widetilde{\cV}$ is bounded. Now using \cite[Remark 3.2, p.208]{Kato} and \cite[eq. (3.16), p.272]{Kato}, we conclude that 
	\begin{equation*}\label{}
		\zeta \in \spec(\widetilde{\cA}+\widetilde{\cV}) \im | \text{Im}(\zeta)| \leq \| \widetilde{\cV} \|.
	\end{equation*}
	By the spectral mapping theorem, $\spec(iN) = i\spec(N)$. It follows that
	\begin{equation*}\label{}
		\la \in \spec(N) \im | \text{Re}(\la)| \leq \| \widetilde{\cV} \|.
	\end{equation*}
	Thus, for all $\la>\|\widetilde{\cV}\|$ we have $\E^u(\ell,\la)\cap\E^s(\ell,\la)=\{0\}$. 
	
	The proof now follows from the same arguments used to prove the first assertion in \cref{lemma:L+_no_crossings}. Namely, via the change of variables \eqref{eq:change_variables} along with
	\begin{equation*}\label{}
		\widetilde{v}_1=v_1, \quad\widetilde{v}_2= \lambda^{1/2}v_2,   \quad\widetilde{v}_3= \lambda^{1/4}v_3,  \quad\widetilde{v}_4= \lambda^{-1/4}v_4,
	\end{equation*}
	we can rewrite \eqref{1st_order_sys} as
	\begin{equation}\label{1st_order_sys_N_rescaled_lam}
		\de{}{y}\begin{pmatrix}
			\widetilde{u}_1 \\ \widetilde{v}_1 \\ \widetilde{u}_2 \\ \widetilde{v}_2 \\ \widetilde{u}_3 \\ \widetilde{v}_3 \\ \widetilde{u}_4 \\ \widetilde{v}_4 \\
		\end{pmatrix} = 
		\left(\begin{array}{@{}c|c@{}}
			0 &  \begin{matrix}
				\f{\sigma_2}{\sqrt{\la}}  & 0 & 1 & 0 \\
				0 & -\f{\sigma_2}{\sqrt{\la}}  & 0 & 1 \\
				1 & 0 & 0 & 0 \\
				0 & 1 & 0 & 0 
			\end{matrix}  \\
			\hline \\[-4mm]
			\begin{matrix}
				1 & 0 & -\f{\sigma_2}{\sqrt{\la}}   & 0  \\ 
				0 & -1 & 0  & -\f{\sigma_2}{\sqrt{\la}}    \\ 
				-\f{\sigma_2}{\sqrt{\la}}   & 0 & \f{\al(x)}{\la}  & 1  \\
				0 & -\f{\sigma_2}{\sqrt{\la}}  & 1    & \f{\eta(x)}{\la}
			\end{matrix}  & 0 
		\end{array}\right)
		\begin{pmatrix}
			\widetilde{u}_1 \\ \widetilde{v}_1 \\ \widetilde{u}_2 \\ \widetilde{v}_2 \\ \widetilde{u}_3 \\ \widetilde{v}_3 \\ \widetilde{u}_4 \\ \widetilde{v}_4 \\
		\end{pmatrix}.
	\end{equation}
	Again, the flow of the associated asymptotic system is close to that of \eqref{1st_order_sys_N_rescaled_lam} for large $\la$. From the transversality of the four dimensional stable and unstable subspaces of the limiting system of \eqref{1st_order_sys_N_rescaled_lam} as $\la\to\infty$, i.e.
	\begin{equation}\label{1st_order_sys_N_rescaled_lam_limit}
		\de{}{y}\begin{pmatrix}
			\widetilde{u}_1 \\ \widetilde{v}_1 \\ \widetilde{u}_2 \\ \widetilde{v}_2 \\ \widetilde{u}_3 \\ \widetilde{v}_3 \\ \widetilde{u}_4 \\ \widetilde{v}_4 \\
		\end{pmatrix} = 
		\left(\begin{array}{@{}c|c@{}}
			0 &  \begin{matrix}
				0  & 0 & 1 & 0 \\
				0 &0  & 0 & 1 \\
				1 & 0 & 0 & 0 \\
				0 & 1 & 0 & 0 
			\end{matrix}  \\
			\hline \\[-4mm]
			\begin{matrix}
				1 & 0 & 0   & 0  \\ 
				0 & -1 & 0  & 0   \\ 
				0   & 0 & 0  & 1  \\
				0 & 0 & 1    & 0
			\end{matrix}  & 0 
		\end{array}\right)
		\begin{pmatrix}
			\widetilde{u}_1 \\ \widetilde{v}_1 \\ \widetilde{u}_2 \\ \widetilde{v}_2 \\ \widetilde{u}_3 \\ \widetilde{v}_3 \\ \widetilde{u}_4 \\ \widetilde{v}_4 \\
		\end{pmatrix},
	\end{equation}
	one can show that there exists a $\la_\infty>\|\widetilde{\cV}\|$ such that $\E^u(x,\la)\cap\Ss(\la)=\{0\}$ for all $x\in\R$ and all $\la\geq \la_\infty$. Hence there exists $\ell_\infty\gg 1$ such that $\E^u(x,\la_\infty)\cap \E^s(\ell,\la_\infty)=\{0\}$ for all $x\in\R$ and $\ell\geq \ell_\infty$. The second assertion follows from the same arguments used to prove the second assertion in \cref{lemma:L+_no_crossings}. 
\end{proof}	
For the proof of \cref{thm:main_lower_bound}, it remains to compute the contribution to the Maslov index from the corner crossing $(x,\la)=(\ell,0)$, i.e.
\begin{equation}\label{c_again}
	\mathfrak{c} = \Mas(\E^u(\cdot,0),\E^s(\ell,0);[\ell-\e,\ell] ) + \Mas(\E^u(\ell,\cdot),\E^s(\ell,\cdot);[0,\e] ).
\end{equation}
{ \begin{lemma}\label{lemma:value_of_c}
	The value of $\mathfrak{c}$ is given by
	\begin{equation}\label{c_value}
		\mathfrak{c} =
		\begin{cases*}
			1 & \,$	\mathcal{I}_1>0, \, \mathcal{I}_2<0$, \\
			0 & \,$\mathcal{I}_1\mathcal{I}_2>0$, \\
			-1  & \,$\mathcal{I}_1<0, \,\mathcal{I}_2>0$.
		\end{cases*}
	\end{equation}
\end{lemma}}
\begin{proof}
For the first term in \eqref{c_again}, i.e. the arrival along $\Gamma_1$, again using property (3) of \cref{prop:enjoys} and equations \eqref{eq:box2} and \eqref{eq:box3}, we have
\begin{align}\label{eq:mathfrak_b}
	\begin{split}
		\Mas(\E^u(\cdot,0), \E^s(\ell,0) ; [\ell-\e,\ell]) &= \Mas(\E_+^u(\cdot,0), \E_+^s(\ell,0) ; [\ell-\e,\ell]) \\
		& \qquad  + \Mas(\E_-^u(\cdot,0), \E_-^s(\ell,0); [\ell-\e,\ell]), \\
		&=1.
	\end{split}
\end{align}
For the second term in \eqref{c_again}, i.e. the departure along $\Gamma_2$, we compute crossing forms. To that end, suppose $\la_0\in[0,\la_\infty]$ (not necessarily zero) is a crossing of the Lagrangian pair $\la\mapsto \left (\E^u(\ell,\la), \E^s(\ell,\la)\right )$. The first-order relative crossing form \eqref{rel_cross_form} is given by
\begin{equation}\label{rel_crossform_lambda}
	\mathfrak{m}_{\la_0}(\E^u(\ell,\cdot),\E^s(\ell,\cdot))(w_0) = \de{}{\la} \w (r(\la), w_0) \big|_{\la=\la_0} - \de{}{\la} \w (s(\la), w_0) \big|_{\la=\la_0},
\end{equation}
where $w_0\in W_1(\E^u(\ell,\cdot),\E^s(\ell,\cdot),\la_0) = \E^u(\ell,\la_0) \cap \E^s(\ell,\la_0)$ and $(r,s)$ is a root function pair for $(\E^u(\ell,\cdot),\E^s(\ell,\cdot))$ with $r(\la_0)=s(\la_0)=w_0$. As in the proof of \cref{lemma:L+_topP}, we compute each of these terms separately. For the first, noting that $r(\la)\in \E^u(\ell,\la)$ for $\la\in[\la_0-\e,\la_0+\e]$, it follows from the definition of $\E^u(\ell,\la)$ that there exists a one-parameter family of solutions $\la\mapsto\mathbf{w}(\cdot;\la)$ to \eqref{1st_order_sys}, such that $\mathbf{w}(x;\la)\to 0$ as $x\to -\infty$, $\mathbf{w}(\ell;\la)=r(\la)$ and $\mathbf{w}(\ell;\la_0)=w_0$. Thus
\begin{align*}
	\de{}{\la} \w (r(\la), w_0) \big|_{\la=\la_0}&= \w \Big(\de{}{\la}\mathbf{w}(\ell,\la),\mathbf{w}(\ell,\la)\Big )\big|_{\la=\la_0},
\end{align*}
and a calculation similar to \eqref{eq:omega_calc} with
\begin{equation}\label{eq:A_lambda_N}
	\p_\la A (x;\la) = \left(
	\begin{array}{cccccccc}
		0_4 & 0_4 \\
		A_{21} & 0_4 
	\end{array}
	\right), \qquad A_{21}= \begin{pmatrix}
		0 & 0 \\ 0 & 1
	\end{pmatrix} \otimes \begin{pmatrix}
		0 & 1 \\ 1 & 0
	\end{pmatrix},
\end{equation}
and $\mathbf{w}=(u_1, v_2, u_2,v_2, u_3, v_3, u_4, v_4)^\top$ yields
\begin{align*}
	\de{}{\la} \w (r(\la), w_0) \big|_{\la=\la_0} = -2 \int_{-\infty}^\ell u_2(x;\la_0) v_2(x;\la_0)\,dx.
\end{align*}
For the second term in \eqref{rel_crossform_lambda} and the same fixed $w_0$, associated with $s(\la)\in \E^s(\ell,\la)$ is a family of solutions $\la\mapsto \widetilde{\mathbf{w}}(\cdot;\la)$ to \eqref{1st_order_sys} such that $\widetilde{\mathbf{w}}(x;\la)\to 0$ as $x\to+\infty$, $\widetilde{\mathbf{w}}(\ell;\la)=s(\la)$ and $\widetilde{\mathbf{w}}(\ell;\la_0)=w_0$. Arguing as for the first term of \eqref{rel_crossform_lambda}, but now using the decay at $+\infty$, we have
\begin{equation*}\label{}
	\de{}{\la} \w (s(\la), w_0) \big|_{\la=\la_0} = \w\Big(\de{}{\la} \widetilde{\mathbf{w}}(\ell;\la),\widetilde{\mathbf{w}}(\ell;\la) \Big)\Big|_{\la=\la_0} = 2 \int_{\ell}^\infty \widetilde{u}_2(x;\la_0) \widetilde{v}_2(x;\la_0)\,dx.
\end{equation*}
Using uniqueness of solutions as in the proof of \cref{lemma:L+_topP}, we conclude
\begin{equation}\label{eq:final_1st_la_form}
	\mathfrak{m}_{\la_0}(\E^u(\ell,\cdot),\E^s(\ell,\cdot))(w_0) = -2 \int_{-\infty}^\infty u_2(x;\la_0) v_2(x;\la_0) \,dx.
\end{equation}
\begin{rem}\label{not_sign_definite}
	The form \eqref{eq:final_1st_la_form} is \emph{not} sign definite, and therefore the Maslov index does not afford an exact count of the crossings of the path $\la\mapsto \left (\E^u(\ell,\cdot),\E^s(\ell,\cdot)\right )$ for $\la\in[0,\la_\infty]$. This will be the reason for the inequality (and not equality) \eqref{lower_bound} in  \cref{thm:main_lower_bound}.
\end{rem}
We now evaluate the form \eqref{eq:final_1st_la_form} at $\la_0=0$, where $W_1((\E^u(\ell,\cdot),\E^s(\ell,\cdot),0) = \E^u(\ell,0) \cap \E^s(\ell,0)$. Recalling \eqref{intersection_at_zero_phis}, we have $\E^u(\ell,0) \cap \E^s(\ell,0) = \spn\{\pmb{\phi}(\ell), \pmb{\varphi}(\ell)\}$, where $\pmb{\phi}$ and $\pmb{\varphi}$ are given in \eqref{eq:phi_varphi}. Hence, it suffices to evaluate \eqref{eq:final_1st_la_form} on $w_0 = \pmb{\phi}(\ell) k_1 + \pmb{\varphi}(\ell)k_2$ for some $k_1,k_2\in\R$. Evidently, the family $\mathbf{w}(\cdot;0)=\widetilde{\mathbf{w}}(\cdot;0)$ described above is given by $\mathbf{w}(x;0) = \pmb{\phi}(x) k_1 + \pmb{\varphi}(x)k_2$, so that $u_2(x;0) = \phi'(x)k_1$ and $v_2(x;0) = -\phi(x)k_2$. Hence
\begin{align*}%\label{eq:first_la_degenerate_N}
	\mathfrak{m}_{\la_0}(\E^u(\ell,\cdot),\E^s(\ell,\cdot))(w_0) &=  	2 \int_{-\infty}^{\infty} (\phi'k_1) (\phi k_2) \,dx  = 	\left (\int_{-\infty}^{\infty} \de{}{x}\phi^2 \,dx\right ) \,k_1k_2 = 0,
\end{align*}
since $\phi\in H^4(\R)$. That is, the relative crossing form $\mathfrak{m}_{\la_0}$ in \eqref{rel_crossform_lambda} is identically zero at $\la_0=0$, and the conjugate point $(\ell,0)$ is non-regular in the $\la$ direction. Therefore, $W_2(\E^u(\ell,\cdot),\E^s(\ell,\cdot),0) = \ker \mathfrak{m}_{\la_0}(\E^u(\ell,\cdot),\E^s(\ell,\cdot)) =  \E^u(\ell,0) \cap \E^s(\ell,0)$, and we need to compute higher order crossing forms.

To that end, in this case the second-order relative crossing form \eqref{rel_cross_form} at $\la_0=0$ is given by
\begin{equation}\label{relative_crossform_N}
	\mathfrak{m}^{(2)}_{\la_0}(\E^u(\ell,\cdot),\E^s(\ell,\cdot))(w_0) =  \des{}{\la} \w (r(\la), w_0) \big|_{\la=0} -  \des{}{\la} \w (s(\la), w_0) \big|_{\la=0},
\end{equation}
where $w_0\in \E^u(\ell,0) \cap \E^s(\ell,0)$ and $(r,s)$ is a root function pair for $(\E^u(\ell,\cdot),\E^s(\ell,\cdot))$ with $r(0)=s(0)=w_0$ and $\dot r(0) = \dot s(0)$. (Dot denotes $d/d\la$.) For the first term of \eqref{relative_crossform_N}, we again have a one-parameter family $\la\to\mathbf{w}(\cdot;\la)$ decaying to zero as $x\to-\infty$ such that $\mathbf{w}(\ell;\la)=r(\la)$ and $\mathbf{w}(\ell;0)=w_0$. Now 
\begin{align*}
	\begin{split}
		\w \Big(&\des{}{\la}\mathbf{w}(\ell,\la),\mathbf{w}(\ell,\la)\Big ) = \int_{-\infty}^\ell \p_x\, \w(\p_{\la\la} \mathbf{w}(x;\la), \mathbf{w}(x;\la)) dx, \\
		&= \int_{-\infty}^\ell \w\left (\p_{\la\la} \left [A(x;\la)\mathbf{w}(x;\la)\right ],\mathbf{w}(x;\la) \right ) + \w\left (\p_{\la\la} \mathbf{w}(x;\la) , A(x;\la)\mathbf{w}(x;\la)\right )  dx , \\
		\begin{split}
			&= \int_{-\infty}^\ell \w(A_{\la\la}(x;\la)\mathbf{w}(x;\la),\mathbf{w}(x;\la) )  + 2\,\w(A_\la(x;\la)\p_\la\mathbf{w}(x;\la),\mathbf{w}(x;\la) ) \\ 
			&\qquad \qquad    +\w(A(x;\la)\p_{\la\la}\mathbf{w}(x;\la), \mathbf{w}(x;\la))  +\w\left (\p_{\la\la} \mathbf{w}(x;\la) ,A(x;\la)\mathbf{w}(x;\la)\right)dx,
		\end{split}\\
		\begin{split}
			&= \int_{-\infty}^\ell  \langle [A(x;\la)^\top J+ JA(x;\la)]\p_{\la\la}\mathbf{w}(x;\la), \mathbf{w}(x;\la) \rangle \\ &\qquad \qquad \qquad\qquad\qquad  + 2\,\w(A_\la(x;\la)\p_\la\mathbf{w}(x;\la),\mathbf{w}(x;\la) ) \,dx,
		\end{split}\\
		&= 2\int_{-\infty}^\ell \w(A_\la(x;\la)\p_\la\mathbf{w}(x;\la),\mathbf{w}(x;\la) ) dx,
	\end{split}
\end{align*}
where we used \eqref{infsymp} and $A_{\la\la}(x;\la)=0$. Using \eqref{eq:A_lambda_N} and evaluating at  $\la=0$, we find
\begin{align}
	\des{}{\la} \w (r(\la), w_0)  &= \w \Big(\des{}{\la}\mathbf{w}(\ell,\la),\mathbf{w}(\ell,\la)\Big )\Big|_{\la=0}, \nonumber\\
	&= -2\int_{-\infty}^\ell 
	u_2(x;0) \p_\la  v_2(x;0) +  v_2(x;0) \p_\la u_2(x;0)\,dx. \label{first_2nd_order_form_term}
\end{align}
For the second term in \eqref{relative_crossform_N}, we have a family $\la\to\widetilde{\mathbf{w}}(\cdot;\la)$ decaying to zero as $x\to+\infty$ such that $\widetilde{\mathbf{w}}(\ell;\la)=s(\la)$ and $\widetilde{\mathbf{w}}(\ell;0)=w_0$, and a similar argument to that used to arrive at \eqref{first_2nd_order_form_term} shows
\begin{align}\label{second_2nd_order_form_term}
	\mathfrak{m}^{(2)}_{\la_0}(\E^s(\ell,\cdot),\E^u(\ell,0))(q) =  2\int_{\ell}^\infty 
	\widetilde{u}_2(x;0) \p_\la  \widetilde{v}_2(x;0) +  \widetilde{v}_2(x;0) \p_\la \widetilde{u}_2(x;0)\,dx.
\end{align}
By uniqueness of solutions we have $\mathbf{w}(\cdot;0) = \widetilde{\mathbf{w}}(\cdot;0)$. Furthermore, since $\dot r(0) = \dot s(0)$,
\begin{equation}\label{initial_cond}
	\p_\la \mathbf{w}(\ell;0) = \dot{r}(0) = \dot{s}(0) = \p_\la \widetilde{\mathbf{w}}(\ell;0).
\end{equation}
Now, both $\p_\la\mathbf{w}(\cdot;0)$ and $\p_\la\widetilde{\mathbf{w}}(\cdot;0)$ solve the inhomogeneous differential equation
\begin{equation}\label{eq_inhomog_la_zero}
	\de{}{x}\left (\p_\la\mathbf{w}\right ) = A\left (\p_\la\mathbf{w}\right ) +  A_\la \left (\pmb{\phi} \, k_1+\pmb{\varphi} \,k_2 \right ),
\end{equation}
obtained by differentiating \eqref{1st_order_sys_N_vector} with respect to $\la$ and evaluating at $\la=0$, and using that $\mathbf{w}(\cdot;0) = \pmb{\phi} \, k_1+\pmb{\varphi} \, k_2$. (Note that $k_1,k_2\in\R$ are determined by the fixed vector $w_0$, where $w_0=\mathbf{w}(\ell;0)=\pmb{\phi}(\ell)k_1+\pmb{\varphi}(\ell)k_2$.) It follows from \eqref{initial_cond} and uniqueness of solutions of \eqref{eq_inhomog_la_zero} that indeed $\p_\la\mathbf{w}(x;0)=\p_\la\widetilde{\mathbf{w}}(x;0)$ for \emph{all} $x\in\R$. Collecting \eqref{first_2nd_order_form_term} and \eqref{second_2nd_order_form_term} together, \eqref{relative_crossform_N} then becomes
\begin{equation}\label{eq:final_2nd_la_form}
	\mathfrak{m}^{(2)}_{\la_0}(\E^u(\ell,\cdot),\E^s(\ell,\cdot))(w_0) =  -2\int_{-\infty}^\infty {u}_2(x;0) \p_\la  {v}_2(x;0) +  {v}_2(x;0) \p_\la {u}_2(x;0)\,dx.
\end{equation}
We need to understand the function $\p_\la\mathbf{w}(\cdot;0)$. Notice that it solves the inhomogeneous equation \eqref{eq_inhomog_la_zero} if and only if its third and fourth entries $\p_\la u_2(\cdot;0)$ and $\p_\la v_2(\cdot;0)$ solve 
\begin{align}\label{eq:inhomog_N}
	N\begin{pmatrix}
		\p_\la u_2(\cdot;0) \\ -\p_\la v_2(\cdot;0)
	\end{pmatrix} = \begin{pmatrix}
		\phi_x\, k_1 \\  -\phi\, k_2
	\end{pmatrix}.
\end{align}
This follows from differentiating the eigenvalue equation \eqref{eq:N_evp} with respect to $\la$, evaluating at $\la=0$ and making the substitutions (as in \eqref{subs})
\[
\p_\la u(\cdot;0) = \p_\la u_2(\cdot;0), \quad \p_\la v(\cdot;0)=-\p_\la v_2(\cdot;0), \quad  u(\cdot;0)= \phi_x\, k_1,  \quad v(\cdot;0)=-\phi\, k_2.
\] 
Both equations in \eqref{eq:inhomog_N},
\begin{align}\label{inhomog_eqns}
	\begin{gathered}
		-L_-\p_\la v_2(\cdot;0) = -\phi_x\, k_1, \\
		L_+ \p_\la u_2(\cdot;0)  = -\phi \, k_2,
	\end{gathered}
\end{align}
are solvable by virtue of the Fredholm alternative, since $\langle\phi',\phi\rangle_{L^2(\R)}=0$ and hence $\phi_x\in\ker(L_-)^\perp$ and  $\phi\in\ker(L_+)^\perp$. Denoting by $\widehat{v}$ and $\widehat{u}$ any solutions to
\begin{equation}\label{L-_inhomog_eqn}
	-L_- v = \phi_x \qquad\text{and} \qquad L_+u = \phi
\end{equation}
in $H^4(\R)$ respectively (note the sign change in both equations from \eqref{inhomog_eqns}), \eqref{eq:final_2nd_la_form} becomes
\begin{align}\label{2nd_form_final}
	\mathfrak{m}^{(2)}_{\la_0}(\E^u(\ell,\cdot),\E^s(\ell,\cdot))(w_0) &=   
	2\left (\int_{-\infty}^\infty \phi_x \,\widehat{v} \,dx\,\right )k_1^2 - 
	2\left (\int_{-\infty}^\infty \phi \,\widehat{u} \,dx\,\right )k_2^2,
\end{align}
recalling that $u_2=\phi_xk_1$ and $v_2=-\phi \,k_2$.

Having computed the form, we count the number of negative squares. Using \eqref{second_order_initial}, and defining $\mathcal{I}_1$ and $\mathcal{I}_2$ as in \eqref{key_integrals}, we find that
\begin{equation}\label{mc}
	\Mas(\E^u(\ell,\cdot),\E^s(\ell,\cdot);[0,\e] ) = -n_- (\mathfrak{m}^{(2)}_{\la_0} ) = 
	\begin{cases*}
		0 & \,$	\mathcal{I}_1>0, \, \mathcal{I}_2<0$, \\
		-1 & \,$\mathcal{I}_1\mathcal{I}_2>0$, \\
		-2  & \,$\mathcal{I}_1<0, \,\mathcal{I}_2>0$.
	\end{cases*}
\end{equation}
This completes the calculation of the second term in \eqref{c_again}. Combining \eqref{mc} with \eqref{eq:mathfrak_b} yields \eqref{c_value}. 
\end{proof}

We are now ready to prove \cref{thm:main_lower_bound}.
\begin{proof}[Proof of \cref{thm:main_lower_bound}]
	By homotopy invariance and additivity under concatenation, we have
	\begin{multline*}\label{}
		\Mas(\E^u(\cdot,0),\E^s(\ell,0);[-\infty,\ell] ) +\Mas(\E^u(\ell,\cdot),\E^s(\ell,\cdot);[0,\la_\infty] )  \\  -\Mas(\E^u(\cdot,\la_\infty),\E^s(\ell,\la_\infty);[-\infty,\ell] ) -\Mas(\E^u(-\infty,\cdot),\E^s(\ell,\cdot);[0,\la_\infty] ) =0.
	\end{multline*}
	By \cref{lemma:Mas_zero_bottom_right} the last two terms on the left hand side vanish. Recalling the definition of $\mathfrak{c}$ from \eqref{mathfrakc} and using the concatenation property once more,
	\begin{equation}\label{eq:proof1}
		\Mas(\E^u(\cdot,0),\E^s(\ell,0);[-\infty,\ell-\e] ) + \mathfrak{c} +  \Mas(\E^u(\ell,\cdot),\E^s(\ell,\cdot);[\e,\la_\infty] )  =0.
	\end{equation}
	Since the Maslov index counts \emph{signed} crossings, the number of crossings along $\Gamma_2$ for $\la>0$ is bounded from below by the absolute value of the Maslov index of this piece, i.e.
	\begin{equation}\label{eq:proof2}
		n_+(N) \geq |\Mas(\E^u(\ell,\cdot),\E^s(\ell,\cdot);[\e,\la_\infty] ) |.
	\end{equation}
	Combining \eqref{eq:proof1} and \eqref{eq:proof2} with \cref{lemma_P-Q}, the inequality \eqref{lower_bound} follows. The statement of the theorem then follows from the computation of $\mathfrak{c}$ in \cref{lemma:value_of_c}.
\end{proof}
We conclude with the proof of \cref{thm:VK_criterion}, for which we will need the following lemma. The first assertion gives a sufficient condition for monotonicity of the Maslov index along $\Gamma_2$, and is adapted from \cite[Lemma 5.1]{CCLM23}. The second assertion is given in \cite[Lemma 5.2]{CCLM23}.
\begin{lemma}\label{lemma:monotone_cross}
	If $L_-$ is a nonpositive operator, then each crossing $\la=\la_0>0$ of the path $\la\mapsto\left ( \E^u(\ell,\la), \E^s(\ell,\la)\right )$ is positive.
	Moreover, in this case $\spec(N) \subset \R \cup i\R$. 
\end{lemma}
\begin{proof}
	If $\la=\la_0$ is a crossing then the eigenvalue equations
	\begin{equation}\label{EVP_chap5}
		-L_-v = \la_0 u, \qquad L_+u = \la_0 v
	\end{equation}
	are satisfied for some $\tilde{u},\tilde{v}\in H^4(\R)$. Notice that $\la_0>0$ necessitates that \emph{both} $\tilde{u}$ and $\tilde{v}$ are nontrivial. 
	
	Solving the first equation in \eqref{EVP_chap5} yields $\tilde{v} = \al \phi + \tilde{v}_\perp$ for some $\al\in\R$, where $\ker(L_-) = \spn\{\phi\}$ and  $\tilde{v}_\perp\in\ker(L_-)^\perp$. Therefore
	\begin{equation}\label{eq:Lneg}
		\langle L_- \tilde{v}, \tilde{v}\rangle_{L^2(\R)} = \langle L_- ( \al \phi + \tilde{v}_\perp), \al \phi + \tilde{v}_\perp\rangle_{L^2(\R)} =  \langle L_- \tilde{v}_\perp, \tilde{v}_\perp\rangle_{L^2(\R)} <0
	\end{equation}
	because $L_-$ is nonpositive and $\hat{v}_\perp\in\ker(L_-)^\perp$. Now analysing the crossing form \eqref{eq:final_1st_la_form} for the path $\la\mapsto\left ( \E^u(\ell,\la), \E^s(\ell,\la)\right )$, where $v_2=-\tilde{v}$ and $u_2=\tilde{u}$, we have
	\begin{align*}
		\mathfrak{m}_{\la_0}(\E^u(\ell,\cdot),\E^s(\ell,\cdot))(q)  = -\f{2}{\la_0}\int_{-\infty}^{\infty} \left (\la_0\, u_2\right ) v_2 \,dx = -\f{2}{\la_0}\langle L_- \tilde v, \tilde v\rangle_{L^2(\R)} > 0,
	\end{align*}
	which was to be proven. The second statement may be proven using similar arguments as in the proof of \cite[Lemma 5.1]{CCLM23}. Namely, we can rewrite  \eqref{eq:N_evp} as the selfadjoint eigenvalue problem 
	\begin{equation}\label{saevp}
		\left (-L_-|_{X_c} \right )^{1/2} \Pi L_+ \Pi \left (-L_-|_{X_c} \right)^{1/2} w =  \la^2 w,
	\end{equation}
	where $X_c = \ker(L_-)^\perp$, $\Pi$ is the orthogonal projection in $L^2(\R)$ onto $X_c$, $\left (-L_-|_{X_c} \right )^{1/2} $ is well-defined because $-L_-$ is nonnegative, and $w= \left (-L_-|_{X_c} \right )^{1/2} \Pi v$. It follows that $\la^2\in\R$. For more on the equivalence of  \eqref{eq:N_evp} with \eqref{saevp}, see \cite[Lemma 3.21]{CCLM23}. We omit the details here. 
\end{proof}
\begin{proof}[Proof of \cref{thm:VK_criterion}]
	If $Q=0$ then it follows from \cref{lemma:monotone_cross} that
	\begin{equation*}
		\Mas(\E^u(\ell,\cdot), \E^s(\ell,\cdot); [\e,\la_\infty]) = n_+(N)
	\end{equation*} 
	for $\e$ small enough. Using this and \cref{lemma_P-Q} in \eqref{eq:proof1}, we obtain 
	\begin{equation*}\label{}
		n_+(N) = P-Q-\mathfrak{c}  = 1 - \mathfrak{c}.
	\end{equation*}
	For the evaluation of $\mathfrak{c}$, using \eqref{L-_inhomog_eqn} we have
	\begin{equation*}\label{}
		\mathcal{I}_1 = \int_{-\infty}^\infty \phi_x \,\widehat{v} \,dx=-\int_{-\infty}^{\infty}  \left (L_- \widehat{v}\right ) \widehat{v} \,dx \geq 0,
	\end{equation*}
	because $Q=0$. An argument similar to \eqref{eq:Lneg} shows that in fact $\mathcal{I}_1>0$. \Cref{lemma:value_of_c} now yields the value of $\mathfrak{c}$. In particular, if $\mathcal{I}_2>0$, then $\mathfrak{c}=0$ and $n_+(N) = 1$. On the other hand, if $\mathcal{I}_2<0$, then $\mathfrak{c}=1$ and $n_+(N) = 0$, and in this case $\spec(N)\subset i\R$  by the second assertion of \cref{lemma:monotone_cross}.
\end{proof}

\section{Application: spectral stability of the Karlsson and H\"o\"ok solution}\label{sec:KH_stability}

In this section we apply our theory to confirm the spectral stability of the Karlsson and H\"o\"ok solution \eqref{KH_soln}, i.e.
\begin{equation*}
	\phi_{\text{KH}}(x) = \sqrt{\f{3}{10}} \sech^2\left (\f{x}{2\sqrt{5}}\right ),
\end{equation*}
which solves \eqref{SWE} for $\beta=4/25, \sigma_2=-1$. In this case $\esspec(L_\pm) =  (-\infty,-4/25 )$. While orbital (and hence spectral) stability was proven in \cite{NataliPastor15}, the following merely serves to showcase how our analytical results may be implemented for a given soliton solution. 

In \cref{fig:curves}, we have plotted the set of points in the $\la x$-plane where the unstable bundle nontrivially intersects the stable subspace of the asymptotic system for each of the eigenvalue problems for $L_+$ and $L_-$. More precisely, for the $L_+$ problem we have the following construction (the construction for the $L_-$ problem is similar). We denote the eigenvalues of the $\la$-dependent asymptotic matrix $A_+(\la)$ by $\{\pm\mu_1(\la),\pm\mu_2(\la)\}$ and the corresponding stable and unstable eigenvectors by $\mathbf{p}_{1,2}(\la)$ and $\mathbf{q}_{1,2}(\la)$ respectively, so that for $i=1,2$,
\begin{align*}
	\ker\left (  A_+(\la) + \mu_i(\la) \right ) = \spn \{\mathbf{p}_i(\la)\}, \qquad \ker\left (  A_+(\la) - \mu_i(\la) \right ) = \spn \{\mathbf{q}_i(\la)\}
\end{align*}
(similar to \eqref{spatialevalsL+}--\eqref{kernels}). Collecting the vectors $\mathbf{p}_{1,2}$ in the columns of a $4\times 2$ frame and right multiplying by the inverse of the resulting upper $2\times 2$ block yields the following real frame for $\Ss_+(\la)$,
\begin{equation}
	\begin{pmatrix}
		I \\ S_+(\la)
	\end{pmatrix}, \qquad S_+(\la)\in\R^{2\times 2}.
\end{equation}
Next, we require a frame for the unstable bundle, suitably initialised. Making similar manipulations as above on the scaled unstable eigenvectors $e^{-\mu_1(\la)\ell}\mathbf{q}_1(\la), e^{-\mu_2(\la)\ell}\mathbf{q}_2(\la)$ yields a real frame for the unstable subspace, which we denote by
\begin{equation}\label{ICunstable}
	\begin{pmatrix}
		I \\ U_+(\la)
	\end{pmatrix}, \qquad U_+(\la)\in\R^{2\times 2}.
\end{equation}
A frame for the unstable bundle is then given by solutions to \eqref{1st_order_sys_L+} initialised at \eqref{ICunstable}. We denote this frame by $ \mathbf{U}_+(x,\la) =
(X_+(x,\la), Y_+(x,\la))$. The \emph{eigenvalue curves} are then given by the locus of points $(\la,x)$ such that  
\begin{equation*}
	\det \begin{pmatrix}
		X_+(x,\la) & I \\ Y_+(x,\la) & S_+(\la)
	\end{pmatrix} = \det (S_+(\la)X_+(x,\la)-Y_+(x,\la) ) = 0,
\end{equation*}
see \cref{fig:curves}. (The name follows from the fact that each point $(\la,x)$ on such curves represents an eigenvalue $\la$ of the operator $L_+$ with domain $\dom(L_+) = \{ u\in H^4(-\infty, s): (u''(s)+\sigma_2 u(s), u(s), u'(s), u'''(s)) \in \Ss_+(\la)  \}$.) The intersections of the eigenvalue curves with $\Gamma_1$ (where $\la=0$) thus represent conjugate points, while the crossings along $\Gamma_2$ (where $x=\ell$) represent the eigenvalues of $L_+$. 
\begin{rem}
	Strictly speaking, as per our analysis we should be using $\E_\pm^s(x,\ell)$ for some large $\ell$ as our reference plane. However, for the purposes of graphical illustration, it suffices to use $\Ss_\pm(\la)$, the train of which is an arbitrarily small perturbation of the train of $\E_\pm^s(x,\ell)$ (for large enough $\ell$).
\end{rem}
\begin{figure}[h]
	\hspace*{\fill}
	\subcaptionbox{\label{}$L_+$ eigenvalue curves.} 
	{\includegraphics[width=0.42\textwidth]{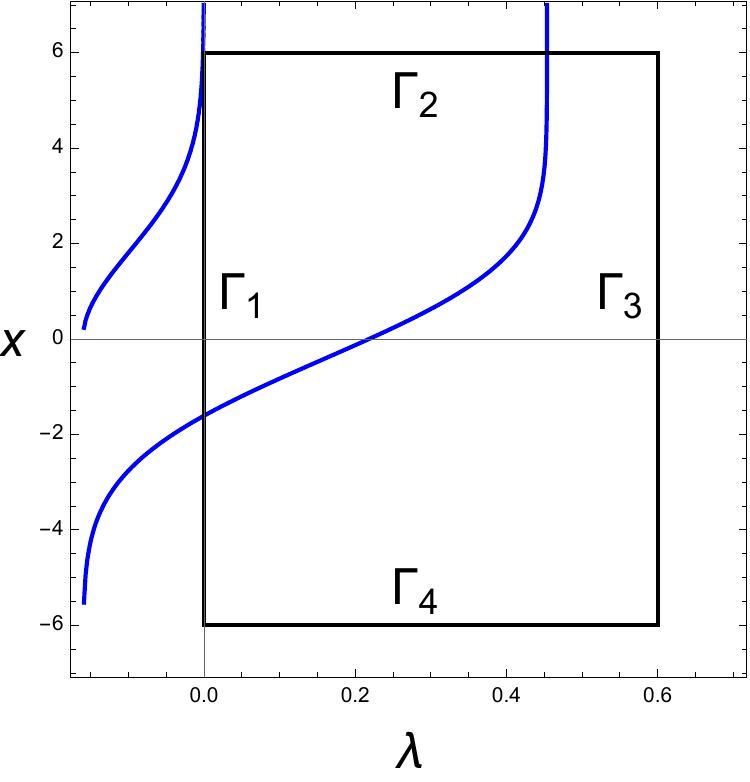}}\hfill 
	\subcaptionbox{\label{L-curves}$L_-$ eigenvalue curves.}
	{\includegraphics[width=0.42\textwidth]{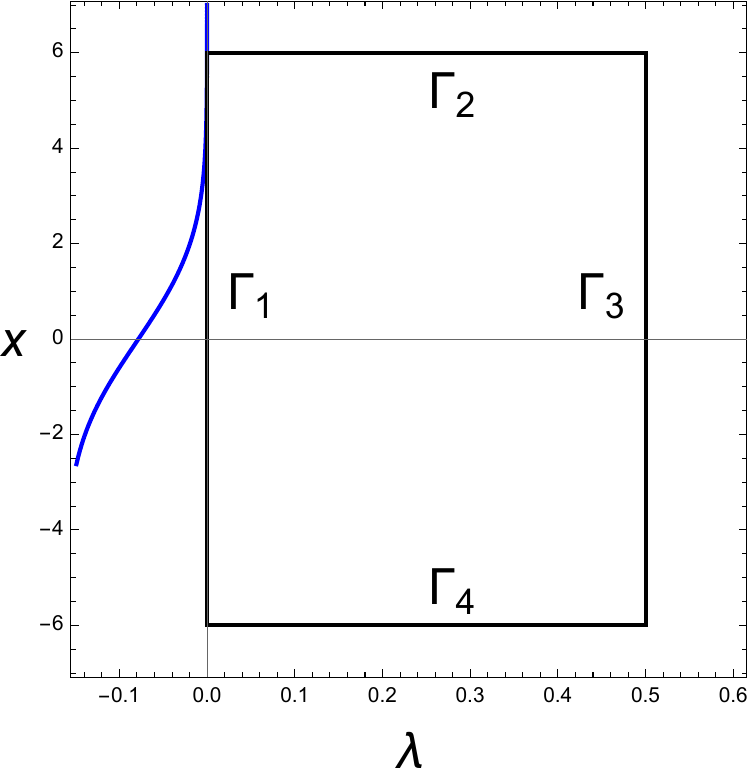}}
	\hspace*{\fill}
	\caption{$L_+$ and $L_-$ eigenvalue curves and Maslov box for the Karlsson and H\"o\"ok solution $\phi_{\text{KH}}$, where $\be=4/25$, $\sigma_2=-1$ and $\ell=6$. In both cases, an eigenvalue curve is asymptotic to the line $\la=0$ (but never crosses).} 
	\label{fig:curves}
\end{figure}

From \cref{fig:curves}, we conclude that there is one conjugate point for the $L_+$ problem, while there are none for the $L_-$ problem (the eigenvalue curve in each subfigure which is asymptotic to $\la=0$ never crosses $\la=0$). \Cref{thm:Lpm_eval_counts} correctly predicts the number of crossings along $\Gamma_2$ in both cases, i.e. $P=1, Q=0$. By \cref{thm:VK_criterion}, spectral stability of the Karlsson and H\"o\"ok solution is thus determined by the sign of $\mathcal{I}_2$. Using the theory developed in \cite{Albert92}, it was shown in \cite[Remark 4.6]{NataliPastor15} that $\mathcal{I}_2<0$. Hence $\phi_{\text{KH}}$ is spectrally stable.

\section{Concluding remarks}\label{sec:concluding_rems}

\Cref{thm:main_lower_bound} remains true for pure quartic solitons, i.e. soliton solutions to \eqref{4NLS_original} with $\beta_2=0$. In this case, non-dimensionalising \eqref{4NLS_original} with the transformations
\[ 
\psi = \sqrt{ \f{24\gamma}{|\beta_4|}} \Psi, \qquad z = \f{24}{|\beta_4|}x,
\]
(where $\gamma>0$, $\be_4<0$) leads to \eqref{4NLS} with $\sigma_2=0$ (after interchanging $z$ back to $x$). { Our proof of \cref{thm:main_lower_bound} can thus be applied to PQSs with the following adjustments. The substitutions \eqref{subs} with $\sigma_2=0$ reduce \eqref{uv_system} with $\sigma_2=0$ to the first order system \eqref{1st_order_sys} with $\sigma_2=0$ and $\al(x) = 3\phi(x)^2 -\beta$, $\eta(x) = -\phi(x)^2+\beta$. In lieu of \eqref{assumptions} and \eqref{assumption2}, one simply requires that $\be>0$ to ensure that the origin in \eqref{1st_order_SWE_Ham} is hyperbolic, the essential spectra of $L_\pm$ are confined to the negative half line, and the spatial eigenvalues of $L_\pm$ are distinct at $\la=0$. The stable subspace $\Ss_\pm(0)$ now has the frame
\begin{align}\label{S+frames2}
		\mathbf{S}_\pm &= \begin{pmatrix}
			I \\ S_\pm
		\end{pmatrix}, \qquad S_\pm= \frac{1}{\sqrt{2 \sqrt{\beta } }} \left(
		\begin{array}{cccc}
			\mp 1   &  -\sqrt{\beta }\\
			 -\sqrt{\beta } & \pm \beta
		\end{array}
		\right),
	\end{align}
	and computing crossing forms in $x$ yield expressions identical to those previously computed, but with $\sigma_2=0$. For example, for the $L_+$ problem, the first order form \eqref{FOF} is unchanged, while the third order forms described in Cases 1 and 2 in the proof of \cref{lemma:L+_Mas_left_conjpoints} are given by
	\[
		\mathfrak{m}_{x_0}^{(3)} = - \frac{3a_0^2 \phi(x_0)^2}{\sqrt{\beta}} \qquad \text{and} \qquad \mathfrak{m}_{x_0}^{(3)}= - \frac{3a_0^2 \phi(x_0)^2}{4\sqrt{\beta}}
	\]
	respectively. Monotonicity of the paths $x\mapsto(\E_\pm^u(x,0),\Ss_\pm(0))$ is therefore preserved. The proof of \cref{thm:main_lower_bound} for PQSs is then as given in \cref{sec:proof_lower_bound}. 
}	

In this work, similar to the analyses in \cite{Corn19,HLS18,Howard21}, the central objects are Lagrangian pairs comprising the unstable and stable bundles. For the fourth order selfadjoint eigenvalue problems discussed in \cref{sec:L+L__counts}, we showed that monotonicity holds for the pairs
\begin{equation}\label{paths7}
	\Gamma \ni (x,\la) \mapsto ( \E_\pm^u(x,\la), \E_\pm^s(\ell,\la) )\in \cL(2)\times \cL(2)
\end{equation}
 in both $x$ and $\la$. Another possible choice for the second entry of the pair in \eqref{paths7}, as in \cite{BJP24,Howard21}, is to use the fixed \emph{sandwich} plane with frame
\[
\begin{pmatrix}
	1 & 0 \\ 0 & 0 \\ 0 & 0 \\ 0 & 1
\end{pmatrix},
\]
with respect to which the unstable bundles $\E_\pm^u(x,\la)$ will be monotonic in both $x$ and $\la$. It should then be possible to argue as in \cite{BCLJMS18,BJP24}, by first considering the problem on the half line  $(-\infty, \ell]$ for large $\ell$, and then using asymptotic arguments similar to those in \cite{SandstedeScheel2000} to recover a Morse--Maslov theorem on $\R$. However, for the eigenvalue problem for $N$, it is unclear how to recover the contribution $\mathfrak{c}$ using the sandwich reference plane (of $\R^{8}$), which obscures the existence of a crossing at the top left corner of the Maslov box.

Generally speaking, a rigorous proof of spectral stability or instability using the Maslov index is a two-step process. The first involves proving that the existence of unstable eigenvalues for an arbitrary stationary state can be determined by counting conjugate points. The second involves proving the existence (or non-existence) of conjugate points given a particular stationary state. Regarding the latter, most often this has been done numerically \cite{CDB09,Beck_Jaquette22,BJP24}, and in some cases analytically \cite{CJ18,BCLJMS18} (see, for example, the beautiful argument in \cite{BCLJMS18}, where a clever choice of reference plane allows the authors to exploit the reversibility symmetry of the underlying PDE to force a conjugate point for pulses in reaction-diffusion systems). For the current problem, it would be interesting to analytically determine exact counts of the $L_+$ and $L_-$ conjugate points given a particular soliton solution; one difficulty in doing so is the inability to determine exact formulas for both basis vectors of the unstable bundle $\E_\pm^u(x,0)$. More feasibly, our results facilitate the numerical verification of (in)stability given a particular soliton in the following cases: (1) using \cref{cor:JonesGrillakis} and \cref{thm:Lpm_eval_counts} to prove instability by showing $|p_c-q_c| \geq 2$, and (2) when either $P=p_c=0$ or $Q=q_c=0$. As pointed out in \cref{sec:proof_lower_bound}, in the latter case our lower bound becomes an exact count, $n_+(N) = |P - Q - \mathfrak{c}| = |p_c - q_c - \mathfrak{c}|$, and spectral stability or instability follows by showing that $|p_c-q_c-\mathfrak{c}|=0$ or $|p_c-q_c-\mathfrak{c}|\geq 1$, respectively.

Finally, we comment on the case when either $\mathcal{I}_1= 0$ or $\mathcal{I}_2=0$. In this case, the second order form \eqref{2nd_form_final} is degenerate, and one proceeds by computing the third order crossing form. The Fredholm Alternative dictates that the algebraic multiplicity of $0\in\spec(N)$ increases by { at least} one for each quantity $\mathcal{I}_1, \mathcal{I}_2$ that vanishes. Similar to \eqref{2nd_form_final}, any third order form will be given by the $L_2(\R)$ inner products, akin to $\mathcal{I}_1$ and $\mathcal{I}_2$, of functions in $\ker(N)$ and $\ker(N^3)\backslash \ker(N^2)$. Due to the Hamiltonian structure of $N$, which implies symmetry in $\spec(N)$ about the real and imaginary axes, we expect that all higher order crossing forms of odd order will be identically zero. The contribution to the Maslov index is then given by the sum of the negative indices of any nondegenerate even order forms.

\begin{appendices}
	\section{Appendix: Removal of \cref{hypo:nonzero}.}\label{sec:appendixA}
	
	In order to complete the proof of \cref{lemma:L+_Mas_left_conjpoints}, we need to account for the case when \cref{hypo:nonzero} fails, i.e. $\phi(x_0)=0$. In this case all of the crossing forms computed in the proof of \cref{lemma:L+_Mas_left_conjpoints} are identically zero, and we need to compute higher order crossing forms. The key observation here is that since $\phi$ solves the standing wave equation \eqref{SWE}, a fourth order ODE, for every $x\in\R$ at least one of the elements of the set $\{\phi(x),\phi'(x), \phi''(x), \phi'''(x)\}$ is nonzero. As we will see, it will follow that computing sufficiently many higher order crossing forms will yield enough nonzero summands in the left hand side of \eqref{dimension_add_up}, i.e. such that
	\begin{equation}
		\sum_{k\geq 1} \left ( n_+(\mathfrak{m}_{x_0}^{(k)})  + n_-(\mathfrak{m}_{x_0}^{(k)}) \right ) = \dim \E_\pm^u(x_0,0)\cap \Ss_\pm(0).
	\end{equation}
	Similar to the proof of \cref{lemma:L+_Mas_left_conjpoints}, we separate the analyses depending on the nature of the intersection $\E_\pm^u(x_0,0)\cap \Ss_\pm(0)$, and focus on the $L_+$ problem only; the proof for the $L_-$ problem is similar. We further split the analysis of each of these cases into three subcases based on the lowest nonzero derivative of $\phi$ at $x_0$. Since the calculations are tedious and similar to those in the proof of \cref{lemma:L+_Mas_left_conjpoints}, in the interest of expediency we present only the main results. 
	
	The following facts are needed for the results listed thereafter. If $\phi(x_0)=0$ and $\phi'(x_0)\neq0$, then it follows from \eqref{CminusSBS} and \eqref{widetildeC} that
	\begin{equation}\label{C=SBS}
		C_+ = S_+B_+S_+, \qquad C'_+=0, \qquad C''(x_0) = \begin{pmatrix}
			0 & 0 \\ 0 & 6 \phi'(x_0)^2
		\end{pmatrix}.
	\end{equation}
	If $\phi(x_0)=\phi'(x_0)=0$ and $\phi''(x_0)$, then 
	\begin{equation}\label{caseii}
		C_+ = S_+B_+S_+, \qquad C_+''' = C_+''=C_+'=0, \qquad C_+^{(4)}(x_0) = \begin{pmatrix}
			0 & 0 \\ 0 & 18 \phi''(x_0)^2
		\end{pmatrix}.
	\end{equation}
	Finally, if  $\phi^{(k)}(x_0)=0$ for $0\leq k \leq 2$ and $\phi'''(x_0)\neq0$, then $C_+ = S_+B_+S_+$ and 
	\begin{equation}\label{caseiii}
		C_+^{(5)} = C_+^{(4)} = C_+''' = C_+''=C_+'=0, \qquad C_+^{(6)}(x_0) = \begin{pmatrix}
			0 & 0 \\ 0 & 60 \phi'''(x_0)^2
		\end{pmatrix}.
	\end{equation}
	Expressions for derivatives of $X(x)$ and $Y(x)$ at $x_0$ up to order 9 are required, and can be simplified using \eqref{C=SBS} -- \eqref{caseiii} where appropriate. The vectors $h_i$ are found using \eqref{k_equations}; in particular, assuming the set $\{h_0,\dots, h_{k-2}\}$ satisfies \eqref{k_equations}, then  $\{h_0,\dots, h_{k-2}, h_{k-1}\}$ solves
	\begin{align}
		(Y_0-S_+X_0) h_{k-1} = -\sum_{i=0}^{k-2} {k-1 \choose i}\left ( Y^{(k-1-i)}(x_0) - S_+ X^{(k-1-i)}(x_0)\right )h_{i}.
	\end{align}
	In this case the forms $\mathfrak{m}^{(k)}_{x_0}$ are computed via \eqref{kth_crossing_form_practice}, which amounts to
	\begin{equation}
		\mathfrak{m}^{(k)}_{x_0}(\E^u_+(\cdot,0),\Ss_+(0))(w_0) =   -\sum_{i=0}^{k-1} {k \choose i} \left \langle  \left ( Y^{(k-i)}(x_0) - S_+ X^{(k-i)}(x_0)\right )h_{i}, k_0 \right \rangle.
	\end{equation}
	
	\emph{Case 1: $\dim \E^u_+(x_0,0) \cap \Ss_+(0)=1, b_0\neq 0$.} In this case \eqref{W1_not_spns1} holds, i.e. $W_1 = \E^u_+(x_0,0)\cap \Ss_+(0) = \spn \{a_0\mathbf{s}_1 + b_0 \mathbf{s}_2\}$ for some fixed $a_0$ and $b_0\neq 0$.
	
	\emph{Case (i): $\phi(x_0)=0, \phi'(x_0)\neq0$.} An identical calculation to that in the proof of \eqref{lemma:L+_Mas_left_conjpoints} shows that $h_1\in\ker (Y_0-S_+X_0)$ and $\mathfrak{m}^{(2)}_{x_0}=0$. Using \eqref{C=SBS}, it can be shown that $h_2\in\ker(Y_0-S_+X_0)$ and
	\begin{equation}\label{m3phi0case1i}
		\mathfrak{m}^{(3)}_{x_0}(\E^u_+(\cdot,0),\Ss_+(0))(w_0) = - \left \langle C''_+ k_0,k_0\right \rangle = -6[ \phi'(x_0)]^2 b_0^2 < 0. 
	\end{equation}
	\emph{Case (ii): $\phi(x_0)=0, \phi'(x_0)=0, \phi''(x_0)\neq0$.} We now have $\mathfrak{m}^{(k)}_{x_0}=0$ for $k=1,2,3$. Using that $C_+'''(x_0) = C_+''(x_0)=0$, it can be shown that $h_3\in\ker(Y_0-S_+X_0)$. With $h_k\in\ker(Y_0-S_+X_0)$ for $k=1,2,3$, one finds $\mathfrak{m}^{(4)}_{x_0}=0$. For the fifth order form, one additionally has $h_4\in\ker(Y_0-S_+X_0)$ and 
	\begin{equation}\label{m5_case2ii}
		\mathfrak{m}^{(5)}_{x_0}(\E^u_+(\cdot,0),\Ss_+(0))(w_0) = - \left \langle C^{(4)}_+ k_0,k_0\right \rangle = -18 \left [\phi''(x_0)\right ]^2 b_0^2< 0. 
	\end{equation}
	\emph{Case (iii): $\phi(x_0)=0, \phi'(x_0)=\phi''(x_0)=0, \phi'''(x_0)\neq0$.} Now $\mathfrak{m}^{(k)}_{x_0}=0$ for $k=1,2,3,4,5$. Using $C_+^{(5)}(x_0) = C_+^{(4)}(x_0)=0$, it can be shown that $h_5\in\ker(Y_0-S_+X_0)$ and $\mathfrak{m}^{(6)}_{x_0}=0$. Moreover, one finds $h_6\in\ker(Y_0-S_+X_0)$ and 
	\begin{equation}\label{m7_case2iii}
		\mathfrak{m}^{(7)}_{x_0}(\E^u_+(\cdot,0),\Ss_+(0))(w_0) = - \left \langle C^{(6)}_+ k_0,k_0\right \rangle = -60 \left [\phi'''(x_0)\right ]^2 b_0^2< 0. 
	\end{equation}
	
	\emph{Case 2: $\dim \E^u_+(x_0,0) \cap \Ss_+(0)=1, b_0=0$.} In this case \eqref{assumption_W1} holds, i.e. $W_1 = \E^u_+(x_0,0)\cap \Ss_+(0) = \spn \{\mathbf{s}_1\}$, and for any $i\geq 1$, we have $C_+^{(i)}k_0=0$.
	
	\emph{Case (i): $\phi(x_0)=0, \phi'(x_0)\neq0$.} As in the proof of \cref{lemma:L+_Mas_left_conjpoints}, we have $\mathfrak{m}^{(k)}_{x_0}=0$ for $k=1,2,3$. Using \eqref{C=SBS} and $h_k\in\ker(Y_0-S_+X_0)$ for $k=1,2,$ one finds that $h_3 \in\ker(Y_0-S_+X_0)$ and $\mathfrak{m}^{(4)}_{x_0}=0$. With these $h_k$ one then finds that $h_4$ satisfies 
	\begin{equation}
		(Y_0 - S_+X_0) h_4 = -3C_+''B_+Y_0h_0 = -3 C_+''B_+S_+k_0,
	\end{equation}
	and
	\begin{equation*}
		\mathfrak{m}^{(5)}_{x_0}(\E^u_+(\cdot,0),\Ss_+(0))(w_0) = -3\cdot 4 \left \langle S_+B_+C_+''B_+S_+k_0,k_0 \right \rangle = -3\cdot 4  \left ( \f{6\phi'(x_0)^2a_0^2}{2\sqrt{\beta} - \sigma_2} \right ) <0.
	\end{equation*}
	\emph{Case (ii): $\phi(x_0) = \phi'(x_0)=0, \phi''(x_0)\neq0$.} Now $C_+'''=C_+''=C_+'=0$, and $\mathfrak{m}^{(k)}_{x_0}=0$ for $k=1,2,3,4,5$. With $h_k\in\ker(Y_0-S_+X_0)$ for $k=1,2,3,4$, one finds that $h_5 \in\ker(Y_0-S_+X_0)$ and $\mathfrak{m}^{(6)}_{x_0}=0$. With these $h_k$  one finds that $h_6$ satisfies
	\begin{equation}
		(Y_0 - S_+X_0) h_6 = -5 C_+^{(4)}B_+S_+k_0,
	\end{equation}
	and
	\begin{equation*}
		\mathfrak{m}^{(7)}_{x_0}(\E^u_+(\cdot,0),\Ss_+(0))(w_0) = -5\cdot 6 \left \langle S_+B_+C_+^{(4)}B_+S_+k_0,k_0 \right \rangle = -5\cdot 6 \left ( \f{18\phi''(x_0)^2a_0^2}{2\sqrt{\beta} - \sigma_2} \right ) <0.
	\end{equation*}
	\emph{Case (iii): $\phi(x_0) = \phi'(x_0)=\phi''(x_0)=0, \phi'''(x_0)\neq0$.} Now $C_+^{(i)}=0$ for $1\leq i\leq 5$, and $\mathfrak{m}^{(k)}_{x_0}=0$ for all $1\leq k \leq 7$. With $h_k\in\ker(Y_0-S_+X_0)$ for $1\leq k \leq 6$, one finds that $h_7 \in\ker(Y_0-S_+X_0)$ and $\mathfrak{m}^{(8)}_{x_0}=0$. With these $h_k$ one finds that $h_8$ satisfies
	\begin{equation}
		(Y_0 - S_+X_0) h_8 = -7 C_+^{(6)}B_+S_+k_0,
	\end{equation}
	and
	\begin{equation*}
		\mathfrak{m}^{(9)}_{x_0}(\E^u_+(\cdot,0),\Ss_+(0))(w_0) = - 7\cdot 8\left \langle S_+B_+C_+^{(6)}B_+S_+k_0,k_0 \right \rangle = -7\cdot8 \left ( \f{60\phi'''(x_0)^2a_0^2}{2\sqrt{\beta} - \sigma_2} \right ) <0.
	\end{equation*}
	\emph{Case 3: $\dim \E^u_+(x_0,0) \cap \Ss_+(0)=2$.} Now we have $W_1 = \E^u_+(x_0,0) = \Ss_+(0)$, and $Y_0=S_+X_0$. At the outset, there is no restriction on the vector $k_0=(a,b)^\top$.
	
	\emph{Case (i): $\phi(x_0)=0, \phi'(x_0)\neq0$.} As in the proof of \cref{lemma:L+_Mas_left_conjpoints}, we have $C_+''\neq0$ and $\mathfrak{m}^{(k)}_{x_0}=0$ for $k=1,2$. It can be shown that now $h_1,h_2\in\ker (Y_0-S_+X_0)$ are arbitrary, and the third order form is given by \eqref{m3phi0case1i}. so that $n_-(\mathfrak{m}_{x_0}^{(3)})=1$. Next, using $W_4=\ker\mathfrak{m}_{x_0}^{(3)}=\spn\{\mathbf{s}_1\}$, one finds that $h_3$ is free provided $X_0h_0=k_0\in \ker C''_+$, i.e. $k_0=(a_0,0)^\top$, in which case $\mathfrak{m}^{(4)}_{x_0}=0$. For the fifth order form, with $k_0=(a,0)^\top$, one finds that $h_4$ is free provided $h_1$ satisfies
	\begin{equation}
		C''_+X_0 h_1 = -\f{3}{4} C_+''B_+S_+k_0,
	\end{equation}
	in which case
	\begin{equation}
		\mathfrak{m}^{(5)}_{x_0}(\E^u_+(\cdot,0),\Ss_+(0))(w_0) = -\f{3}{4} \left \langle S_+B_+C_+''B_+S_+k_0,k_0 \right \rangle = -\f{3}{4}  \left ( \f{6\phi'(x_0)^2a_0^2}{2\sqrt{\beta} - \sigma_2} \right ) <0,
	\end{equation}
	so that $n_-(\mathfrak{m}^{(3)}_{x_0})+n_-(\mathfrak{m}^{(5)}_{x_0})=2$.
	
	\emph{Case (ii): $\phi(x_0) = \phi'(x_0)=0, \phi''(x_0)\neq0$.} Now $C_+'''=C_+''=C_+'=0$, and $\mathfrak{m}^{(k)}_{x_0}=0$ for $1\leq k \leq 4$. With $W_k=W_1=\E^u_+(x_0,0) = \Ss_+(0)$ for $1\leq k \leq 5$, we find that $\mathfrak{m}^{(5)}_{x_0}$ is given by \eqref{m5_case2ii}, thus $n_-(\mathfrak{m}_{x_0}^{(5)})=1$. Next, using $W_6=\ker\mathfrak{m}_{x_0}^{(5)}=\spn\{\mathbf{s}_1\}$, one finds that $h_5$ is free provided  $X_0h_0=k_0\in \ker C^{(4)}_+$, i.e. $k_0=(a_0,0)^\top$, in which case $\mathfrak{m}^{(6)}_{x_0}=0$. For the seventh order form, with $k_0=(a,0)^\top$, one finds that $h_6$ is free provided $h_1$ satisfies
	\begin{equation}
		C^{(4)}_+X_0 h_1 = -\f{5}{6} C_+^{(4)}B_+S_+k_0,
	\end{equation}
	in which case
	\begin{equation}
		\mathfrak{m}^{(7)}_{x_0}(\E^u_+(\cdot,0),\Ss_+(0))(w_0) = -\f{5}{6} \left \langle S_+B_+C_+^{(4)}B_+S_+k_0,k_0 \right \rangle = -\f{5}{6}  \left ( \f{18\phi''(x_0)^2a_0^2}{2\sqrt{\beta} - \sigma_2} \right ) <0,
	\end{equation}
	so that $n_-(\mathfrak{m}^{(5)}_{x_0})+n_-(\mathfrak{m}^{(7)}_{x_0})=2$.
	
	\emph{Case (iii): $\phi(x_0) = \phi'(x_0)=\phi''(x_0)=0, \phi'''(x_0)\neq0$.} Now $C_+^{(i)}=0$ for $1\leq i \leq 5$, and $\mathfrak{m}^{(k)}_{x_0}=0$ for $1\leq k \leq 6$. With $W_k=W_1=\E^u_+(x_0,0) = \Ss_+(0)$ for $1\leq k \leq 7$, we find that $\mathfrak{m}^{(7)}_{x_0}$ is given by \eqref{m7_case2iii}, thus $n_-(\mathfrak{m}_{x_0}^{(7)})=1$. Next, using $W_8=\ker\mathfrak{m}_{x_0}^{(7)}=\spn\{\mathbf{s}_1\}$, one finds that $h_7$ is free provided  $X_0h_0=k_0\in \ker C^{(6)}_+$, i.e. $k_0=(a_0,0)^\top$, in which case $\mathfrak{m}^{(8)}_{x_0}=0$. For the ninth order form, with $k_0=(a,0)^\top$, one finds that $h_8$ is free provided $h_1$ satisfies
	\begin{equation}
		C^{(6)}_+X_0 h_1 = -\f{7}{8} C_+^{(6)}B_+S_+k_0,
	\end{equation}
	in which case
	\begin{equation}
		\mathfrak{m}^{(9)}_{x_0}(\E^u_+(\cdot,0),\Ss_+(0))(w_0) = -\f{7}{8} \left \langle S_+B_+C_+^{(6)}B_+S_+k_0,k_0 \right \rangle = -\f{7}{8}  \left ( \f{60\phi'''(x_0)^2a_0^2}{2\sqrt{\beta} - \sigma_2} \right ) <0,
	\end{equation}
	so that $n_-(\mathfrak{m}^{(7)}_{x_0})+n_-(\mathfrak{m}^{(9)}_{x_0})=2$.

\end{appendices}

\subsection*{Acknowledgements} M.C. acknowledges the support of the NSF grant DMS-2418900. R.M. and M.C. acknowledge the support of the Australian Research Council grant \\DP210101102.

\bibliographystyle{abbrv}
\bibliography{mybib}

\end{document}